\documentclass{article}

\usepackage{amsmath,amsthm,amssymb,mathtools}  
\usepackage{dsfont}			
\usepackage{thmtools}		
\usepackage{enumerate}
\usepackage{graphicx}								
\usepackage{subcaption}			

\usepackage[T1]{fontenc}	

\usepackage[hyperindex,breaklinks]{hyperref}		
\hypersetup{colorlinks=true, linkcolor=blue, citecolor=blue}

\usepackage[nameinlink]{cleveref} 	

\usepackage{bm}					
\usepackage{cmll}				

\usepackage[latin1]{inputenc}                       
\usepackage[disable,
			colorinlistoftodos,textsize=scriptsize]{todonotes}		
\newcommand{\SELF}[1]{\todo[color=green!30]{#1}} 

\newtheorem{theorem}{Theorem}[section]
\newtheorem{proposition}[theorem]{Proposition}
\newtheorem{corollary}[theorem]{Corollary}
\newtheorem{lemma}[theorem]{Lemma}

\theoremstyle{definition}
\newtheorem*{definition}{Definition} 

\newtheorem{example}[theorem]{Example}

\declaretheoremstyle[
spaceabove=6pt, spacebelow=6pt,
headfont=\normalfont\bfseries,
notefont=\normalfont\bfseries, 
notebraces={}{},
bodyfont=\normalfont\itshape
]{Estilo1}

\newcommand\inner[1] 		{\langle #1 \rangle}

\def\Proj			{P} 

\def\N{\mathds{N}}

\def\R{\mathds{R}}
\def\C{\mathds{C}}

\def\P{P} 
\def\PP{\mathds{P}}
\def\I{\mathcal{I}}

\def\im	 				{\mathrm{i}}

\def\Plucker		{Pl\"{u}cker}

\def\pperp		{\simperp}

\DeclareMathOperator{\Span}{span}

\DeclareMathOperator{\Ann}{Ann}

\def\lcontr		{\lrcorner\,}				

\def\ganu			{\nu}		
\def\gaomega	{\omega}		

\def\ii			{\mathbf{i}}		
\def\jj			{\mathbf{j}}		
\def\kk			{\mathbf{k}}		

\begin{document}

\title{Grassmann angles between real or complex subspaces}

\author{Andr\'e L. G. Mandolesi 
               \thanks{Instituto de Matemática e Estatística, Universidade Federal da Bahia, Av. Adhemar de Barros s/n, 40170-110, Salvador - BA, Brazil. E-mail: \texttt{andre.mandolesi@ufba.br}}}
               
\date{\today \SELF{v6 - com Identities, mas menos propriedades} }

\maketitle

\abstract{
The Grassmann angle improves upon similar angles between subspaces that measure volume contraction in orthogonal projections.
It works in real or complex spaces, with important differences, and is asymmetric, what makes it more efficient when dimensions are distinct. 
It can be seen as an angle in Grassmann algebra, being related to its products and those of Clifford algebra, and gives the Fubini-Study metric on Grassmannians, an asymmetric metric on the full Grassmannian, and Hausdorff distances between full sub-Grassmannians.
We give formulas for computing it in arbitrary bases, and identities for angles with certain families of subspaces, some of which are linked to real and complex Pythagorean theorems for volumes and quantum probabilities.
Unusual features of the angle with an orthogonal complement, or the angle in complex spaces, are examined.
	
\vspace{.5em}

\vspace{3pt}

}

\section{Introduction}

Measuring the separation between subspaces is important in many areas, from geometry and linear algebra to functional analysis \cite{Knyazev2010}, statistics \cite{Hotelling1936}, data mining \cite{Jiao2018}, etc.
In high dimensions no single number fully describes this separation, and different concepts are used, each suitable for certain purposes: Jordan's principal or canonical angles \cite{Jordan1875}, Dixmier's minimal angle \cite{Dixmier1949}, Friedrichs angle \cite{Friedrichs1937}, maximal angle \cite{Krein1948}, gap \cite{Kato1995}, and others. Some important references on the subject are \cite{Afriat1957,Deutsch1995,Miao1992,Miao1996,Wedin1983}, and recent works can be found in \cite{Baksalary2009,Galantai2008,Galantai2006,Gunawan2005,Hitzer2010a,Knyazev2010}.

The literature on angles between subspaces can be confusing to the uninitiated. Many authors call their favorite concept \emph{the angle}, as if the term had a clear and unique meaning, and present it without mention of its peculiarities, limitations, or alternatives.
Equivalent concepts are given different names, or presented in ways which seem unrelated.
All this can lead even experienced researchers to error (for an example, see the discussion of Dorst et al. \cite{Dorst2007} in \cref{sc:Related angles}). 

Certain \emph{volume projection angles} describe how volumes contract in orthogonal projections between subspaces \cite{Gluck1967,Gunawan2005}. They keep reappearing  under various names and more or less equivalent definitions, in terms of principal angles \cite{Galantai2006,Hitzer2010a}, determinants \cite{Risteski2001}, Grassmann \cite{Gorski1968,Jiang1996,Venticos1956,Wedin1983} or Clifford algebras \cite{Dorst2007}, and sometimes lie behind other formalisms \cite{Afriat1957,Miao1992}. 
They are often defined, or results proven, only for real spaces, but not all concepts adapt well to complex spaces, whose geometry is important in applications like quantum information and computation \cite{Bengtsson2017,Nielsen2006quantum,Ortega2002}.

We organize and extend the theory on such angles, refining and unifying them in a \emph{Grassmann angle}.
Its properties are systematically developed, and results scattered in the literature are brought together and generalized, in a formalism that makes them clearer and proofs simpler.

A subtly important alteration makes the angle asymmetric for subspaces of distinct dimensions, reflecting their dimensional asymmetry. This leads to better properties and more general results, like a triangle inequality for angles between subspaces of different dimensions.

Our angle works well in complex spaces, and most results remain valid, though it is not the same as the angle in the underlying real spaces. Its relation to volume projection changes in a crucial way, which is related to complex Pythagorean theorems for volumes \cite{Mandolesi_Pythagorean}, with important consequences for quantum theory \cite{Mandolesi_Born}. 

Grassmann angles correspond to angles in Grassmann algebra. They give the Fubini-Study metric in Grassmannians, an asymmetric metric in the full Grassmannian, making it more than a disjoint union of Grassmannians, and Hausdorff distances between full sub-Grassmannians. 

The angle with a subspace and with its orthogonal complement are linked to products of cosines and sines of principal angles studied by some authors \cite{Afriat1957,Miao1992,Miao1996}.
The relation between these angles is not that of an usual angle complement, being quite intricate. Its analysis gives an obstruction on complex structures for pairs of real subspaces. 

In \cite{Mandolesi_Products} we relate Grassmann angles to various products of Grassmann and Clifford algebras (inner and exterior products, contractions, Clifford geometric product, etc.), and use them to get geometric interpretations for the Clifford product and some of its properties.

Here we use the products to get formulas for computing the angles in arbitrary bases, and identities for the angles with certain families of subspaces, some of which are related to generalized Py\-thag\-o\-re\-an theorems \cite{Mandolesi_Pythagorean} and quantum probabilities.

\Cref{sc:preliminaries} has concepts and results we will need. 
\Cref{sc:Grassmann angle} introduces the Grassmann angle and its basic properties, and \cref{sc:metric} describes its metric properties. 
\Cref{sc:Complementary Grassmann angle} studies the angle with an orthogonal complement. 
\Cref{sc:Formulas} gives formulas for computing the angles.
\Cref{sc:Pythagorean trigonometric,sc:other identities} give various identities, and \cref{sc:conclusion} closes with a few remarks.
\Cref{sc:Related angles} reviews similar angles.

\section{Preliminaries}\label{sc:preliminaries}

In this article $X$ is a $n$-dimensional vector space over $\R$ (real case) or $\C$ (complex case), with an inner product  $\inner{\cdot,\cdot}$ (Hermitian product in the complex case, with conjugate linearity in the left  entry).
In the complex case, $V_\R$ denotes the underlying real space of a complex subspace $V$, with inner product $\operatorname{Re}\inner{\cdot,\cdot}$.

A \emph{$k$-subspace} is a $k$-dimensional subspace, and a \emph{line} is a 1-subspace. Given $v\in X$, the sets $\R v=\{cv:c\in\R\}$ and, in the complex case, $\C v=\{cv:c\in\C\}$ are, respectively, the \emph{real} and \emph{complex lines} of $v$, if $v\neq 0$. 
In the complex case, $\R v$ is to be understood as a real line in $X_\R$.

$\Proj_W$ and $\Proj^V_W$ denote, respectively, the orthogonal projections $X\rightarrow W$ and $V\rightarrow W$,  for subspaces $V,W\subset X$.

\subsection{Angles between real or complex vectors}

We review definitions of angles between vectors, as the complex case has different angles \cite{Scharnhorst2001}. We also include the $0$ vector.

\begin{definition}
	For nonzero $v,w\in X$, the \emph{(Euclidean) angle} $\theta_{v,w}\in[0,\pi]$, the \emph{Hermitian angle} $\gamma_{v,w}\in[0,\frac{\pi}{2}]$, the \emph{complex angle} $\zeta_{v,w}\in\C$ and, if $v\not\perp w$, the \emph{phase difference} $\phi_{v,w}\in\R$ from $v$ to $w$, are given by
	\begin{equation}\label{eq:angles vectors}
	\begin{aligned}
		\cos\theta_{v,w} &= \frac{\operatorname{Re}\inner{v,w}}{\|v\| \|w\|}, \hspace{40pt}&
		\cos\gamma_{v,w} &= \frac{|\langle v,w \rangle|}{\|v\| \|w\|}, \\
		\cos\zeta_{v,w} &= \frac{\inner{v,w}}{\|v\| \|w\|}, &
		e^{i\phi_{v,w}} &= \frac{\inner{v,w}}{|\inner{v,w}|}. 
	\end{aligned}
	\end{equation}
	We also define $\theta_{0,0} = \gamma_{0,0} = \zeta_{0,0} = \theta_{0,v} =\gamma_{0,v} = \zeta_{0,v} =0$ and $\theta_{v,0} = \gamma_{v,0}  = \zeta_{v,0} = \frac{\pi}{2}$. 
\end{definition}

Angles with $0$ are unusual, and have the asymmetry $\theta_{0,v}\neq \theta_{v,0}$. We will find this to be helpful, being due to $\dim \R 0\neq\dim \R v$.

In the real case, $\zeta_{v,w} = \theta_{v,w}$ is the usual angle between vectors, $\gamma_{v,w} = \min\{\theta_{v,w},\pi-\theta_{v,w}\}$, and $\phi_{v,w} = 0$ or $\pi$.
In the complex case (\cref{fig:angulos vetores complexos}), $\theta_{v,w}$ is the usual angle in $X_\R$, $\gamma_{v,w} = \theta_{v,Pv}$ for $P=P_{\C w}$, and $\phi_{v,w}$ is the oriented angle from $Pv$ to $w$ (positive in the orientation given by the complex structure). If $v\not\perp w$ then $\cos\zeta_{v,w} = e^{i\phi_{v,w}} \cos\gamma_{v,w}$ and
\begin{equation}\label{eq:spherical Pythag vectors}
	\cos\theta_{v,w} = \cos\phi_{v,w} \cos\gamma_{v,w}.
\end{equation} 

\begin{figure}
	\centering
	\includegraphics[width=0.7\linewidth]{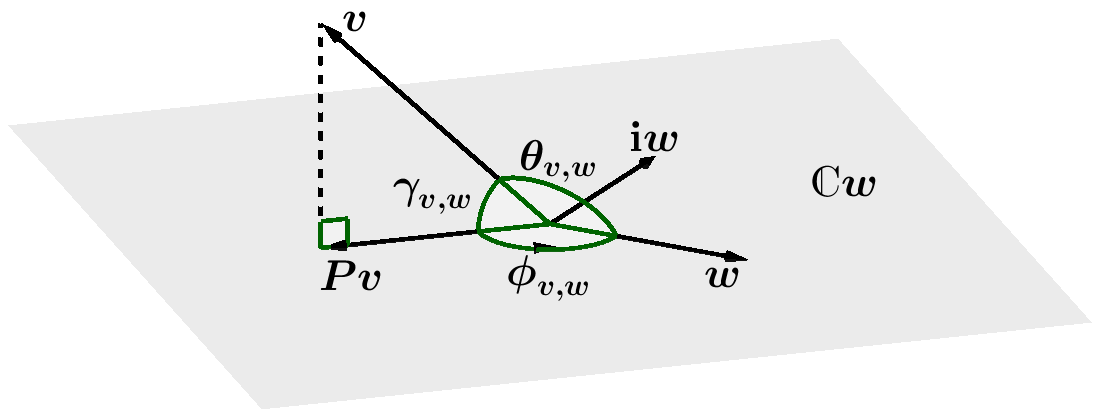}
	\caption{Angles for complex vectors.}
	\label{fig:angulos vetores complexos}
\end{figure}

\begin{definition}
	In the complex case, $v,w\in X$, $v\not\perp w$, are \emph{aligned} if $\phi_{v,w}=0$.
\end{definition}

This happens when $\inner{v,w}>0$, and means $Pv=cw$ for some $c>0$.
The phase difference is such that $e^{i\phi_{v,w}}v$ gets aligned with $w$.

\begin{example}
	In $\C^2$, for $v=(1,1)$ and $w=(i-1,0)$ we have $\theta_{v,w}=120^\circ$, $\gamma_{v,w}=45^\circ$, $\phi_{v,w}=135^\circ$ and $\zeta_{v,w}= \arccos(\frac{\sqrt{2}}{2} e^{i\frac{3\pi}{4}})$. 
\end{example}

%
%
%
%

\subsection{Angles between subspaces}

In high dimensions the relative position of subspaces can be more complicated than in $\R^3$, and many angle concepts are used to describe it.
We review some of them, and specify simple cases in which we drop the qualifiers and just talk about \emph{the} angle between subspaces.

\subsubsection{Minimal and directed maximal angles}

Let $V, W\subset X$ be nonzero subspaces.

\begin{definition}
The \emph{minimal angle} $\theta^{\min}_{V,W}\in[0,\frac{\pi}{2}]$ between  $V$ and $W$ is
$ \theta^{\min}_{V,W} = \min\left\{ \theta_{v,w} : v\in V, w\in W, v\neq 0, w\neq 0 \right\}$. 
\end{definition} 

In the complex case, using $\gamma_{v,w}$ instead of $\theta_{v,w}$ gives the same result. 

Even if the minimal angle is useful at times, the information it provides is limited. For example, $\theta^{\min}_{V,W} =0$ if, and only if, $V\cap W\neq\{0\}$, in which case it tells us nothing else about the relative position of $V$ and $W$.

\begin{definition}
The \emph{directed maximal angle} $\theta^{\max}_{V,W}\in[0,\frac{\pi}{2}]$ from $V$ to $W$ is
$ \displaystyle\theta^{\max}_{V,W} = \max_{v\in V}\min_{w\in W} \theta_{v,w}$.
\end{definition} 

If $\dim V\leq \dim W$ then $\theta^{\max}_{V,W}$ is their largest principal angle, described below, otherwise $\theta^{\max}_{V,W}=\frac{\pi}{2}$.

\subsubsection{Principal angles}

Detailed information about the relative position of subspaces requires a list of principal angles \cite{Galantai2006,Golub2013,Jordan1875}.

\begin{definition}
	Let $V,W\subset X$ be nonzero subspaces, $p=\dim V$, $q=\dim W$ and $m=\min\{p,q\}$.
	Orthonormal bases $(e_1,\ldots,e_p)$ of $V$ and $(f_1,\ldots,f_q)$ of $W$ are associated \emph{principal bases}, formed by \emph{principal vectors}, with \emph{principal angles} $0\leq \theta_1\leq\ldots\leq\theta_m\leq\frac \pi 2$, if
	\begin{equation}\label{eq:inner ei fj}
		\inner{e_i,f_j} = \delta_{ij}\cos\theta_i.
	\end{equation}
\end{definition}

Note that $\theta_i = \theta_{e_i,f_i}$ ($=\gamma_{e_i,f_i}$ in the complex case) for $1\leq i\leq m$.
The number of null principal angles equals $\dim V\cap W$.

Principal bases can be obtained via a singular value decomposition: for $P=\Proj^V_W$, the $e_i$'s and $f_i$'s are orthonormal eigenvectors of $P^*P$ and $PP^*$, respectively, and the eigenvalues of $P^*P$, if $p\leq q$, or $PP^*$ otherwise, are the $\cos^2\theta_i$'s.
The $\theta_i$'s are uniquely defined, but the $e_i$'s and $f_i$'s are not  (e.g., $-e_i$ and $-f_i$ are alternative principal vectors).

They can also be described recursively: $e_1$ and $f_1$ form the minimal angle $\theta_1$ of $V$ and $W$; in their orthogonal complements we obtain $e_2$, $f_2$ and $\theta_2$ in the same way; and so on.
For $i>m$ other vectors are chosen to complete an orthonormal basis.

A geometric interpretation is that the unit sphere of $V$ projects to an ellipsoid in $W$ (\cref{fig:principal vectors}).  In the real case, for $1\leq i\leq m$, the $e_i$'s project onto its semi-axes, of lengths $\cos\theta_i$, and the $f_i$'s point along them. In the complex case, for each $1\leq i\leq m$ there are two semi-axes of equal lengths, corresponding to projections of $e_i$ and $\im e_i$. In particular, this means each principal angle will be twice repeated in the underlying real spaces.

\begin{figure}
	\centering
	\includegraphics[width=0.7\linewidth]{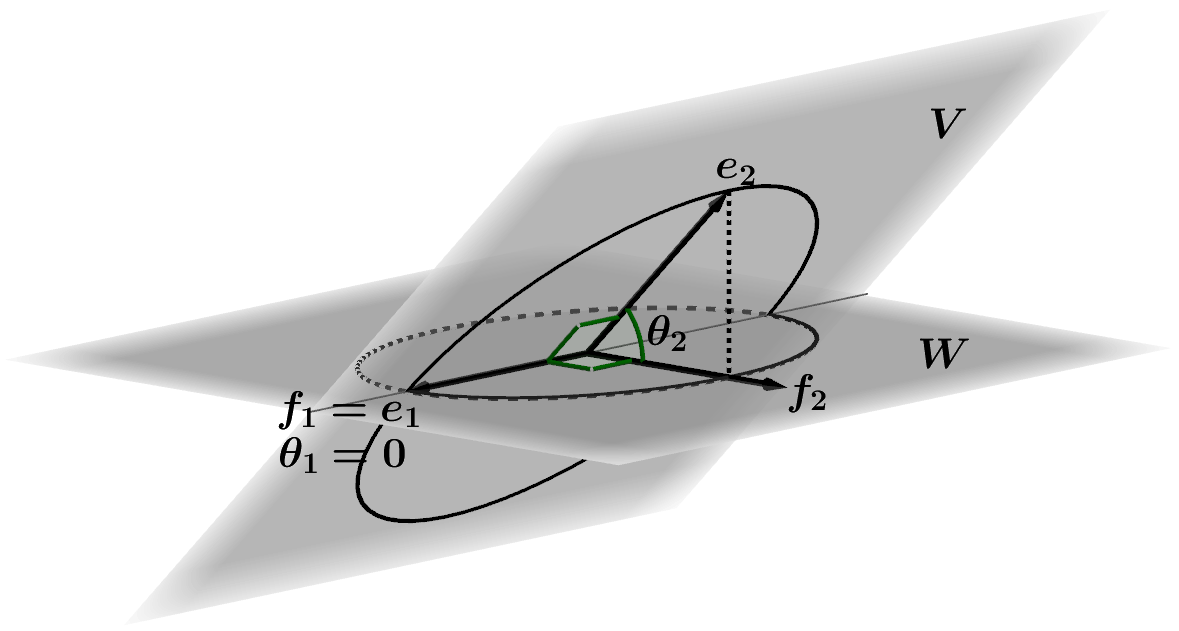}
	\caption{Principal vectors and angles for two planes in $\R^3$.}
	\label{fig:principal vectors}
\end{figure}

\begin{proposition}\label{pr:Proj diagonal}
$P=\Proj^V_W$ is represented in associated principal bases by a $q\times p$ diagonal matrix formed with the $\cos\theta_i$'s, and 
\begin{equation}\label{eq:Pei}
	Pe_i=\begin{cases}
		f_i\,\cos\theta_i \ \text{ if } 1\leq i\leq m, \\
		0 \hspace{31pt}\text{ if } i>m.
	\end{cases}
\end{equation}
\end{proposition}

\begin{example}\label{ex:real principal angles}
In $\R^4$,  $e_1=(1,0,1,0)/\sqrt{2}$, $e_2=(0,1,0,1)/\sqrt{2}$, $f_1=(1,0,0,0)$ and $f_2=(0,1,0,0)$ are principal vectors for $V=\Span(e_1,e_2)$ and $W=\Span(f_1,f_2)$, with principal angles $\theta_1=\theta_2=45^\circ$.
\end{example}

\begin{example}\label{ex:complex principal angles}
In $\C^4$, $e_1=(1,1,0,0)/\sqrt{2}$, $e_2=(0,0,\im ,\sqrt{3})/2$, $f_1=(1+\im ,1-\im,0,0)/2$ and $f_2=(0,0,\im ,0)$ are principal vectors for $V=\Span_\C(e_1,e_2)$ and $W=\Span_\C(f_1,f_2)$, with principal angles $\theta_1=45^\circ$ and $\theta_2=60^\circ$.
In the underlying $\R^8$, these subspaces have as principal vectors
\[\begin{aligned}
e_1 &= (1,0,1,0,0,0,0,0)/\sqrt{2}, &f_1 &= (1,1,1,-1,0,0,0,0)/2, \\
\tilde{e}_1 & = \im e_1 = (0,1,0,1,0,0,0,0)/\sqrt{2}, &\tilde{f}_1 &= \im f_1 = (-1,1,1,1,0,0,0,0)/2, \\
e_2 &= (0,0,0,0,0,1,\sqrt{3},0) /2, &f_2 &= (0,0,0,0,0,1,0,0),\\
\tilde{e}_2 &= \im e_2 = (0,0,0,0,-1,0,0,\sqrt{3})/2, &\tilde{f}_2 &= \im f_2 = (0,0,0,0,-1,0,0,0), \end{aligned}\]
with principal angles $\theta_1=\tilde{\theta}_1=45^\circ$ and $\theta_2=\tilde{\theta}_2=60^\circ$. 
\end{example}

Principal angles fully describe the relative position of two subspaces.
The following result is proven in \cite{Wong1967} for the real case, but the proof also works for the complex one.

\begin{proposition}\label{pr:relative position}
Given two pairs  $(V,W)$ and $(V',W')$ of subspaces of $X$, with $\dim V'=\dim V$ and $\dim W'=\dim W$, there is an orthogonal transformation (unitary, in the complex case) taking $V$ to $V'$ and $W$ to $W'$ if, and only if, both pairs have the same principal angles.
\end{proposition}

Dealing with a list of angles can be cumbersome, and in applications one often uses  whichever angle better describes the properties of interest.
This result shows any good angle concept must be a function of principal angles, and several concepts of distance between subspaces, or metrics in Grassmannians, are also given in terms of them \cite{Edelman1999,Hamm2008,Ye2016}.
\SELF{Other refs: Qiu2005,Wong1967, Neretin2001,Wedin1983} 

But even if an angle captures some important relation between the subspaces, it can miss on other information. Unfortunately, angles between subspaces are often presented without an explanation of their shortcomings, and this can lead to misunderstandings.

\subsubsection{The angle in simple cases}

Some cases are simple enough that the relation between the subspaces can be expressed in easier terms, without risking ambiguity.

\begin{definition}
We define \emph{the angle} $\theta_{V,W}$ between subspaces $V,W\subset X$ only in the following cases:
\begin{enumerate}[i)]
\item  $\theta_{\{0\},V}=0$ and, if $V\neq\{0\}$, $\theta_{V,\{0\}}=\frac{\pi}{2}$.
\item If their principal angles are $(0,\ldots,0,\theta_m)$ then $\theta_{V,W}=\theta_m$.
\item If their principal angles are $(0,\ldots,0,\frac{\pi}{2},\ldots,\frac{\pi}{2})$ then $\theta_{V,W}=\frac{\pi}{2}$.
\end{enumerate}
\end{definition}

Case (i) is for convenience. In (ii) there might be no 0's, and $\theta_m$ can be 0. In (iii) there might be no 0's, but there must be at least one $\frac{\pi}{2}$.

In particular, $\theta_{V,W}$ is defined if $V$ or $W$ is a line, or if $V\cap W$ has codimension 1 in $V$ or $W$, coinciding with the usual angle between line and subspace, or the dihedral angle between hyperplanes. 



\begin{proposition}\label{pr:the angle lines}
Let $v,w\in X$, $V=\Span(v)$ and $W=\Span(w)$.
\begin{enumerate}[i)]
\item $\theta_{\R v,\R w} = \min\{\theta_{v,w}, \pi-\theta_{v,w}\}$.\label{it:angle real lines}
\item $\theta_{\C v,\C w} =\gamma_{v,w}$, in the complex case.\label{it:angle complex lines}
\item $|\langle v,w \rangle| = \|v\| \|w\|\cos \theta_{V,W}$. \label{it:theta span v w}
\end{enumerate}
\end{proposition}

\subsection{Grassmann algebra}

In Grassmann algebra $\Lambda X$ \cite{Marcus1975,Yokonuma1992}, \emph{$p$-blade} is a simple multivector $\nu=v_1\wedge\ldots\wedge v_p \in\Lambda^p X$ of grade $p$, with $v_1,\ldots,v_p\in X$. If $\nu\neq 0$, it \emph{represents} a $p$-subspace $V=\Span(v_1,\ldots,v_p) =\{x\in X:x\wedge\nu=0\}$, and $\Lambda^p V =\Span(\nu)$.
A nonzero
\SELF{Melhor excluir $\nu=0$, pois $\Lambda^0 V=\R=\Span(\nu)$ and $V=\Ann(\nu)$ only for $\nu\neq 0$
}
scalar $\nu\in\Lambda^0 X$ is a $0$-blade, representing $\{0\}$. 

The inner product of $\nu=v_1\wedge\ldots\wedge v_p$ and $\omega=w_1\wedge\ldots\wedge w_p$ is $\inner{\nu,\omega} = \det\!\big(\inner{v_i,w_j}\big)$. It is extended linearly (sesquilinearly, in the complex case), with distinct $\Lambda^p X$'s being  orthogonal, and $\inner{\nu,\omega}=\bar{\nu}\omega$ for  $\nu,\omega\in\Lambda^0 X$.

The norm $\|\nu\|=\sqrt{\inner{\nu,\nu}}$ of $\nu=v_1\wedge\ldots\wedge v_p$ gives, in the real case, the $p$-dimensional volume of the parallelotope spanned by $v_1,\ldots,v_p$.  

Given a subspace $W\subset X$, $P=\Proj_W:X\rightarrow W$ extends naturally to $P=\Proj_{\Lambda W}:\Lambda X \rightarrow \Lambda W$, and 
$P(v_1\wedge\ldots\wedge v_p) = Pv_1\wedge\ldots\wedge Pv_p$.

For $\nu\in\Lambda^p X$, the \emph{(left) contraction} $\nu\lcontr$ is the adjoint of the (left) exterior product $\nu\wedge$, i.e. for any $\omega\in\Lambda^q X$ and $\mu\in\Lambda^{q-p} X$,
\begin{equation}\label{eq:contraction}
	\inner{\mu,\nu \lcontr  \omega} = \inner{\nu\wedge\mu,\omega}.
\end{equation}
This operation is asymmetric, with $\nu\lcontr\omega=0$ if $p>q$, and generalizes the inner product, as $\nu\lcontr\omega = \inner{\nu,\omega}$ if $p=q$.

\subsection{Orientation}

In the real case, there are various ways to specify one of the 2 possible orientations of a subspace $V$.
In the complex case, the complex structure naturally picks a preferred orientation in $V_\R$, and this usually ends all talk of orientations in complex spaces. 

We adopt a definition of orientation that is common in the real case, but also works in the complex one, with the difference that it gives a continuum of complex orientations instead of only 2 alternatives.

\begin{definition}
	An \emph{orientation} of a $p$-subspace $V$ is an unit $\nu\in\Lambda^p V$, or, equivalently, the equivalence class $\{c\,\nu:c>0\}$ of a nonzero $\nu\in\Lambda^p V$.
	An \emph{oriented subspace} is a subspace with a choice of orientation.
	A blade \emph{orients} $V$ if it represents $V$ and gives its orientation.
\end{definition}

The set of orientations of $V$ is the unit circle in $\Lambda^p V$ (in the real case, 2 points).
In the complex case $U(V)/SU(V)$ gives the same result, generalizing the usual definition $O(V)/SO(V)$ of the set of real orientations.

\subsection{Grassmannians}

In a complex projective space, the Fubini-Study distance  \cite{Goldman1999} 
between lines is their angle, as in \cref{pr:the angle lines}\emph{\ref{it:angle complex lines}}. 
In a real projective space, with the round metric (i.e. as a quotient of the unit sphere), the distance between  lines is also their angle, now as in \cref{pr:the angle lines}\emph{\ref{it:angle real lines}}. 
In both cases we call the angle between lines their \emph{Fubini-Study distance} $d_{FS}$, 
\[ d_{FS}\big(\!\Span(v),\Span(w)\big) = \arccos \frac{|\inner{v,w}|}{\|v\| \|w\|}. \]
Geodesics for the Fubini-Study metric are, in $\R\PP^n$, quotients of great circles of the sphere $S^{n}$, and in  $\C\PP^n$ they are great circles in the complex projective line $\C\PP^1\cong S^2$ determined by any two points.

For a a $p$-subspace $V\subset X$, the \emph{Grassmannian} $G_k(V)$ is the set of its $k$-subspaces. With an appropriate differential structure, it is a compact manifold \cite{Griffiths1994,Kobayashi1996}.
\SELF{Compact since the compact group $O(V)$ or $U(V)$ acts transitively}
The \emph{full Grassmannian}  of all subspaces of $V$ is $G(V)=\cup_{k=0}^p G_k(V)$.
Its \emph{\Plucker\  embedding} into the projective space $\PP(\Lambda V)$ maps each $k$-subspace $U\subset V$ to its line $\Lambda^k U$.
Identifying $G(V)$ with its image,  $d_{FS}$ gives it a metric.
As distinct $\Lambda^k V$'s are mutually orthogonal, points in different $G_k(V)$'s are at  distance $\frac{\pi}{2}$.


We denote by $\tilde{G}_k(V)$ the Grassmannian of oriented k-subspaces of $V$, and $\tilde{G}(V)=\cup_{k=0}^p \tilde{G}_k(V)$ is the full Grassmannian of oriented subspaces.
Identifying each oriented subspace  with the unit blade orienting it, we embed $\tilde{G}(V)$ in the unit sphere $S(\Lambda V)$ of $\Lambda V$, as $\tilde{G}(V) = \{$unit blades in $\Lambda V\}$. 
The orthogonal group (unitary, in the complex case) of $V$ acts transitively on each $\tilde{G}_k(V)$, hence $\tilde{G}(V)$ is compact.
\SELF{Compact groups}

\subsection{Coordinate blades and subspaces}\label{sc:Coordinate decomposition}

\begin{definition}
	Let 
	$ \I_p^q= \{ (i_1,\ldots,i_p)\in\N^p : 1\leq i_1 < \ldots<i_p\leq q \}$, for integers $1\leq p\leq q$. 
	The \emph{norm} of a multi-index $\ii=(i_1,\ldots,i_p)\in\I_p^q$ is $\|\ii\|=i_1+\ldots+i_p$, and its \emph{complement} is $\ii'=(1,\ldots,\hat{i_1},\ldots,\hat{i_p},\ldots,q)\in\I_{q-p}^q$, where $\hat{i_k}$ indicates that index is removed.
	Also, let $\I_0^q=\{0\}$, with $\|0\|=0$ and $0'=(1,\ldots,q)$, and for $\ii\in\I_q^q$ let $\ii'=0$.
\end{definition}

\begin{definition}
	Given a basis $\beta=(w_1,\ldots,w_q)$ of $W\subset X$, and $1\leq p\leq q$, the \emph{coordinate $p$-subspaces} of $\beta$ are the $\binom{q}{p}$ subspaces given by 
	\begin{equation}\label{eq:coordinate subspaces}
		W_\ii = \Span(w_{ i_1},\ldots,w_{ i_p}),
	\end{equation}
	for $\ii=(i_1,\ldots,i_p)\in\I_p^q$,
	and represented by the \emph{coordinate $p$-blades} of $\beta$,
	\begin{equation}\label{eq:coordinate blades}
		\omega_\ii = w_{ i_1}\wedge\ldots\wedge w_{ i_p}\in \Lambda^p W_\ii.
	\end{equation}
	The coordinate $0$-subspace and $0$-blade of $\beta$ are $W_0=\{0\}$ and $\omega_0=1$. 
\end{definition}

When $\beta$ is orthonormal, $\{\omega_\ii\}_{\ii\in\I_p^q}$ is an orthonormal basis of $\Lambda^p W$.

\begin{definition}
	Given a decomposed nonzero blade $\omega= w_1\wedge\ldots\wedge w_q\in\Lambda^q X$, take the basis $\beta=(w_1,\ldots,w_q)$ of its subspace. For $0\leq p\leq q$ and $\ii\in \I_p^q$, the \emph{coordinate decomposition} of $\omega$ (w.r.t. $\ii$ and $\beta$) is
	\begin{equation}\label{eq:multiindex decomposition}
		\omega = \varepsilon_\ii \,\omega_\ii\wedge\omega_{\ii'},
	\end{equation}
	with $\omega_\ii$ and $\omega_{\ii'}$ as in \eqref{eq:coordinate blades}, and
	$\varepsilon_\ii = (-1)^{\|\ii\|+\frac{p(p+1)}{2}}$.
	\SELF{Tanto faz ser $(-1)^{\|\ii\|\pm \frac{p(p+1)}{2}}$}
\end{definition}

The exponent in $\varepsilon_\ii$ is due to the $(i_1-1)+\ldots+(i_p-p)$ transpositions needed to reorder $(\ii,\ii')$ as $(1,\ldots,q)$.
With this decomposition we can write an explicit formula for the contraction \cite{Mandolesi_Products}.

\begin{proposition}
	Given $\nu\in\Lambda^p X$ and a blade $\omega= w_1\wedge\ldots\wedge w_q\in\Lambda^q X$, with $p\leq q$, 
	\begin{equation}\label{eq:contraction coordinate decomposition} 
		\nu \lcontr  \omega = \sum_{\ii\in\I_p^q} \varepsilon_\ii \,\inner{\nu,\omega_\ii}\, \omega_{\ii'}.
	\end{equation}
\end{proposition}

\subsection{Principal partitions}

Let $V,W\subset X$ be nonzero subspaces and $P=\Proj_W$.

\begin{definition}
	A \emph{principal coordinate subspace}\footnote{Not to be confused with the term `principal subspace' used by some authors for $\Span(e_i,f_i)$, where $e_i\in V$ and $f_i\in W$ are corresponding principal vectors.} 
	of $V$ w.r.t. $W$ is a coordinate subspace of a principal basis of $V$ w.r.t. $W$. 
	Principal coordinate subspaces of the same basis are \emph{coprincipal}.
	A partition $V=\bigoplus_i V_i$ is \emph{principal} w.r.t. $W$ if the $V_i$'s are coprincipal subspaces of $V$ w.r.t. $W$.
\end{definition}

Note that $\{0\}$ is always principal, and principal partitions are orthogonal. 
For simplicity, we omit the reference to $W$ when there is no ambiguity.


\begin{proposition}\label{pr:principal blades proj orth}
	If $V_1,V_2\subset V$ are distinct coprincipal $r$-subspaces, represented by blades $\nu_1$ and $\nu_2$, then $\inner{\nu_1,\nu_2}=\inner{P\nu_1,P\nu_2}=0$. \SELF{Usa p/\cref{pr:identity Theta principal subspaces}.} 
\end{proposition}
\begin{proof}
	The principal basis has elements $e_1\in V_1$ and $e_2\in V_2$ such that $e_1\perp V_2$ and $e_2\perp V_1$, and therefore $\inner{\nu_1,\nu_2}=0$. 
	By \eqref{eq:Pei}, $Pe_1\perp P(V_2)$ and $Pe_2\perp P(V_1)$, and so $\inner{P\nu_1,P\nu_2}=0$.
\end{proof}

\begin{proposition}\label{pr:principal partition}
	Let $V=\bigoplus_i V_i$ be  an orthogonal partition. The following are equivalent:\SELF{Usa p/ \cref{pr:converse Theta partition}}
	\begin{enumerate}[i)]
		\item $V=\bigoplus_i V_i$ is a principal partition w.r.t. $W$.\label{it:sum Vi principal}
		\item $P(V)=\bigoplus_i P(V_i)$ is a principal partition w.r.t. $V$.\label{it:sum PVi principal}
		\item $P(V)=\bigoplus_i P(V_i)$ is an orthogonal partition.\label{it:sum PVi orthogonal}
	\end{enumerate}
\end{proposition}
\begin{proof}
	\emph{(i\,$\Rightarrow$\,ii)} The $P(V_i)$'s are pairwise disjoint by \eqref{eq:Pei}, coprincipal by \eqref{eq:inner ei fj}. 
	\emph{(ii\,$\Rightarrow$\,iii)} Immediate.
	\emph{(iii\,$\Rightarrow$\,i)} $P(V_i) \perp P(V_j)$ for $i\neq j$ implies $V_i \perp P(V_j)$. By \eqref{eq:inner ei fj}, the union of principal bases of the $V_i$'s is a principal basis of $V$.
\end{proof}

\subsection{Partial orthogonality}

We define a weaker concept of orthogonality for subspaces $V,W\subset X$.

\begin{definition}
	$V$ is \emph{partially orthogonal} to $W$ ($V\pperp W$) if there is a nonzero $v\in V$ such that $\inner{v,w}=0$ for all $w\in W$, i.e. if $V\cap W^\perp\neq\{0\}$. 
\end{definition}

If $\dim V=\dim W$ then $V \pperp W  \Leftrightarrow W\pperp V$, but in general this relation is asymmetric.
Even partial orthogonality both ways does not  imply $V\perp W$.
Some authors \cite{Afriat1957,Baksalary2009} say $V$ is \emph{completely inclined} to $W$ if $V\not\pperp W$, and if $W\not\pperp V$ as well they are \emph{totally inclined}.

\begin{proposition}\label{pr:partial orthogonality}
$V \pperp W \Leftrightarrow \dim P(V)<\dim V \Leftrightarrow \dim W < \dim V$ or a principal angle is $\frac{\pi}{2}$.
\end{proposition}

Partial orthogonality in $X$ corresponds to orthogonality in $\Lambda^p X$.

\begin{proposition}\label{pr:partial orth Lambda orth}
	$V\pperp W \Leftrightarrow \Lambda^p V \perp \Lambda^p W$, where $p=\dim V$.
\end{proposition}
\begin{proof}
	We can assume $V,W\neq\{0\}$ and $p\leq\dim W$.
\SELF{$V=0 \Rightarrow V\not\perp_p W$ and $\Lambda^0 V=\Lambda^0 W=\R$. \\ 
	  $V\neq0,W=0 \Rightarrow V\perp_p W$ and $\Lambda^p V\perp\Lambda^p 0=0$. \\
	  $V>W \Rightarrow V\perp_p W$ and $\Lambda^p W=0$
}
	Let $P=\Proj_W$ and $\nu=e_1\wedge\ldots\wedge e_p$ for a principal basis $(e_1,\ldots,e_p)$ of $V$ w.r.t. $W$.
	Then $\Lambda^p V=\Span(\nu)$ and $\|P\nu\| = \|Pe_1\wedge\ldots\wedge Pe_p\| =\prod_{i=1}^{p}\cos\theta_i$, by \eqref{eq:Pei}, so $\Lambda^p V \perp \Lambda^p W \Leftrightarrow P\nu=0 \Leftrightarrow \theta_p=\frac{\pi}{2} \Leftrightarrow V\pperp W$.
\end{proof}

\begin{proposition}\label{pr:P(V)}
	Let $(f_1,\ldots,f_q)$ be a principal basis of $W$ w.r.t. $V$,
	\SELF{fica implícito que $V,W\neq 0$}
	$\theta_1\leq\ldots\leq\theta_m$ be the principal angles, $r=\max\{i:\theta_i\neq\frac{\pi}{2}\}$, and $P=\Proj_W$. 
	If $V\not\perp W$ then $P(V)=\Span(f_1,\ldots, f_r)$ and the principal angles of $V$ and $P(V)$ are $\theta_1,\ldots,\theta_r$.
	If $V\not\pperp W$ then $r=\dim V$ and $P(V)$ is represented by $P\nu$, where $\nu$ is a blade representing $V$.
\end{proposition}
\begin{proof}
	Follows from \eqref{eq:Pei} and \cref{pr:partial orthogonality}, taking $\nu=e_1\wedge\ldots\wedge e_p$ for the associated principal basis $(e_1,\ldots,e_p)$ of $V$.
\end{proof}

\section{Grassmann angle}\label{sc:Grassmann angle}

The following angle is based on similar volume projection angles found in the literature (see \cref{sc:Related angles} for a review), unifying and extending them, and also introducing a small, yet important, modification.

\begin{definition}
Let $V,W\subset X$ be nonzero subspaces, with principal angles $\theta_1,\ldots,\theta_m$, where $m=\min\{\dim V,\dim W\}$.
The \emph{Grassmann angle} $\Theta_{V,W}\in[0,\frac{\pi}{2}]$ of $V$ with $W$ is 
\[ \Theta_{V,W}=\begin{cases}
\arccos(\cos\theta_1\cdot\ldots\cdot\cos\theta_m) \ \text{ if } \dim V\leq \dim W, \\
\frac{\pi}{2} \hspace{100pt} \text{ if } \dim V> \dim W.
\end{cases}\]
We also define $\Theta_{\{0\},\{0\}}=\Theta_{\{0\},V}=0$ and $\Theta_{V,\{0\}}=\frac{\pi}{2}$.
\end{definition}

Besides being defined for real or complex subspaces of same or distinct dimensions, what sets this angle apart from similar ones is its unusual asymmetry: in general, $\Theta_{V,W}\neq \Theta_{W,V}$ when dimensions are different.
It reflects the dimensional asymmetry between the subspaces, leading to better and more general results (see \cref{sc:Asymmetry}). 
Such feature is so relevant, in fact, that we used the above definition just to make it explicit, even though there are simpler alternatives (e.g. \cref{pr:Theta Pnu}).

Anyway, as principal angles are symmetric with respect to interchange of $V$ and $W$, so is $\Theta_{V,W}$ when dimensions are equal.

\begin{proposition}
If $\dim V=\dim W$ then $\Theta_{V,W}=\Theta_{W,V}$. 
\end{proposition}

In simple cases, where there is an unambiguous concept of angle between subspaces, $\Theta_{V,W}$ coincides with it, as when $V$ is a line, or $V$ and $W$ are planes in $\R^3$.
It also has many usual angle properties. 

\begin{proposition}\label{pr:properties Grassmann}
Let $U,V,W\subset X$ be subspaces and $P=\Proj_W$. 
\begin{enumerate}[i)]
\item $\Theta_{V,W}=0 \ \Leftrightarrow\ V\subset W$.\label{it:Theta zero}
\item $\Theta_{V,W}=\frac{\pi}{2} \ \Leftrightarrow\ V \pperp W$. \label{it:Theta pi2}
\item If  $\theta_{V,W}$ is defined and $\dim V\leq\dim W$ then $\Theta_{V,W}=\theta_{V,W}$. \label{it:Theta theta}
\item If $L\subset X$ is a line and $v\in L$ then $\|Pv\|=\|v\|\cos\Theta_{L,W}$. \label{it:Pv}
\item $\Theta_{V,W}=\Theta_{V,P(V)}$.\label{it:Theta P(V)}
\item If $U\perp V+W$ then $\Theta_{V,W}=\Theta_{V,W\oplus U}$.\label{it:Theta W+U} 
\item If $V'$ and $W'$ are the orthogonal complements of $V\cap W$ in $V$ and $W$, respectively, then $\Theta_{V,W}=\Theta_{V',W'}$. \label{it:orth complem inter} 
\item $\Theta_{T(V),T(W)} = \Theta_{V,W}$ for any orthogonal transformation $T:X\rightarrow X$ (unitary, in the complex case). \label{it:transformation}
\end{enumerate}
\end{proposition}

Still, interpreting $\Theta_{V,W}$ requires some care, specially in the complex case.
The next examples reveal strange features that we discuss later. 

\begin{example}\label{ex:Theta>thetas}
	In \cref{ex:real principal angles}, $\Theta_{V,W}= \arccos(\frac{\sqrt{2}}{2}\cdot\frac{\sqrt{2}}{2}) =60^\circ$, despite the fact that all lines in $V$ make a $45^\circ$ angle with $W$.
\end{example}

\begin{example}\label{ex:complex Grassmann angle}
	In \cref{ex:complex principal angles}, $\Theta_{V,W}=\arccos(\frac{\sqrt{2}}{2}\cdot\frac{1}{2})\cong 69.3^\circ$, but for the underlying real spaces $\Theta_{V_\R,W_\R}=\arccos(\frac{\sqrt{2}}{2}\cdot\frac{\sqrt{2}}{2}\cdot\frac{1}{2}\cdot\frac{1}{2}) \cong 82.8^\circ$. 
\end{example}

The angle can be computed via projection matrices, as follows. We give more general formulas in \cref{sc:Formulas}.

\begin{proposition}\label{pr:determinant projection}
If $P$ is a matrix representing $\Proj^V_W$ in orthonormal bases of $V$ and $W$ then $\cos^2\Theta_{V,W}=\det(\bar{P}^T P)$.
If $\dim V=\dim W$ then $\cos\Theta_{V,W}=|\det P|$.
\end{proposition}
\begin{proof}
It is enough to consider principal bases of $V$ and $W$, for which the result follows from \cref{pr:Proj diagonal}. 
\end{proof}

Grassmann angles satisfy a spherical Pythagorean theorem  (\cref{fig:spherical Pythagorean}), which is valid without dimensional conditions thanks to the asymmetry (to see why, consider planes $V,W\subset \R^3$ and $U=V\cap W$).

\begin{theorem}\label{pr:spherical Pythagorean theorem}
Let $V,W\subset X$ and $U\subset W$ be subspaces, and $P=\Proj_W$. Then
$\cos\Theta_{V,U} = \cos\Theta_{V,P(V)}\cos\Theta_{P(V),U}$.
\end{theorem}
\begin{proof}
Assume $V\not\pperp W$, so $\dim P(V)=\dim V$. If $P_1$, $P_2$, $P_3$ are matrices representing $\Proj^V_U$, $\Proj^{P(V)}_U$, $\Proj^V_{P(V)}$, respectively, in orthonormal bases, then $P_1=P_2P_3$
\SELF{$v=P_Wv+b$, $b\perp W$, so $P_U^V v = P_U v = P_U P_W v+0 = P_U^{PV} P_{PV}^V v$}
and $P_3$ is square, so  $\det(\bar{P}_1^T P_1) = |\det P_3|^2 \det(\bar{P}_2^TP_2)$.
\SELF{$= \det(\bar{P}_3^T\bar{P}_2^TP_2P_3) = \det \bar{P}_3^T \det(\bar{P}_2^TP_2)\det P_3$}
The result follows from \cref{pr:determinant projection}.
\end{proof}

\begin{figure}
\centering
\includegraphics[width=0.8\linewidth]{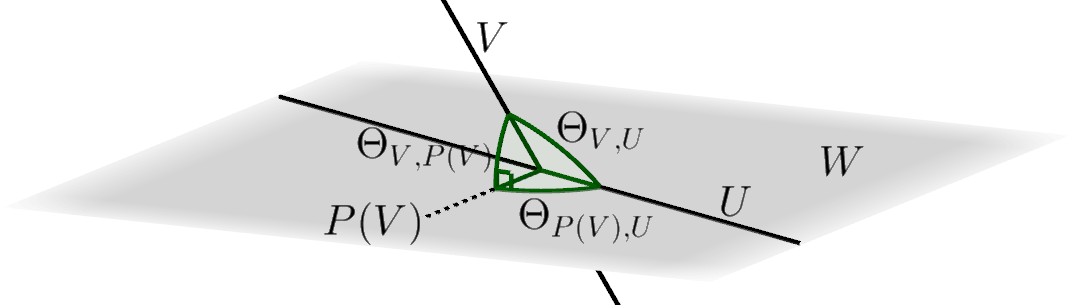}
\caption{$\cos\Theta_{V,U} = \cos\Theta_{V,P(V)}\cos\Theta_{P(V),U}$.}
\label{fig:spherical Pythagorean}
\end{figure}

\begin{corollary}\label{pr:W' sub W}
$\Theta_{V,W} \leq \Theta_{V,W'}$ for any subspace $W'\subset W$, with equality if, and only if, $V\pperp W$ or $P(V)\subset W'$, where $P=\Proj_W$. 
\end{corollary}



\begin{proposition}\label{pr:V' sub V}
$\Theta_{V,W} \geq \Theta_{V',W}$ for any subspace $V'\subset V$, with equality if, and only if, $V'\pperp W$ or $V'^\perp \cap V \subset W$. 
\end{proposition}
\begin{proof}
If  $V\pperp W$ the inequality is trivial, as $\Theta_{V,W}=\frac{\pi}{2}$, and equality is equivalent to $V'\pperp W$, which also happens if $V'^\perp \cap V \subset W$, as in this case the nonzero $v\in V$ which is orthogonal to $W$ must be in $V'$. 

If $V\not\pperp W$ then $V'\not\pperp W$, so  $\dim P(V) = \dim V$ and $\dim P(V') = \dim V'$, where $P=\Proj_W$. Complete orthonormal bases of $V'$ and $P(V')$ to orthonormal bases of $V$ and $P(V)$. If $P_1$, $P_2$ and $P_3$ are matrices for $\Proj^V_{P(V)}$, $\Proj^{V'}_{P(V')}$ and $\Proj^{V'^\perp\cap V}_{P(V')^\perp\cap P(V)}$, respectively, in these bases, then $P_1=\begin{psmallmatrix}
P_2 & B \\ 
0 & P_3
\end{psmallmatrix}$ for some matrix $B$. Thus
$ \cos\Theta_{V,P(V)} = |\det P_1| = |\det P_2| \cdot |\det P_3| \leq |\det P_2|  = \cos\Theta_{V',P(V')}$. We have equality if, and only if, $V'^\perp\cap V \subset P(V')^\perp\cap P(V)$, which happens if, and only if, $V'^\perp \cap V \subset W$. 
\end{proof}


Asymmetry keeps us from interchanging $V$ and $W$ in $\Theta_{V,W}$, but there is a workaround (\cref{fig:angle VWperp}).
If $\dim V=p<q=\dim W$, decompose $W = W_P \oplus W_\perp$, where $W_P=\Span(f_1,\ldots,f_p)$ and $W_\perp=\Span(f_{p+1},\ldots,f_q)$ for a principal basis $(f_1,\ldots,f_q)$ of $W$ w.r.t. $V$.
By \cref{pr:P(V)}, if $V\not\pperp W$ then $W_P=\Proj_W(V)$ and $W_\perp = W\cap V^\perp$, hence the notation.

\begin{proposition}
	If $\dim V<\dim W$ then $\Theta_{V,W} = \Theta_{W,V\oplus W_\perp}$.
\end{proposition}
\begin{proof}
	If $(e_1,\ldots,e_p)$ is the associated principal basis of $V$, and $\theta_1,\ldots,\theta_p$ are the principal angles, then, by \eqref{eq:inner ei fj}, $W$ and $W_\perp\oplus V$ have principal bases
	$(f_{p+1},\ldots,f_q,f_1,\ldots,f_p)$ and $(f_{p+1},\ldots,f_q,e_1,\ldots,e_p)$, respectively, with principal angles $0,\ldots,0,\theta_1,\ldots,\theta_p$.
\end{proof} 

\begin{figure}
	\centering
	\includegraphics[width=0.7\linewidth]{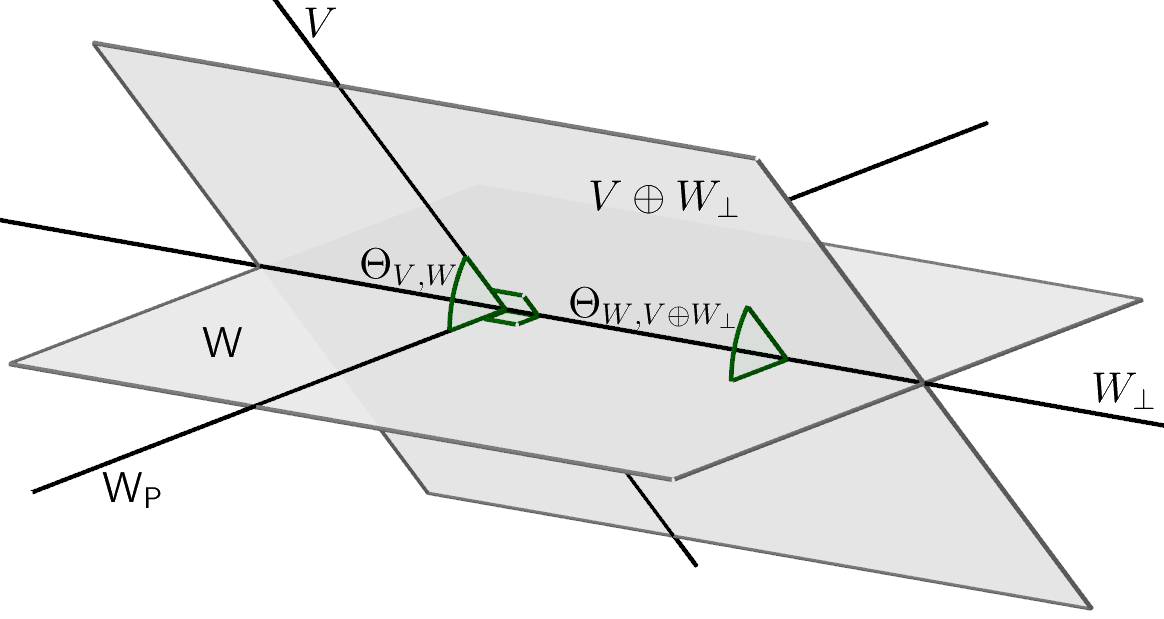}
	\caption{$\Theta_{V,W} = \Theta_{W,V\oplus W_\perp}$, if $\dim V<\dim W$.}
	\label{fig:angle VWperp}
\end{figure}

\subsection{An angle in the Grassmann algebra}\label{sc:relation Grassmann algebra}

We show $\Theta_{V,W}$ can be seen as an angle in the Grassmann algebra $\Lambda X$.
Let $V,W\subset X$ be subspaces represented by blades $\nu,\omega\in\Lambda X$, respectively.



\begin{proposition}\label{pr:Theta Pnu}
	$\|\P\nu\| = \|\nu\| \cos\Theta_{V,W}$, where $P=\Proj_W$.
\end{proposition}
\begin{proof}
	Let $p=\dim V$ and $q=\dim W$.
	If $p=0$ then $P\nu=\nu$ and $\Theta_{V,W}=0$, and if $p>q$ then $P\nu=0$ and $\Theta_{V,W}=\frac{\pi}{2}$.
	If $0<p\leq q$, we can take $\nu=e_1\wedge\ldots\wedge e_p$ for a principal basis $(e_1,\ldots,e_p)$ of $V$ w.r.t. $W$, and \eqref{eq:Pei} gives $\|\P\nu\| = \|Pe_1\wedge\ldots\wedge Pe_p\|= \prod_{i=1}^p\cos\theta_i = \cos\Theta_{V,W}$.
\end{proof}

Note the role the asymmetry plays in this formula's full generality. 

\begin{theorem}\label{pr:angle external powers}
$\Theta_{V,W}= \theta_{\Lambda^p V,\Lambda^p W}$, where $p = \dim V$.
\end{theorem}
\begin{proof}
By the above result, 
\SELF{\ref{pr:Theta Pnu}}
$\Theta_{V,W}=\theta_{\nu,P\nu}=\theta_{\Lambda^p V,\Lambda^p W}$, since $\Lambda^p V=\Span(\nu)$ and $P\nu$ is the orthogonal projection of $\nu$ on $\Lambda^p W$.
\end{proof}

So $\Theta_{V,W}$ is the angle,  in $\Lambda X$, between the line $\Lambda^p V$ and the subspace $\Lambda^p W$. 
If $\dim V=\dim W$, it is the angle (as in \cref{pr:the angle lines}) between the lines of $\nu$ and $\omega$ (images of $V$ and $W$ in the embedding $G(X)\hookrightarrow\PP(\Lambda X)$).
As an angle between subspaces becomes an angle with a line, we get a link with elliptic geometry, which, ultimately, is behind \cref{pr:spherical Pythagorean theorem}.


\begin{corollary}\label{pr:Theta inner blades}
$|\inner{\nu,\omega}| = \|\nu\|\|\omega\|\cos \Theta_{V,W}$, for same grade blades.
\end{corollary}

A generalization for the contraction of blades can be obtained by decomposing them in terms of principal vectors and using \eqref{eq:contraction coordinate decomposition} (for details, see \cite{Mandolesi_Products}).

\begin{proposition}\label{pr:Theta norm contraction}
	$\|\nu\lcontr\omega\| = \|\nu\|\|\omega\| \cos \Theta_{V,W}$, for any blades.
\end{proposition}

This formula holds without any grade conditions because the Grassmann angle asymmetry matches that of the contraction.

\subsection{Grassmann angles for oriented subspaces}

For oriented subspaces of same dimension, it is convenient to define angles that also describe relative orientations. 
Let $V,W\subset X$ be $p$-subspaces, oriented by $\nu,\omega\in\Lambda^p X$, respectively.

\begin{definition}
	If $V\not\pperp W$, the \emph{phase difference} $\phi_{V,W}$ from $V$ to $W$ is $\phi_{V,W} = \phi_{\nu,\omega}$ (as in \eqref{eq:angles vectors}).
	Orientations are \emph{aligned} if $\phi_{V,W}=0$.
\end{definition}

Orientations are aligned if $\Proj_W\nu$ gives $W$ the same orientation as $\omega$.

By \cref{pr:angle external powers} and \eqref{eq:spherical Pythag vectors}, $\cos \theta_{\nu,\omega} = \cos \phi_{V,W} \cdot \cos\Theta_{V,W}$, so $\theta_{\nu,\omega}$ reflects not only the separation between $V$ and $W$, but also their lack of alignment.
The complex angle \eqref{eq:angles vectors} does the same, but keeps these pieces of information separate. This and \cref{pr:Theta inner blades} suggest a generalization of $\Theta_{V,W}$.

\begin{definition}
The Grassmann angle $\mathbf{\Theta}_{V,W}$ for oriented subspaces of same dimension is given by 
\begin{equation}\label{eq:oriented Theta inner blades}
	\cos \mathbf{\Theta}_{V,W} = \frac{\inner{\nu,\omega}}{\|\nu\|\|\omega\|}.
\end{equation}
\end{definition}

In the real case $\mathbf{\Theta}_{V,W}\in[0,\pi]$, but $\mathbf{\Theta}_{V,W}\in\C$ in the complex one.

\begin{proposition}
This angle satisfies:
\begin{enumerate}[i)]
\item $\mathbf{\Theta}_{W,V} = \bar{\mathbf{\Theta}}_{V,W}$.
\item $\mathbf{\Theta}_{V,W} = \zeta_{\nu,\omega} \ (=\theta_{\nu,\omega}$ in the real case). \label{it:oriented angle blades}
\item If $V\not\pperp W$ then $\cos \mathbf{\Theta}_{V,W} =  e^{i\phi_{V,W}}\cos\Theta_{V,W}$.
\end{enumerate}
\end{proposition}

Note that \emph{(\ref{it:oriented angle blades})} is similar to \cref{pr:angle external powers}, but now we have an angle between the images of $V$ and $W$ in the embedding $\tilde{G}(X) \hookrightarrow S(\Lambda X)$.

\begin{example}\label{ex:oriented Theta}
Let $(e_1,e_2,e_3)$ be the canonical basis of $\C^3$, $v_1=(1,\im,0)$, $v_2=(\im,-1,-1)$, and $V$ and $X_{ij}$ be the subspaces oriented by $\nu=v_1\wedge v_2$ and $e_{ij}=e_i\wedge e_j$, respectively. Then:
\begin{itemize}
\item $\mathbf{\Theta}_{V,X_{12}} = \frac{\pi}{2}$, so that $V\pperp X_{12}$;
\item $\mathbf{\Theta}_{V,X_{13}} = \frac{3\pi}{4}$, so that $\Theta_{V,X_{13}} = \frac{\pi}{4}$  and $\phi_{V,X_{13}} = \pi$ (the projection of $\nu$ on $X_{13}$ has the orientation of $-e_{13}$);
\item $\mathbf{\Theta}_{V,X_{23}} = \arccos(\im\frac{\sqrt{2}}{2})$, so that $\Theta_{V,X_{23}} = \frac{\pi}{4}$ and $\phi_{V,X_{23}} = \frac{\pi}{2}$ (the orientation of $\im\nu$ is aligned with that of $e_{23}$).
\end{itemize}
\end{example}

%
%
%

\subsection{Projection factors}\label{sc:Projection Factors}

Let us take another look at \cref{pr:Theta Pnu}.
In the real case, $\|\nu\|$ and $\|\P\nu\|$ are $p$-dimensional volumes ($p=\dim V$) of a parallelotope and its orthogonal projection (\cref{fig:projecao}), so $\cos\Theta_{V,W}$ measures how volumes contract when projecting from $V$ to $W$.
If $p\leq \dim W$, this can be understood by noting that $\cos\Theta_{V,W}=\prod_i \cos\theta_i$ and each $\cos\theta_i$ is the factor by which lengths in a principal axis $\R e_i\subset V$ contract when projected.

\begin{figure}
	\centering
	\includegraphics[width=0.7\linewidth]{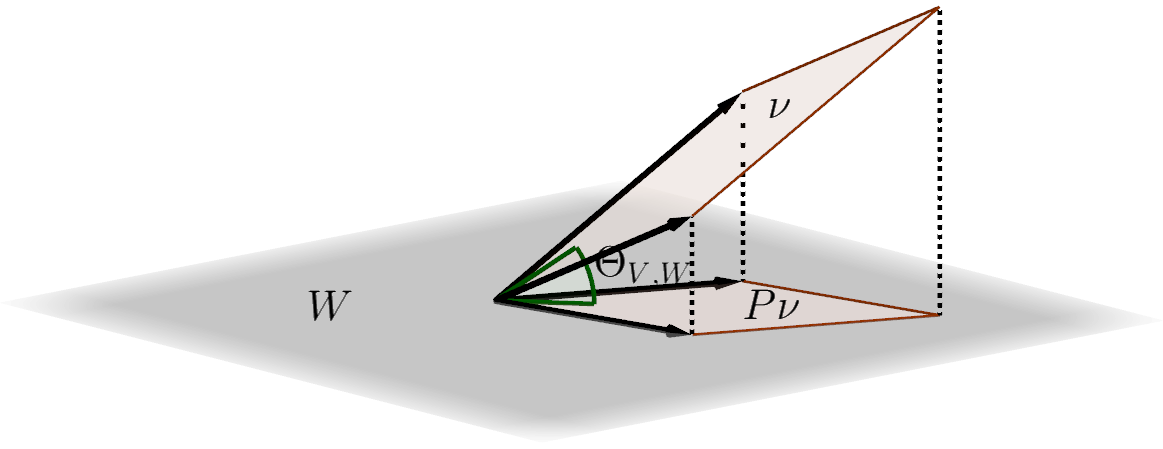}
	\caption{$\|\P\nu\| = \|\nu\| \cos\Theta_{V,W}$.}
	\label{fig:projecao}
\end{figure}

Some authors \cite{Gluck1967,Gunawan2005} take this as the defining property of the angle between subspaces, but the complex case requires an adjustment: 
top-dimensional volumes in $V$ contract by $\cos^2\Theta_{V,W}$, as each $\cos\theta_i$ describes the contraction of 2 axes, $\R e_i$ and $\R(\im e_i)$.

Let $V,W\subset X$ be subspaces, $P=\Proj_W$, and $|\!\cdot\!|_k$ be the $k$-dimensional Lebesgue measure (in the complex case, taken in the underlying real spaces $V_\R$ and $W_\R$, with twice the complex dimension).

\begin{theorem}\label{pr:Grassmann angle volume projection}
Let $S\subset V$ be a Lebesgue measurable set and $p=\dim V$.
\begin{enumerate}[i)]
\item $|P(S)|_p = |S|_p \cos\Theta_{V,W}$ in the real case. \label{it:real volume projection}
\item $|P(S)|_{2p} = |S|_{2p} \cos^2\Theta_{V,W}$ in the complex case. \label{it:complex volume projection}
\end{enumerate}
\end{theorem}
\begin{proof}
Consider first the real case, and assume $\dim V\leq\dim W$ (otherwise the result is trivial, as $|\!\cdot\!|_k =0$ on $W$ for $k>\dim_\R W$).
As $P$ is linear, the ratio of $|P(S)|_p$ to $|S|_p$ is independent of $S$.
Take $S$ to be the unit cube spanned by principal vectors $e_1,\ldots,e_p$ of $V$ w.r.t. $W$. By \eqref{eq:Pei}, $P(S)$ is the orthogonal parallelotope spanned by $f_1\cos\theta_1,\ldots,f_p\cos\theta_p$, so that $|P(S)|_p=\cos\Theta_{V,W}$. 
The complex case is similar, with $S$ being the unit cube spanned by $e_1,\im e_1,\ldots,e_p,\im e_p$, so each $\cos\theta_i$ is multiplied twice.
\end{proof}

A comparison with \cref{pr:Theta Pnu} suggests that to interpret $\|\nu\|$ in the complex case we should consider the square root of some volume of twice the complex dimension. 
Indeed, one can check that, for a complex $p$-blade $\nu=v_1\wedge\ldots\wedge v_p$, $\|\nu\|^2$ gives the $2p$-dimensional volume of the parallelotope spanned by $v_1,\im v_1,\ldots,v_p, \im v_p$. 

The theorem links Grassmann angles to projection factors \cite{Mandolesi_Pythagorean}. 

\begin{definition}
	The \emph{projection factor} of $V$ on $W$ is $\pi_{V,W}=\frac{|P(S)|_k}{|S|_k}$, for any Lebesgue measurable set $S\subset V$ with $|S|_k\neq 0$, where $k=\dim V_\R$.
\end{definition}

\begin{corollary}\label{pr:projection factor Theta}
	$\pi_{V,W} = \begin{cases}
		\cos\Theta_{V,W} \hspace{4pt}\text{ in the real case}; \\
		\cos^2\Theta_{V,W} \text{ in the complex case.}
	\end{cases}$
\end{corollary}

In the complex case $\pi_{V,W}=\pi_{V_\R,W_\R}$, as the Lebesgue measures are taken in the underlying real spaces. So we have:

\begin{corollary}\label{pr:cos underlying real Theta}
	In the complex case:
	\begin{enumerate}[i)]
		\item $\cos\Theta_{V_\R,W_\R} = \cos^2 \Theta_{V,W}$.
		\item $\Theta_{V_\R,W_\R} \geq \Theta_{V,W}$, with equality if, and only if, $V\subset W$ or $V\pperp W$. \label{it:Theta VR bigger}
	\end{enumerate}
\end{corollary}

This agrees with the results of \cref{ex:complex Grassmann angle}.
In \cref{sc:Underlying real spaces} we discuss what it means for $\Theta_{V,W}$ and $\Theta_{V_\R,W_\R}$ to be different.

\begin{example}
In \cref{ex:Theta>thetas}, areas in $V$ contract by half when orthogonally projected on $W$.
\end{example}

\begin{example}\label{ex:complex projection factors}
In \cref{ex:complex Grassmann angle}, $4$-dimensional volumes in $V$ contract by $\cos^2\Theta_{V,W}=\cos \Theta_{V_\R,W_\R}=\frac{1}{8}$ when projected on $W$.
\end{example}

\begin{example}\label{ex:sum projections}
	In \cref{ex:oriented Theta}, any $4$-dimensional volume in $V$ vanishes when projected on $X_{12}$, and shrinks by $\frac{1}{2}$ when projected on $X_{13}$ or $X_{23}$. Note that the sum of the projections equals the original volume. This corresponds to a complex Pythagorean theorem \cite{Mandolesi_Pythagorean}, which can be obtained by combining \cref{pr:Grassmann angle volume projection} with results of \cref{sc:Pythagorean trigonometric}.
\end{example}

\subsection{Exotic features}\label{sc:Exotic features}

Besides its nice properties, $\Theta_{V,W}$ also has strange ones. 
Some affect similar angles as well, but are not usually discussed, and can lead to error.
Still, if properly tapped they can be useful (see \cref{sc:simult complex}).

\subsubsection{Underdetermination of relative position}

By \cref{pr:relative position}, in general $\Theta_{V,W}$ does not fully describe the relative position of $V$ and $W$. But it was never meant to. Its purpose is just to capture, in a single number,  important properties of such position.

This is commonplace for angles between subspaces, but as it goes against most people's intuition regarding angles, a warning may help avoid misunderstandings:
even for pairs of subspaces with equal dimensions, having  equal Grassmann angles is no guarantee that an orthogonal or unitary transformation can take one pair to the other.

\subsubsection{Not an angle in $X$}\label{sc:not angle between lines}

Let $V,W\subset X$ be nonzero subspaces, with principal angles $\theta_1\leq\ldots\leq\theta_m$.
From the angle definition we get:

\begin{proposition}
	$\Theta_{V,W}\geq \theta_m$, with equality if, and only if, $\theta_m=\frac{\pi}{2}$, or $\dim V\leq\dim W$ and $\theta_i=0$ for all $i<m$.
\end{proposition}

This means $\Theta_{V,W}$ is not really an (ordinary) angle in $X$, in the sense of corresponding to an angle between a line in $V$ and its projection on $W$, as it is, in general, strictly greater than all such angles (as in \cref{ex:Theta>thetas}).
Indeed, the geometric interpretation of $\theta_m$, or \cref{pr:V' sub V}, gives:

\begin{corollary}
	There is a line $L\subset V$ with $\theta_{L,W}=\Theta_{V,W}$ if, and only if, $V\pperp W$ or $V\cap W$ has at most codimension 1 in $V$.
\end{corollary}

\subsubsection{Asymmetry}\label{sc:Asymmetry}

Grassmann angles are asymmetric by definition: if a line $V$ makes a $20^\circ$ angle with a plane $W$ then $\Theta_{V,W}=20^\circ$, but we chose to set $\Theta_{W,V}=90^\circ$.
This choice is unusual, as \cref{sc:Related angles} shows: for subspaces of different dimensions, it is customary to take the angle between the smaller one and its projection on the other, or some equivalent construction.

But we have good reasons for it.
For example, if $\dim V>\dim W$ we have $|P(S)|=0$ in \cref{pr:Grassmann angle volume projection}, and $\Lambda^p W=\{0\}$ in \cref{pr:angle external powers}. In both cases, consistency requires $\Theta_{V,W}=\frac{\pi}{2}$. 

Also, this asymmetry, far from being a problem, is quite beneficial. It reflects the asymmetry between subspaces of different dimensions, and allows the angle to carry some dimensional information ($\Theta_{V,W}\neq\frac{\pi}{2}$ implies $\dim \Proj_W V = \dim V \leq \dim W$), leading to simpler proofs and more general results.
For example, propositions \ref{pr:spherical Pythagorean theorem}, \ref{pr:Theta norm contraction} and others in next sections 
\SELF{\ref{pr:Theta Pnu}, \ref{pr:exterior product}, \ref{pr:formula any base dimension}, \ref{pr:formula complementary angle bases}}
only hold without restrictions thanks to it.

Still, symmetrized versions of $\Theta_{V,W}$ can be useful at times.

\begin{definition}
	The \emph{min-} and \emph{max-symmetrized Grassmann angles} are, respectively, $\check{\Theta}_{V,W} = \min\{\Theta_{V,W},\Theta_{W,V}\}$ and	$\hat{\Theta}_{V,W} = \max\{\Theta_{V,W},\Theta_{W,V}\}$. 
\end{definition}

The min-symmetrized one  agrees with how one usually talks about the angle between a plane and a line ($\check{\Theta}_{W,V}=20^\circ$ in the above example), and is often adopted implicitly (see \cref{sc:Related angles}).
But it has worse properties, as it loses information about principal vectors not corresponding to a principal angle. 
For example, it does not satisfy a triangle inequality (consider 2 lines and their plane).
Even so, it is related to the fat dot and Hestenes inner products of Clifford geometric algebra \cite{Mandolesi_Products}.

The max-symmetrized one may seem of little use, as $\hat{\Theta}_{V,W}=\frac \pi 2$ for any line and plane, or whenever dimensions are different. But with it we can extend \cref{pr:Theta inner blades} for distinct grades, and \cref{sc:metric} gives it other uses.
\SELF{\ref{pr:Theta Fubini any dim},\ref{pr:Hausdorff Theta},\ref{pr:Hausdorff Lambda}}
It is related to the scalar product of Clifford algebra \cite{Mandolesi_Products}.

\subsubsection{Orthogonal complement}

Another strange feature is that the Grassmann angle with the orthogonal complement of a subspace is not the usual complement of an angle.

\begin{example}\label{ex:Grassmann angle complement}
In \cref{ex:real principal angles}, $W^\perp=\Span(f_3,f_4)$ for $f_3=(0,0,1,0)$ and $f_4=(0,0,0,1)$, and both principal angles for $V$ and $W^\perp$ are also $45^\circ$. Thus $\Theta_{V,W^\perp}=60^\circ  \neq 90^\circ - \Theta_{V,W}$, since $\Theta_{V,W}=60^\circ$. 
\end{example}

As $\sin \Theta_{V,W} \neq \cos \Theta_{V,W^\perp}$, this sine does not measure projections on $W^\perp$, contrary to what one might expect from examples where $V$ is a line.
We discuss this in \cref{sc:Complementary Grassmann angle}, but for now note that $\Theta_{V,W}$ is an angle in $\Lambda X$, and in general $\Lambda (W^\perp) \neq \left(\Lambda W\right)^\perp$.

\subsubsection{Underlying real spaces}\label{sc:Underlying real spaces}

In \cref{ex:complex Grassmann angle}, $\Theta_{V,W}\neq\Theta_{V_\R,W_\R}$, which may seem strange, as metrically a complex space $X$ is not different from its underlying $X_\R$.
An explanation is that Grassmann algebras over $\R$ and $\C$ differ, with $\Lambda(X_\R)$ and $ (\Lambda X)_\R$ not even having the same dimension. 
Another is that $\Theta_{V,W}$ and $\Theta_{V_\R,W_\R}$ are different ways to encode information about the same projection factor $\pi_{V,W}=\pi_{V_\R,W_\R}$, as seen in \cref{pr:projection factor Theta} and \cref{ex:complex projection factors}.



One might say $\Theta_{V,W}$ ought to be defined as equal to $\Theta_{V_\R,W_\R}$, but this is inconvenient. Working in $X_\R$ doubles dimensions and squanders symmetries of the complex structure, causing the redundancy of principal angles appearing twice.
This would also conflict with other definitions  \cite{Galantai2006,Venticos1956} and make formulas differ in the complex case (e.g. \cref{pr:Theta inner blades} would become $|\langle v,w \rangle| = \|v\| \|w\|\sqrt{\cos \Theta_{V_\R,W_\R}}$).

\section{Metric properties}\label{sc:metric} 

We prove a triangle inequality for Grassmann angles, giving conditions for equality, show they give the Fubini-Study metric in Grassmannians $G_p(X)$, and analyze geodesics in $\PP(\Lambda^p X)$ connecting points of $G_p(X)$. 
They also give an asymmetric metric in the full Grassmannian $G(X)$, and Hausdorff distances between full sub-Grassmannians.

\subsection{Triangle inequality}\label{sc:Triangle inequality}

For $U,V,W\subset X$ of same dimension, \cref{pr:angle external powers} translates the spherical triangle inequality of elliptic geometry
\SELF{[p.38]{Reid2005}} 
(in the complex case, a triangle inequality for Hermitian angles) 
\SELF{[p.16]{Goldman1999}, em termos da métrica de Fubini}
into $\Theta_{U,W} \leq \Theta_{U,V} + \Theta_{V,W}$. 
The asymmetry extends it for distinct dimensions: if $P=\Proj_W$, $\Theta_{U,V}\neq\frac\pi 2$ and $\Theta_{V,W}\neq\frac\pi 2$ then $U$, $P_V(U)$ and $PP_V(U)$ have equal dimensions, 
and propositions \ref{pr:W' sub W}, \ref{pr:properties Grassmann}\emph{\ref{it:Theta P(V)}} and \ref{pr:V' sub V} give
$\Theta_{U,W} \leq \Theta_{U,PP_V(U)} \leq \Theta_{U,P_V(U)} + \Theta_{P_V(U),PP_V(U)} = \Theta_{U,V} + \Theta_{P_V(U),W} \leq \Theta_{U,V} + \Theta_{V,W}$.

We give a more detailed proof, so we can get conditions for equality.

\begin{theorem}\label{pr:spherical triangle inequality}
$\Theta_{U,W} \leq \Theta_{U,V} + \Theta_{V,W}$ for any subspaces $U,V,W\subset X$.
\end{theorem}
\begin{proof}
Let  $P=\Proj_W$, 
and assume $\Theta_{U,W}\neq 0$, $\Theta_{U,V}\neq \frac\pi 2$ and $\Theta_{V,W}\neq \frac\pi 2$, which implies $\Theta_{P_V(U),W}\neq\frac\pi 2$.
We can also assume $\Theta_{P_V(U),W}\neq 0$, for if $P_V(U)\subset W$ then $\Theta_{U,W} \leq \Theta_{U,P_V(U)} = \Theta_{U,V}$ by \cref{pr:W' sub W}.
These conditions give $r=\dim U=\dim P_V(U)\leq \dim V\leq \dim W$.
\SELF{e os denominadores abaixo não se anulam.}

Let $\mu$ and $\nu=P_V \mu/\|P_V \mu\|$ be unit blades representing $U$ and $P_V(U)$, respectively, and $\omega_\nu,\omega_\mu\in \Lambda^r W$, $\omega_\nu^\perp, \omega_\mu^\perp\in(\Lambda^r W)^\perp$ be given by
\SELF{$\omega_\mu=\omega_\nu$ if $U$ pperp $W$ é para ajudar na prova da \cref{pr:triangle equality}}
\begin{align*}
\omega_\nu&=\frac{P \nu}{\|P \nu\|},  & 
\omega_\mu&=\begin{cases}
			\frac{P \mu}{\|P \mu\|} \ \text{ if } \Theta_{U,W}\neq \frac\pi 2, \\ 
			\ \ \omega_\nu \ \ \text{ if } \Theta_{U,W}= \frac\pi 2, 
	\end{cases} \\ 
\omega_\nu^\perp &=\frac{\nu-P\nu}{\|\nu-P\nu\|},   &
\omega_\mu^\perp&=\frac{\mu-P\mu}{\|\mu-P\mu\|}.
\end{align*}
With \cref{pr:Theta Pnu} we obtain
\begin{align}
\mu &=\omega_\mu\cdot\cos\Theta_{U,W} + \omega_\mu^\perp\cdot\sin\Theta_{U,W}, \nonumber \\
\nu &=\omega_\nu\cdot\cos\Theta_{P_V(U),W} + \omega_\nu^\perp\cdot\sin\Theta_{P_V(U),W}. \label{eq:nu}
\end{align}
As $\inner{\mu,\nu}> 0$, \cref{pr:Theta inner blades} gives
\SELF{$\inner{\mu,\nu}\in\R \Rightarrow \inner{\mu,\nu} = \cos\Theta_{U,P_V(U)}$.}
\SELF{Se $\omega_\mu=e^{i\varphi}\omega_\nu$ dá uma lei dos cossenos elíptica hermitiana.}
\begin{align}
\cos\Theta_{U,P_V(U)}  
&= \inner{\omega_\mu,\omega_\nu} \cos\Theta_{U,W} \cos\Theta_{P_V(U),W}  \nonumber \\
& + \inner{\omega_\mu^\perp,\omega_\nu^\perp}  \sin\Theta_{U,W} \sin\Theta_{P_V(U),W} \nonumber \\
&\leq \cos(\Theta_{U,W}-\Theta_{P_V(U),W}), \label{eq:inequality}
\end{align}
and therefore $\Theta_{U,V} = \Theta_{U,P_V(U)} \geq \Theta_{U,W}-\Theta_{P_V(U),W} \geq \Theta_{U,W}-\Theta_{V,W}$, by \cref{pr:V' sub V}.
\SELF{As $\Theta_{U,P_V(U)}\in[0,\frac{\pi}{2}]$ and \\ $\Theta_{U,W}-\Theta_{P_V(U),W}\in[-\frac{\pi}{2},\frac{\pi}{2}]$}
\end{proof}

\begin{figure}
\centering
\includegraphics[width=0.9\linewidth]{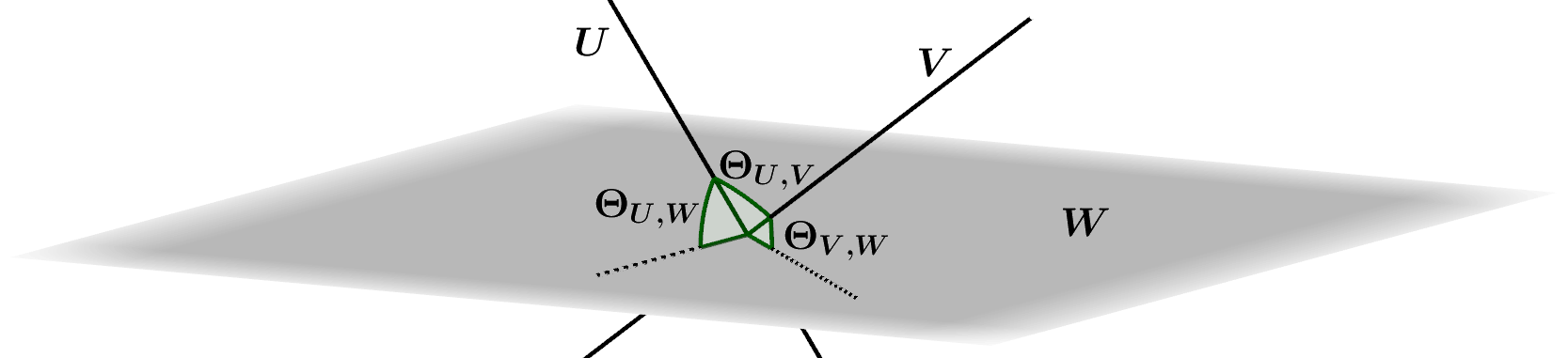}
\caption{Triangle inequality, $\Theta_{U,W} \leq \Theta_{U,V} + \Theta_{V,W}$.}
\label{fig:spherical triangle inequality}
\end{figure}

\Cref{fig:spherical triangle inequality} illustrates the inequality for lines $U$ and $V$ and a plane $W$.
The order of the subspaces in it is important, as moving the lines closer to the plane we can have $\Theta_{U,V} > \Theta_{U,W} + \Theta_{V,W}$, and $\Theta_{U,V} \leq \Theta_{U,W} + \Theta_{W,V}$ only holds due to the asymmetry.
On the other hand, the asymmetry makes it false that $\Theta_{U,W}\geq |\Theta_{U,V} - \Theta_{V,W}|$, preventing us from getting this, in the usual way, from the triangle inequality. Instead, we have:

\begin{corollary}
	$\Theta_{U,W}\geq \max\{ \Theta_{U,V} - \Theta_{W,V} \,,\, \Theta_{V,W} - \Theta_{V,U}\}$.
	\SELF{Aplica a desigualdade triangular a $\Theta_{U,V}$ and $\Theta_{V,W}$, tendo cuidado com a ordem.}
\end{corollary}

\begin{corollary}
		$\Theta_{U,W}\geq |\Theta_{U,V} - \Theta_{V,W}|$, for subspaces of same dimension.
\end{corollary}

To get conditions for equality in \cref{pr:spherical triangle inequality}, we need some lemmas.
The first one is an easy result of Grassmann algebra. Note that $\nu$ and $\omega$ need not be blades nor homogeneous multivectors.

\begin{lemma}\label{pr:wedge disjoint multivectors}
	Let $V,W\subset X$ be disjoint\footnote{i.e. with intersection $\{0\}$.}\ subspaces, $\nu\in\Lambda V$ and $\omega\in\Lambda W$. Then $\nu\wedge\omega = 0\ \Leftrightarrow\ \nu=0$ or $\omega=0$.
	\SELF{$(v_i),(w_j)$ bases $V,W$ \\ $\Rightarrow$ $(v_i,w_j)$ basis $V\oplus W$. \\[3pt]
		$(\nu_I)_{I\in\cup \I_i^p}$ basis $\Lambda V$, \\ $(\omega_J)_{J\in\cup \I_j^q}$ basis $\Lambda W$, \\ $(\nu_I\wedge\omega_J)$ basis $\Lambda(V\oplus W)$. \\[3pt]
		$\nu=\sum a_I\nu_I$, $\omega=\sum b_J\omega_J$, \\$\nu\wedge\omega=\sum a_I b_J \nu_I\wedge\omega_J$. \\[3pt]
		$\nu\wedge\omega=0$  $\Rightarrow$ $a_Ib_J=0 \ \forall I,J$ \\ 
		Some $a_I\neq 0$ $\Rightarrow$ $b_J=0\ \forall J$, e vice-versa.} 
\end{lemma}

\begin{lemma}\label{pr:linear combination blades}
Let $\mu,\nu,\omega\in\Lambda^r X$ be blades representing distinct subspaces $U, V,W\subset X$, respectively, and $A=U\cap V\cap W$. 
If $\nu=a\mu+b\omega$ for nonzero $a,b\in\C$ then $\dim A=r-1$ and there are $u\in U$, $v\in V$, $w \in W$ and an unit blade $\xi\in\Lambda^{r-1} A$ such that $\mu=u\wedge\xi$, $\nu=v\wedge\xi$, $\omega=w\wedge\xi$ and $v=au+bw$. 
Moreover, $u$, $v$ and $w$ can be chosen to be in any given complement of $A$ in $X$. If they are in $A^\perp$ then $\inner{u,w}=\inner{\mu,\omega}$.
\end{lemma}
\begin{proof}
If $x\in U\cap W$ then $x\wedge \nu=x\wedge(a\mu+b\omega)=0$, so $x\in V$. As $\mu$ and $\omega$ are also linear combinations of the other blades, 
$A=U\cap W=U\cap V=V\cap W$. 
Let $s=\dim A<r$, and $X'$ be a complement of $A$ in $X$. Then $U'=U\cap X'$, $V'=V\cap X'$ and $W'=W\cap X'$ are disjoint $(r-s)$-dimensional complements of $A$ in $U,V$ and $W$. 
\SELF{Let $U''=$complem $A$ in $U$. $u\in U \Rightarrow u=a_1+u''$. Mas $u''=a_2+x'$, so $u=a_3+x'$. $x'=u''-a_2\in U \Rightarrow x'\in U'$.}

Given an unit $\xi\in\Lambda^s  A$, we get $\mu=\mu'\wedge\xi$, $\nu=\nu'\wedge\xi$ and $\omega=\omega'\wedge\xi$ for blades $\mu'\in\Lambda^{r-s}U'$, $\nu'\in\Lambda^{r-s}V'$ and $\omega'\in\Lambda^{r-s}W'$.
As $\nu'-a\mu'-b\omega'\in\Lambda X'$ and $(\nu'-a\mu'-b\omega')\wedge\xi=0$,
\SELF{$\nu=a\mu+b\omega$}
\cref{pr:wedge disjoint multivectors} gives $\nu'=a\mu'+b\omega'$.\SELF{Precisa do Lemma pois ainda não sabe se \\ $\nu'-a\mu'-b\omega'$ é blade.}

For any nonzero vectors $u'\in U'$ and $w'\in W'$ we have $\nu'\wedge u'\wedge w' = (a\mu'+b\omega')\wedge u'\wedge w' = 0$.
Since $U'$ and $V'$ are disjoint, $\nu'\wedge u'\neq 0$, thus $w'\in V'\oplus\Span(u')$. 
As $w'\not\in V'$ and $u'$ was arbitrary, this implies 
$\dim U'=1$, 
\SELF{If $w=v_1+c_1 u_1=v_2+c_2u_2$ for some $u_1,u_2\in U'$, $v_1,v_2\in V'$ and $c_1,c_2\neq 0$, then $c_1u_1-c_2u_2=v_2-v_1\in U'\cap V'$, so $c_1u_1=c_2u_2$.}\
so $s=r-1$.
Thus $\mu',\nu'$ and $\omega'$ are vectors $u\in U', v\in V'$ and $w\in W'$, respectively, with $v=au+bw$.

If $X'=A^\perp$, we have $\inner{\mu,\omega} = \inner{u\wedge\xi,w\wedge\xi} = \inner{u,w}\cdot\|\xi\|^2$.
\end{proof}

\begin{proposition}\label{pr:triangle equality}
Given subspaces $U,V,W\subset X$, 
\begin{equation}\label{eq:equality}
\Theta_{U,W} = \Theta_{U,V} + \Theta_{V,W}
\end{equation}
if, and only if, one of the following conditions is satisfied:
\begin{enumerate}[i)]
\item $U\subset V$ and either $U\pperp W$ or $U^\perp\cap V\subset W$; \label{it:UV}
\item $V\subset W$ and either $U\pperp W$ or $P_W(U)\subset V$; \label{it:VW} 
\item There are nonzero $u, w\in X$ with $\inner{u,w}\geq 0$, $v=au+bw$ with $a,b>0$,  and subspaces $A, B, C\subset X$ orthogonal to $\Span(u,w)$ and to each other, such that \label{it:UVW ABC}
\begin{align*}
U &= \Span(u)\oplus A, \\
V &= \Span(v)\oplus A\oplus B, \\
W&= \Span(w)\oplus A\oplus B\oplus C.
\end{align*}
\end{enumerate}
Moreover, in this last case 
$\theta_{u,v}=\Theta_{U,V}$, $\theta_{u,w}=\Theta_{U,W}$ and $\theta_{v,w}=\Theta_{V,W}$.
\end{proposition}
\begin{proof}
(\!\emph{\ref{it:UV}}) and (\!\emph{\ref{it:VW}}) correspond, by propositions \ref{pr:V' sub V} and \ref{pr:W' sub W}, to when $\Theta_{U,V}=0$ or $\Theta_{V,W}=0$, and the other two angles are equal.

If \eqref{eq:equality} holds, but (\!\emph{\ref{it:UV}}) and (\!\emph{\ref{it:VW}}) do not, then $\Theta_{U,W}, \Theta_{U,V}, \Theta_{V,W}\neq 0$ and $\Theta_{U,V}, \Theta_{V,W}\neq \frac{\pi}{2}$, so $\Theta_{P_V(U),W}\neq \frac{\pi}{2}$.
As $\Theta_{U,W} \leq \Theta_{U,P_V(U)} + \Theta_{P_V(U),W} \leq \Theta_{U,V} + \Theta_{V,W}$, we get 
\begin{equation}\label{eq:triangle equality PV(U)}
\Theta_{U,W} = \Theta_{U,P_V(U)} + \Theta_{P_V(U),W},
\end{equation} 
so that $\Theta_{U,W}> \Theta_{P_V(U),W}$, since $\Theta_{U,P_V(U)}=\Theta_{U,V}> 0$.
We also get $\Theta_{P_V(U),W} = \Theta_{V,W}$, which, by \cref{pr:V' sub V} and since $P_V(U)\not\pperp W$, gives  $P_V(U)^\perp\cap V \subset W$.
As $V\not\subset W$, this implies $\Theta_{P_V(U),W} \neq 0$.
\SELF{$P_V(U)\not\subset W$}

Let $\mu,\nu,\omega_\mu,\omega_\nu,\omega_\mu^\perp,\omega_\nu^\perp$ be the unit $r$-blades in the proof of \cref{pr:spherical triangle inequality}.
As \eqref{eq:triangle equality PV(U)} implies equality in \eqref{eq:inequality}, $\Theta_{U,W}\neq 0$, and  $\Theta_{P_V(U),W}\neq 0$ or $\frac{\pi}{2}$, we get $\omega_\mu=\omega_\nu$ (by definition, if $\Theta_{U,W}=\frac{\pi}{2}$) and $\omega_\mu^\perp = \omega_\nu^\perp$.
So \eqref{eq:nu} becomes 
\[ \nu = \omega_\mu\cdot\cos\Theta_{P_V(U),W} + \frac{\mu-P\mu}{\|\mu-P\mu\|}\cdot\sin\Theta_{P_V(U),W}. \]
As $P\mu=\omega_\mu\cdot\cos\Theta_{U,W}$ and $\|\mu-P\mu\|=\sin\Theta_{U,W}$,
\SELF{$\inner{\mu-P\mu,\mu-P\mu}$ \\ $= 1+\|P\mu\|^2-2Re\inner{\mu,P\mu}$ \\ $=1-\|P\mu\|^2 = 1-\cos^2\theta$}
we get
\[ \nu = \omega_\mu \cdot \frac{\sin(\Theta_{U,W}-\Theta_{P_V(U),W})}{\sin \Theta_{U,W}} +\mu \cdot \frac{\sin\Theta_{P_V(U),W}}{\sin \Theta_{U,W}}, \]
so that $\nu=a\mu+b\,\omega_\mu$ with $a,b> 0$.

Let $A=U\cap P_V(U)\cap K$, where $K\subset W$ is represented by $\omega_\mu$.
As $U\not\subset V$, $U\not\subset W$ and $P_V(U)\not\subset W$, the subspaces $U$, $P_V(U)$ and $K$ are distinct. 
\Cref{pr:linear combination blades}\SELF{Apllied to $\mu$, $\nu$ and $\omega_\mu$}\  gives nonzero vectors $u\in U\cap  A^\perp$, $v\in P_V(U)\cap  A^\perp$ and $w \in K\cap  A^\perp$ such that $v=au+bw$, $\inner{u,w} = \inner{\mu,\omega_\mu} \geq 0$, 
$U=\Span(u)\oplus A$, $P_V(U)=\Span(v)\oplus A$ and $K=\Span(w)\oplus A$.

Let $B=P_V(U)^\perp\cap V\subset W$.
\SELF{$\subset W$ proven above}
Then $V=P_V(U)\oplus B$, with $B$ orthogonal to  $A$, $v$, $u$, \SELF{$P_V(U)\perp B\subset V\Rightarrow U\perp B$}
$w$ (as $w\in\Span(u,v)$) and $K$.

Let $C=(K\oplus B)^\perp\cap W$. Then $W=K\oplus B\oplus C$, with $C$ orthogonal to $B$, $K$, $w$ and $A$. 
If $\Theta_{U,W}\neq \frac\pi 2$ then $K=P(U)$,
\SELF{as $K=Ann(P\mu)$}
and if $\Theta_{U,W}= \frac\pi 2$ then $u\perp W$, as $A\subset W$.
In either case, $C$ is orthogonal to $u$.
\SELF{$P(U)\perp C \subset W \Rightarrow C \perp U$} 

So (\!\emph{\ref{it:UVW ABC}}) is satisfied. Under its conditions, it is immediate, in the real case, that $\Theta_{U,W}=\theta_{u,w}$, $\Theta_{U,V}=\theta_{u,v}$, $\Theta_{V,W}=\theta_{v,w}$ and $\theta_{u,w}=\theta_{u,v}+\theta_{v,w}$. In the complex case, we must use $\inner{u,w}\geq 0$ to get $\Theta_{U,W}=\gamma_{u,w}=\theta_{u,w}$, and also $a,b>0$ for the other angles and \eqref{eq:equality}.
\end{proof}

When dimensions are equal the conditions become simpler.

\begin{corollary}\label{pr:triangle equality same dim}
For subspaces $U,V,W\subset X$ of dimension $p$,
\eqref{eq:equality} holds if, and only if, $V=U$ or $W$, or $\dim(U\cap V\cap W)=p-1$ and there are nonzero $u\in U$, $v\in V$, $w \in W$ in an isotropic\footnote{A real subspace $R\subset X$ is \emph{isotropic} if $\inner{u,w}\in \R$ for all $u,w\in R$.}\
real plane orthogonal to $U\cap V\cap W$, with $\R v$ in the smaller pair of angles formed by $\R u$ and $\R w$.
\end{corollary}

Note that, in the complex case, such real plane must be orthogonal with respect to the Hermitian product, not the underlying real product.\SELF{The real plane lies in a 2-dim complex subspace orthogonal to $U\cap V\cap W$}

\begin{example}
If $\alpha, \beta$ and $\gamma$ are the dihedral angles between the faces of a nondegenerate
\SELF{i.e. the 3 faces planes intercept only at the vertex.} 
trihedral angle (\cref{fig:triangle inequality2}) then $\min\{\alpha,180^\circ-\alpha\}< \beta+\gamma$.
This can also be obtained from the fact that the sum of the angles of a nondegenerate spherical triangle is strictly greater than $180^\circ$.
\SELF{$\alpha+\beta+\gamma>180$. Se $\alpha\leq 90$, $\beta+\gamma>90\geq\alpha=\min$. Se $\alpha>90$, $\min=180-\alpha<\beta+\gamma$}

\begin{figure}
\centering
\includegraphics[width=0.5\linewidth]{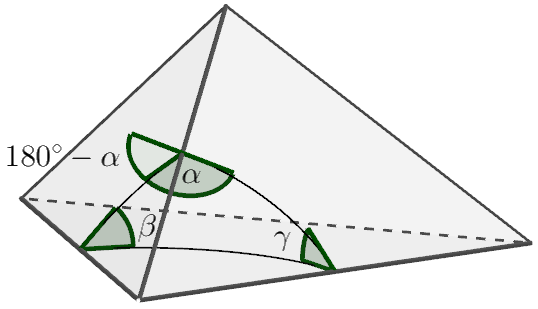}
\caption{$\min\{\alpha,180^\circ-\alpha\}< \beta+\gamma$.}
\label{fig:triangle inequality2}
\end{figure}
\end{example}

\subsection{Fubini-Study metric on Grassmannians}\label{sc:Fubini-Study distance}

The triangle inequality suggests $\Theta_{V,W}$ may give a metric. This is indeed the case in each Grassmannian $G_p(X)$, where, by \cref{pr:angle external powers}, it gives the Fubini-Study metric $d_{FS}$ (the real case appears in \cite{Edelman1999}).
\SELF{[p.337]}

\begin{theorem}\label{pr:Theta Fubini equal dim}
$d_{FS}(V,W)=\Theta_{V,W}$ for any $V,W\in G_p(X)$.
\end{theorem}

For $U,V,W\in G_p(X)$, \eqref{eq:equality} means that, in the \Plucker\  embedding, $V$ lies in a minimal geodesic
\SELF{Riemannian Geometry, do Carmo p.73 corol. 3.9 : shortest path (if exists) is always geodesic} 
of $\PP(\Lambda^p X)$ joining $U$ and $W$. 
So \cref{pr:triangle equality same dim} provides information about geodesics  connecting points  of $G_p(X)$.

\begin{proposition}\label{pr:geodesic Plucker}
Given distinct $U, W \in G_p(X)$, a minimal geodesic connecting them in $\PP(\Lambda^p X)$ intercepts $G_p(X)$ at another point if, and only if, $\dim(U\cap W)=p-1$. When this happens:
\begin{enumerate}[i)]
\item Any geodesic through $U$ and $W$ lies entirely in $G_p(X)$, and is given by $V(t)=(U\cap W)\oplus\Span(u\cos t + w\sin t)$, with $t\in[0,\pi)$, for some nonzero $u\in U$ and $w\in W$ such that $u,w\perp U\cap W$ and $\inner{u,w}\in\R$.
\item In the complex case, the complex projective line determined by $U$ and $W$ lies in $G_p(X)$, and its elements can be described by $V(t,\varphi)=(U\cap W)\oplus\Span(u\cos t+w e^{i\phi}\sin t)$, with $t\in[0,\pi)$ and $\phi\in[0,2\pi)$, for some nonzero $u\in U$ and $w\in W$ such that $u,w\perp U\cap W$.
\end{enumerate}
\end{proposition}


So, either $G_p(X)$ contains the whole geodesic through two of its points (what only happens for special pairs of points if $1<p<n-1$, where $n=\dim X$),
\SELF{Lines or hyperplanes always have intersection of codim 1}
or does not contain any other point of such geodesic.
This corrects a result of \cite{Jiang1996} stating that if $n>3$ and $p>1$ then $G_p(X)$ is not metrically convex (i.e. it has points with no between-point).
The right condition is $1<p<n-1$, as $G_{n-1}(X) = \PP(\Lambda^{n-1} X)$. 


%
%
%

\begin{example}
	To see why $u$ and $w$ must be in an isotropic plane, let $U,W,V\subset X=\C^2$ be spanned by $u=(1,0)$, $w=(1,\sqrt{3})/2$ and $v=u+w$, respectively. Then $\gamma_{u,w}=60^\circ$ and $\gamma_{u,v}=\gamma_{v,w}=30^\circ$, so $V$ lies in the geodesic segment between $U$ and $W$ in $\PP(X)=\C\PP^1$.
	
	Taking  $u=(i,0)$ instead, we get $\gamma_{u,w}=60^\circ$ and $\gamma_{u,v}=\gamma_{v,w}\cong 38^\circ$, so $U$, $V$ and $W$ are not in the same geodesic of $\PP(X)$ anymore. However, $\theta_{u,w}=90^\circ$ and $\theta_{u,v}=\theta_{v,w}=45^\circ$, so $\Span_\R(v)$ does lie in the geodesic segment between $\Span_\R(u)$ and $\Span_\R(w)$ in $\PP(X_\R)=\R\PP^3$.
\end{example}

\subsection{Asymmetric metric on the full Grassmannian}\label{sc:Pseudometric}

In the full Grassmannian $G(X)$, $\Theta_{V,W}$ falls short of being a metric, as for different dimensions it lacks symmetry and the identity of indiscernibles ($\Theta_{V,W}=0 \not\Rightarrow V=W$). 
But it gives a weaker kind of metric.

\begin{definition}
An \emph{asymmetric metric}  on a non-empty set $M$ is a function $d:X\times X\rightarrow [0,\infty)$ such that, for all $x,y,z\in M$,
\begin{enumerate}[i)]
\item $d(x,y)=d(y,x)=0 \Leftrightarrow x=y$; 
\item $d(x,z) \leq d(x,y)+d(y,z)$.
\end{enumerate}
\end{definition}

We adopt the terminology of \cite{Anguelov2016,Mennucci2014}. Other terms are quasi-metric \cite{Albert1941,Cobzas2012} or $T_0$-quasi-pseudometric \cite{Kazeem2014}.
Some authors require the stronger condition $d(x,y)=0 \Leftrightarrow x=y$.
\cite{Kunzi2001} has an extensive review of works on nonsymmetric metrics.

\begin{theorem}\label{pr:full Grassmannian}
$G(X)$ is an asymmetric metric space, with distances given by the Grassmann angle.
\end{theorem}
\begin{proof}
	Follows from propositions \ref{pr:properties Grassmann}\emph{\ref{it:Theta zero}} and \ref{pr:spherical triangle inequality}.
\end{proof}

If $d$ is an asymmetric metric then $D(x,y)=\max\{d(x,y),d(y,x)\}$ is a metric.
Indeed, the max-symmetrized angle gives the Fubini-Study metric in $G(X)$, extending \cref{pr:Theta Fubini equal dim}: 

\begin{theorem}\label{pr:Theta Fubini any dim}
	$d_{FS}(V,W)=\hat{\Theta}_{V,W}$ for any $V,W\in G(X)$.
\end{theorem}

Despite being useful in applications \cite{Gruber2009,Renard2018}, $G(X)$ has received little attention from geometers.
A reason is that $d_{FS}$ reduces its geometry to a union of $G_p(X)$'s separated by a distance of $\frac{\pi}{2}$.
Also, this metric does not describe so well the separation between subspaces of different dimensions. For example, the distance between a line $L$ and a plane $W$ is always $\frac{\pi}{2}$, even if $L\subset W$.

The asymmetric metric given by Grassmann angles does a better job, with $\Theta_{L,W}\rightarrow 0$ as $L$ gets closer to being contained in $W$. The fact that $\Theta_{W,L} = \frac{\pi}{2}$ does not change is also meaningful, as $W$ is never any closer to being contained in $L$.
And, despite not being an usual metric, this angle has many useful properties, so we believe it should prove well suited for applications of $G(X)$. 
\SELF{A topologia do novo espaço ainda é de disjoint union, mas a métrica não.}

\subsection{Hausdorff distances}\label{sc:Hausdorff distance}

Haussdorff distances are important similarity measures (used, like the Grassmannians, in computer vision and pattern matching \cite{Huttenlocher1993,Knauer2011}), and give another example of asymmetric metric. 

\begin{definition}
For non-empty\SELF{Não precisa se usar $\sup$ e $\inf$, e tomar cuidado com as particularidades de $\sup\emptyset$ e $\inf\emptyset$.}\ compact sets $S$ and $T$ in a metric space $(M,d)$, the \emph{Hausdorff distance} is $ H(S,T) = \max\left\{ h(S,T), h(T,S)\right\}$, where $\displaystyle h(S,T) = \max_{s \in S} d(s,T)$
\SELF{$= \max_{s \in S} \min_{t\in T} d(s,t)$} 
is the \emph{directed Hausdorff distance}
from $S$ to $T$. 
\end{definition}

Some authors call $h$ the (one-sided) Hausdoff distance, and $H$ the bidirectional or two-sided Hausdorff distance.
$h(S,T)$ is the largest distance from points in $S$ to their closest points in $T$, or the smallest $\epsilon$ such that a closed $\epsilon$-neighborhood of $T$ contains $S$, and each set is contained within distance $H(S,T)$ of the other (\cref{fig:hausdorff}).
$H$ is a metric in the set of non-empty compact subsets of $M$, and $h$ is an asymmetric metric.
\SELF{$h(R,T) \leq h(R,S)+h(S,T)$ for compact $R,S,T$. \\[3pt]
Fix $r\in R$. For some $s,t$ \\ $d(r,S)=d(r,s)$, \\ $d(s,T)=d(s,t)$. \\[3pt]
$d(r,T)\leq d(r,t)$ \\ $\leq d(r,s)+d(s,t) $ \\ $=d(r,S)+d(s,T)$ \\ $\leq d_h(R,S)+d_h(S,T)$. \\[3pt] Use $d_h(R,T)=\max\limits_r d(r,T)$
}
For $x\in M$ and non-empty compact sets $S,T\subset M$, $d(x,T) \leq d(x,S) + h(S,T)$.
\SELF{$h(\{x\},T)=d(x,T)$}

\begin{figure}
	\centering
	\includegraphics[width=0.35\linewidth]{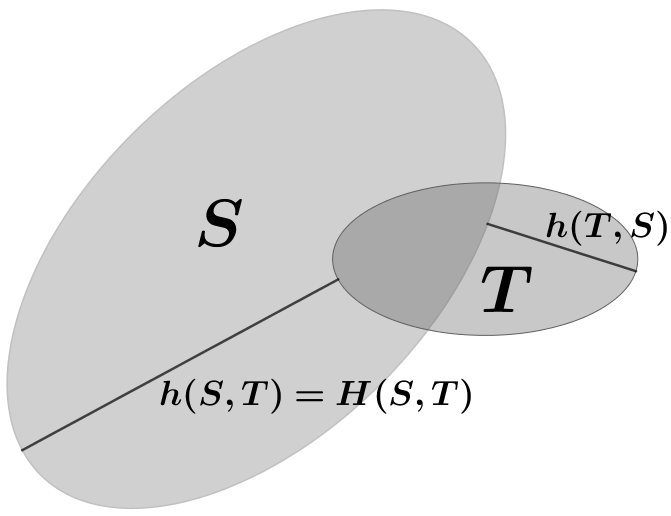}
	\caption{Haussdorff distances.}
	\label{fig:hausdorff}
\end{figure}

Grassmann angles give directed Haussdorff distances between full sub-Grassmannians.
For subspaces $V,W\subset X$, consider $G(V)$ and $G(W)$ embedded in $\PP(\Lambda X)$. Then:

\begin{proposition}\label{pr:Hausdorff Theta}
$h(G(V),G(W)) = \Theta_{V,W}$. 
\end{proposition}
\begin{proof}
For $V'\in G(V)$ and $W'\in G(W)$, propositions \ref{pr:W' sub W} and \ref{pr:Theta Fubini any dim} give $d_{FS}(V',W') = \hat{\Theta}_{V',W'} \geq  \Theta_{V',W'} \geq \Theta_{V',W}$, with equalities for $W'=P_W(V')$.
\SELF{$\in G_p(W)$ if $V\not\pperp W$, otherwise the distances are $\frac{\pi}{2}$}
So $d_{FS}(V',G(W)) = \Theta_{V',W}$,
\SELF{$ = d_{FS}(V,G_p(W))$ if $p=\dim V\leq \dim W$}
and the result follows from \cref{pr:V' sub V}.
\SELF{Equality if $V'=V$}
\end{proof}

So, in $\PP(\Lambda X)$, $G(V)$ is contained within distance $\Theta_{V,W}$ from $G(W)$.
Also, Grassmann angles give an asymmetric metric in the set of full sub-Grassmannians of $G(X)$.
Of course, max-symmetrizing we get that $H(G(V),G(W)) = \hat{\Theta}_{V,W} = d_{FS}(V,W)$ gives a metric.

\begin{corollary}
$d_{FS}(p,G(W)) \leq d_{FS}(p,G(V)) + \Theta_{V,W}$, for  $p\in\PP(\Lambda X)$. 
\end{corollary}

For full sub-Grasmannians of oriented subspaces, embedded in the unit sphere $S(\Lambda X)$ (with distance $d$ given by the Euclidean angle), we also have:

\begin{proposition}\label{pr:Hausdorff Lambda}
$h(\tilde{G}(V),\tilde{G}(W)) =\Theta_{V,W}$.
\end{proposition}
\begin{proof}
Any $\nu\in \tilde{G}(V)$ represents a subspace $V'\subset V$, and $d(\nu,\tilde{G}(W)) = \min\limits_{\omega\in \tilde{G}(W)} \theta_{\nu,\omega} =  \theta_{\nu,\Proj_{W}\nu} =  \Theta_{V',W} \leq  \Theta_{V,W}$. 
\SELF{\ref{pr:angle external powers}, \ref{pr:V' sub V}}
\end{proof}

One might expect the result to be $\mathbf{\Theta}_{V,W}$, but consider \cref{fig:angulos vetores complexos} again: minimal distances occur when orientations align.

\begin{corollary}
$d(p,\tilde{G}(W)) \leq d(p,\tilde{G}(V)) + \Theta_{V,W}$ for any $p\in S(\Lambda X)$. 
\SELF{não precisa ser blade}
\end{corollary}


\section{Complementary Grassmann angle}\label{sc:Complementary Grassmann angle}

The Grassmann angle of a subspace with the orthogonal complement of another has special properties that grant it a new name and notation. 
Its importance has been overlooked for similar angles, possibly because the simplicity and generality of our results depend on the asymmetry.

\begin{definition}
The \emph{complementary Grassmann angle} $\Theta_{V,W}^\perp \in [0,\frac{\pi}{2}]$ of subspaces $V,W\subset X$ is $\Theta_{V,W}^\perp=\Theta_{V,W^\perp}$.
\end{definition}

The asymmetry gives results, like $\Theta_{V,W}^\perp=\frac{\pi}{2}$ for any two planes in $\R^3$, which seem wrong until we learn how to interpret this angle. But it also allows the following properties to hold without any restrictions.
\SELF{Sem assimetria, \ref{it:Theta perp pi 2} não vale p/planos em $\R^3$, nem \ref{it:Theta perp 0} se forem perpend, e \ref{it:complementary line} não vale se $W=0$ ou $X$
}


\begin{proposition}\label{pr:complementary simple cases}
Let $V,W\subset X$ be any subspaces.
\begin{enumerate}[i)]
\item $\Theta_{V,W}^\perp=0 \ \Leftrightarrow\ $ $V\perp W$.\label{it:Theta perp 0}
\item $\Theta_{V,W}^\perp=\frac{\pi}{2} \ \Leftrightarrow\ V\cap W\neq\{0\}$.\label{it:Theta perp pi 2}
\item $\Theta_{L,W}+\Theta_{L,W}^\perp=\frac{\pi}{2}$ for any line $L\subset X$. \label{it:complementary line}
\end{enumerate}
\end{proposition}

Though for a line (even a complex one) $\Theta_{L,W}^\perp$ is the usual angle complement, and $\cos\Theta_{L,W}^\perp=\sin\Theta_{L,W}$, \cref{ex:Grassmann angle complement} shows this is not always valid. In general, $\cos\Theta_{V,W}^\perp$ will be a product of sines of principal angles. We prove this first by getting principal bases and angles for $V$ and $W^\perp$ (both nonzero), as this construction can be instructive.

Let $p=\dim V$, $q=\dim W$, $n=\dim X$, $r=\dim V\cap W$, $m=\min\{p,q\}$, $m'=\min\{p,n-q\}$ and $P^\perp=\Proj_{W^\perp}$.
From principal bases $(e_1,\ldots,e_p)$ of $V$ and $(f_1,\ldots,f_q)$ of $W$, with principal angles $\theta_1\leq\ldots\leq\theta_m$, we get principal bases $(\tilde{e}_1,\ldots,\tilde{e}_p)$ of $V$ and $(g_1,\ldots,g_{n-q})$ of $W^\perp$, with principal angles $\theta_1^\perp,\ldots,\theta_{m'}^\perp$, as follows: 
\begin{enumerate}[1)]
\item The $\tilde{e}_i$'s are the same as the $e_i$'s, in reverse order: $\tilde{e}_{p+1-i}=e_i$.  \label{it:step 1}
\item For $m<i\leq p$ (if any), let $g_{p+1-i}=e_i\in W^\perp$, so $\theta^\perp_{p+1-i}=0$.  \label{it:step 2}
\item For $r<i\leq m$,  let $g_{p+1-i}=\frac{P^\perp e_i}{\|P^\perp e_i\|}$, so $\theta^\perp_{p+1-i} =\frac{\pi}{2}-\theta_i$. \label{it:step 3}
\SELF{$= \theta_{P^\perp e_i,e_i}$}
\item If $p-r<n-q$,
\SELF{So far we have $p-r$ principal vectors for $W^\perp$} 
add new $g$'s to form an orthonormal basis for $W^\perp$. 
Any $e_i$ with $i\leq r$ is orthogonal to $W^\perp$, so pairing as many of them as possible with the new $g$'s we get principal angles $\theta^\perp_{p+1-i}=\frac{\pi}{2}$. \label{it:step 4}
\end{enumerate}
Since $g_{p+1-i}\in\Span(e_i,f_i)$ in step \ref{it:step 3}, one can check that $\inner{g_i,g_j} = \delta_{ij}$ and $\inner{\tilde{e}_i,g_j} = \delta_{ij}\cos\theta^\perp_i$, so that \eqref{eq:inner ei fj} is satisfied.

\begin{example}\label{ex:complementary}
In $\R^5$, the following are principal vectors and angles for $V=\Span(e_1,e_2,e_3)$ and $W=\Span(f_1,f_2)$:
\begin{equation*}
\begin{aligned}
e_1&=(1,0,0,0,0), &\qquad f_1&=(1,0,0,0,0),  &\qquad \theta_1&=0,\\
e_2&=(0,1,0,0,0), & f_2&=(0,\sqrt{3},1,0,0)/2, & \theta_2&=30^\circ, \\
e_3&=(0,0,0,1,0), 
\end{aligned}
\end{equation*}
and  $\Theta_{V,W}=90^\circ$, as $\dim V>\dim W$, while $\Theta_{W,V}=30^\circ$. Applying the procedure described above, we get 
\begin{equation*}
\begin{aligned}
\tilde{e}_1&=(0,0,0,1,0), &\qquad g_1&=(0,0,0,1,0), &\qquad \theta_1^\perp&=0, \\
\tilde{e}_2&=(0,1,0,0,0), & g_2&=(0,1,-\sqrt{3},0,0)/2, & \theta_2^\perp&=60^\circ,\\
\tilde{e}_3&=(1,0,0,0,0),  & g_3&=(0,0,0,0,1), & \theta_3^\perp&=90^\circ,
\end{aligned}
\end{equation*}
as principal vectors and angles for $V$ and $W^\perp$, so $\Theta_{V,W}^\perp= \Theta_{W,V}^\perp= 90^\circ$.
\end{example}

We can now compute $\Theta^\perp_{V,W}$ in terms of principal angles of $V$ and $W$, providing an interpretation for the product of their sines, studied in \cite{Afriat1957,Miao1992} but never linked to a particular angle.

\begin{theorem}\label{pr:complementary product sines}
If $V,W\subset X$ are nonzero subspaces, with principal angles $\theta_1,\ldots,\theta_m$, then
\begin{equation}\label{eq:complementary product sines}
\cos\Theta_{V,W}^\perp=\prod_{i=1}^m \sin\theta_i.
\end{equation}
\end{theorem}
\begin{proof}
By \cref{pr:complementary simple cases}\emph{\ref{it:Theta perp pi 2}}, $\Theta_{V,W}^\perp=\frac{\pi}{2} \Leftrightarrow \theta_1=0$.
If $\Theta_{V,W^\perp}\neq \frac{\pi}{2}$ then $p=\dim V\leq\dim W^\perp$, and as $\theta_1\neq 0$ no pair is formed in step \ref{it:step 4} above. Then $\cos\Theta_{V,W^\perp}=\prod_{j=1}^{p}\cos\theta^\perp_j =\prod_{i=1}^{p}\cos\theta^\perp_{p+1-i}$, with the $\theta^\perp$'s  as above, and each $\cos\theta^\perp_{p+1-i}$ is either $1$ (step \ref{it:step 2}, if $i>m$) or $\sin \theta_i$ (step \ref{it:step 3}).
\end{proof}

In \cite{Mandolesi_Products} we give an easier proof using the following formula for the exterior product of blades, obtained by projecting one blade on the orthogonal complement of the other and using \cref{pr:Theta Pnu}.

\begin{proposition}\label{pr:exterior product} 
	$\|\nu\wedge\omega\|=\|\nu\|\|\omega\| \cos\Theta^\perp_{V,W}$ for any blades $\nu,\omega\in\Lambda X$ representing subspaces $V,W\subset X$.
\end{proposition}

This formula and \eqref{eq:complementary product sines} only hold unconditionally due to the Grassmann angle asymmetry, without which we would have to require $V\cap W=\{0\}$.
On the other hand, the formula shows $\Theta_{V,W}^\perp$ is symmetric, what can also be obtained from \eqref{eq:complementary product sines}, as principal angles do not depend on the order of $V$ and $W$ (and $\Theta^\perp_{V,W} =  0$ if either space is $\{0\}$).

\begin{corollary}\label{pr:symmetry}
Let $V,W\subset X$ be  any subspaces.
\begin{enumerate}[i)]
	\item $\Theta_{V,W}^\perp = \Theta_{W,V}^\perp$. \label{it:symmetry complementary}
	\item $\Theta_{V,W} = \Theta_{W^\perp,V^\perp}$. \label{it:Theta perp perp}
\end{enumerate}
\end{corollary}

Symmetry \emph{(\ref{it:symmetry complementary})} is a little surprising, for if $\dim V>\dim W$ these angles describe projections of volumes of distinct dimensions.
The explanation is simple: extra dimensions of $V$, already in $W^\perp$, do not affect  the projection.
\emph{(\ref{it:Theta perp perp})} reflects the fact that nonzero principal angles of $V^\perp$ and $W^\perp$ are the same as those of $V$ and $W$ \cite{Gluck1967}.
Note the reversal of $V^\perp$ and $W^\perp$.

The angle $\Theta_{V^\perp,W}$ is a little different from $\Theta_{V,W^\perp}$: its cosine is a product of sines of \emph{nonzero} principal angles, and only if $V+W$ spans $X$. 

\begin{proposition}
If $V,W\subset X$ are nonzero subspaces, with principal angles $\theta_1\leq\ldots\leq\theta_m$, and $\dim V\cap W=r$, then\footnote{If $r=m$ the product of sines is absent, so we have only $\arccos 1$.}
\begin{equation*}
\Theta_{V^\perp,W} = \Theta_{W^\perp,V} = \begin{cases}
	\arccos(1\cdot \prod_{i=r+1}^m \sin\theta_i) \ \text{ if } V+W=X, \\
	\frac{\pi}{2} \hspace{92pt} \text{ if } V+W\neq X.
\end{cases}
\end{equation*}
\end{proposition}
\begin{proof}
$\Theta_{V^\perp,W} = \Theta_{W^\perp,V}$ by \cref{pr:symmetry}\emph{\ref{it:Theta perp perp}}, and we can assume $W\neq X$.
\SELF{Senão $\Theta_{V^\perp,W} =0$ e $r=m$. Precisa para $W^\perp\neq 0$ ter principal values}
If $V+W\neq X$ then $W^\perp\pperp V$.
\SELF{$W^\perp\cap V^\perp\neq\{0\}$}
If $V+W=X$ then $\dim W^\perp\leq\dim V$, and there is no step \ref{it:step 4} above, as $p+q-r=n$.
Any $\theta^\perp$ from step \ref{it:step 2} gives $\cos\theta^\perp=1$.
And step \ref{it:step 3} gives a $\sin\theta_i$ for each $r+1\leq i\leq m$.
\end{proof}

\begin{corollary}
$\Theta_{V^\perp,W} = \Theta_{V,W^\perp}$ if, and only if, $V\oplus W=X$ (in which case they are not $\frac{\pi}{2}$), or $V+W\neq X$ and $V\cap W\neq\{0\}$ (they are $\frac{\pi}{2}$).
\end{corollary}

\Cref{fig:projection-symmetries} exhibits the symmetries of the Grassmann angle.

\begin{figure}
	\centering
	\includegraphics[width=0.3\linewidth]{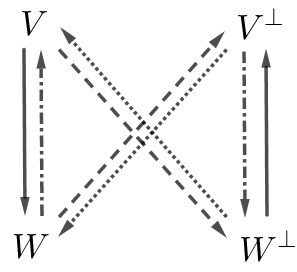}
	\caption{Symmetries. Pairs of same kind arrows are projections with equal Grassmann angles and projection factors $\pi$. Vertical pairs become equal if $\dim V=\dim W$, otherwise a pair has $\pi=0$, both if a principal angle is $\frac{\pi}{2}$. Diagonal pairs are equal if $V\oplus W=X$, otherwise a pair has $\pi=0$, both if $V+W\neq X$ and $V\cap W\neq\{0\}$.}
	\label{fig:projection-symmetries}
\end{figure}

\begin{example}
In \cref{ex:complementary}, $V+W\neq\R^5$, so $\Theta_{W^\perp,V} = 90^\circ$.
Or we can get $\Theta_{W^\perp,V}$ from principal angles of $V$ and $W^\perp$, and deduce $V+W\neq\R^5$.
\end{example}

\begin{example}
	In \cref{ex:complementary} again, let $U=\Span(f_1,f_2,f_3)$, where $f_3=(1,0,0,0,1)/\sqrt{2}$.
	$V$ and $U$ have principal angles $(0,30^\circ,45^\circ)$, while $V$ and $U^\perp$ have only $(45^\circ,60^\circ)$, as $\dim U^\perp=2$.
	Then:
	\begin{itemize}
	\item $\Theta_{V,U} = \Theta_{U,V} = \Theta_{V^\perp,U^\perp} = \Theta_{U^\perp,V^\perp} = \arccos (1\cdot \frac{\sqrt{3}}{2}\cdot \frac{\sqrt{2}}{2}) \cong 52.2^\circ$, so volumes shrink by a factor $\frac{\sqrt{6}}{4}$ when orthogonally projected between $V$ and $U$, as do areas between $V^\perp$ and $U^\perp$.
	\item $\Theta_{V,U}^\perp = \Theta_{U,V}^\perp  = \arccos (0\cdot \frac{1}{2}\cdot \frac{\sqrt{2}}{2}) = 90^\circ$, so volumes vanish when orthogonally projected from $V$ to $U^\perp$, or from $U$ to $V^\perp$, as $V\cap U\neq\{0\}$. Another reason is that $U^\perp$ and $V^\perp$ have smaller dimensions, but the result would be the same for complements in $\R^6$.
	\item As $V+U=\R^5$, $\Theta_{V^\perp,U} = \Theta_{U^\perp,V} = \arccos (\frac{1}{2}\cdot \frac{\sqrt{2}}{2}) \cong 69.3^\circ$, so areas shrink by a factor $\frac{\sqrt{2}}{4}$ when orthogonally projected from $V^\perp$ to $U$, or from $U^\perp$ to $V$. We can also get this from the principal angles of $V$ and $U^\perp$, and then conclude that $V+U=\R^5$.
	\end{itemize} 
\end{example}


Complementary Grassmann angles have properties resembling some from \cref{pr:properties Grassmann}, but with a different meaning.

\begin{proposition}
Let $U,V,W,Y\subset X$ be subspaces, and $P=\Proj_W$.
\begin{enumerate}[i)]
\item $\Theta^\perp_{V,W} = \Theta^\perp_{V,P(V)}$.\label{it:complem P(V)}
\item If $U\perp V+W$ then $\Theta^\perp_{V,W}=\Theta^\perp_{V,W\oplus U}$.\label{it:complementary W+U}
\item If $V,W\subset Y$ then $\Theta^\perp_{V,W}$  is the same whether the complement of $W$ is taken in $Y$ or $X$.\label{it:complementary in U}
\end{enumerate}
\end{proposition}
\begin{proof}
(\!\emph{\ref{it:complem P(V)}}) Follows from \eqref{eq:complementary product sines} and \cref{pr:P(V)}, if $V\not\perp W$, otherwise both angles are $0$.  
(\!\emph{\ref{it:complementary W+U}})  $\Proj_{W\oplus U}(V) = P(V)$.  
(\!\emph{\ref{it:complementary in U}}) \Cref{pr:properties Grassmann}\emph{\ref{it:Theta W+U}} with $U=W^\perp\cap Y^\perp$ gives $\Theta_{V,W^\perp\cap Y} = \Theta_{V,W^\perp}$.
\SELF{$U \perp V+(W^\perp\cap Y)$ and $W^\perp=(W^\perp\cap Y)\oplus U$}
\end{proof}

(\!\emph{\ref{it:complem P(V)}}) is not the same as \cref{pr:properties Grassmann}\emph{\ref{it:Theta P(V)}}, which gives $\Theta^\perp_{V,W} = \Theta_{V,P^\perp(V)}$ for $P^\perp=\Proj_{W^\perp}$. Combining them we get $\Theta_{V,P^\perp(V)} = \Theta_{V,P(V)^\perp}$. 
Likewise, (\!\emph{\ref{it:complementary W+U}}) is not the same as \cref{pr:properties Grassmann}\emph{\ref{it:Theta W+U}}.

In the complex case, $(W_\R)^\perp = (W^\perp)_\R$
\SELF{$x\perp_\R W \Rightarrow$\\ $\operatorname{Re}\inner{x,w}=0 \ \forall w \Rightarrow$\\ $\operatorname{Im}\inner{x,w}\!=\!-\operatorname{Re}\inner{x,i w}\!=\!0$ \\ $\Rightarrow \ x\perp_\C W$}
for any subspace $W\subset X$, even though the first is a $\R$-orthogonal complement and the second is $\C$-orthogonal.
So \cref{pr:cos underlying real Theta} gives:

\begin{proposition}\label{pr:cos underlying real Theta perp}
For complex subspaces $V,W\subset X$:
\begin{enumerate}[i)]
\item $\cos\Theta^\perp_{V_\R,W_\R} = \cos^2 \Theta^\perp_{V,W}$.
\item $\Theta^\perp_{V_\R,W_\R} \geq \Theta^\perp_{V,W}$, with equality if, and only if, $V\perp W$ or $V\cap W\neq\{0\}$.
\end{enumerate}
\end{proposition}

\begin{example}
In \cref{ex:complex principal angles}, $\Theta^\perp_{V,W}=\arccos(\frac{\sqrt{2}}{2}\cdot\frac{\sqrt{3}}{2})\cong 52.2^\circ$ and $\Theta^\perp_{V_\R,W_\R}=\arccos(\frac{\sqrt{2}}{2}\cdot\frac{\sqrt{2}}{2}\cdot\frac{\sqrt{3}}{2}\cdot\frac{\sqrt{3}}{2}) \cong 68^\circ$.
So $4$-dimensional volumes in $V$ (resp. $W$) contract by $\frac{3}{8}$ when orthogonally projected on $W^\perp$ (resp. $V^\perp$).
\end{example}

\subsection{Relation between $\Theta_{V,W}$ and $\Theta^\perp_{V,W}$}\label{sc:Theta Theta perp}

The relation between $\Theta_{V,W}$ and $\Theta^\perp_{V,W}$ is more flexible than the usual one between an angle and its complement, but they are not totally independent, and there are restrictions on the values they can assume.

\begin{proposition}\label{pr:zeta}
Let $V,W\subset X$ be subspaces, with $V\neq\{0\}$. Then:
\begin{enumerate}[i)]
\item $0\leq \cos^2\Theta_{V,W} + \cos^2\Theta_{V,W}^\perp \leq 1$. \label{it:range zeta}
\item $\frac{\pi}{2} \leq \Theta_{V,W} + \Theta_{V,W}^\perp \leq \pi$; \label{it:range Theta}
\item $\Theta_{V,W} + \Theta_{V,W}^\perp = \frac{\pi}{2}\ \Leftrightarrow\ \dim V=1$, or $V\perp W$, or $V\subset W$. \label{it:zeta 1}
\item $\Theta_{V,W} + \Theta_{V,W}^\perp = \pi\ \Leftrightarrow\ V\cap W\neq\{0\}$ and $V\cap W^\perp\neq\{0\}$. \label{it:zeta 0}
\end{enumerate}
\end{proposition}
\begin{proof}
Let $\xi_{V,W} = \cos^2\Theta_{V,W} + \cos^2\Theta_{V,W}^\perp$.
\emph{(\ref{it:range zeta})} If $W=\{0\}$ then $\xi_{V,W}=1$, and otherwise $V$ and $W$ have principal angles $\theta_1\leq\ldots\leq\theta_m$  and $\xi_{V,W} \leq \prod_{i=1}^m \cos^2\theta_i + \prod_{i=1}^m \sin^2\theta_i \leq \cos^2\theta_1 + \sin^2\theta_1  =1$. 
\emph{(\ref{it:range Theta})} If $\Theta_{V,W}+\Theta_{V,W}^\perp < \frac{\pi}{2}$ then $\cos\Theta_{V,W}^\perp > \sin \Theta_{V,W}$ and $\xi_{V,W} > \cos^2\Theta_{V,W} + \sin^2\Theta_{V,W} = 1$.
\emph{(\ref{it:zeta 1})} $\Theta_{V,W}^\perp = \frac{\pi}{2} - \Theta_{V,W} \Leftrightarrow \xi_{V,W}=1$, and the inequalities used to prove \emph{\ref{it:range zeta}} are equalities if, and only if, all $\theta_i$'s are $\frac{\pi}{2}$, or $\dim V\leq \dim W$ and either $m=1$ or all $\theta_i$'s are 0. 
\emph{(\ref{it:zeta 0})} $\Theta_{V,W}=\Theta_{V,W}^\perp=\frac{\pi}{2} \Leftrightarrow V\pperp W$ and $V\cap W\neq\{0\}$.\SELF{Usa \ref{pr:properties Grassmann}\ref{it:Theta pi2} and \ref{pr:complementary simple cases}\ref{it:Theta perp pi 2}} 
\end{proof}

Note that all possibilities happen even in $\R^3$, as planes $V$ and $W$ can form any angle $0\leq\Theta_{V,W}\leq\frac{\pi}{2}$, and $\Theta_{V,W}^\perp=\frac{\pi}{2}$.
If $V=\{0\}$ then $\Theta_{\{0\},W} = \Theta_{\{0\},W}^\perp = 0$ and $\cos^2\Theta_{\{0\},W} + \cos^2\Theta_{\{0\},W}^\perp =2$.

A reason why we can have $\cos^2\Theta_{V,W} + \cos^2\Theta_{V,W}^\perp<1$ is that it equals $\|P\nu\|^2 + \|P^\perp\nu\|^2$, where $\nu$ is an unit blade representing $V$, $P=\Proj_W$ and $P^\perp=\Proj_{W^\perp}$.
If the conditions in \emph{(\ref{it:zeta 1})} are not satisfied, $\nu$ written in terms of bases of $W$ and $W^\perp$ will have mixed components, which are neither in $P\nu$ nor in $P^\perp\nu$, so that $\|P\nu\|^2 + \|P^\perp\nu\|^2<\|\nu\|^2$.
In \cref{sc:Pythagorean trigonometric} we show which squared cosines to add up to get 1.

\begin{figure}
	\centering
	\includegraphics[width=0.7\linewidth]{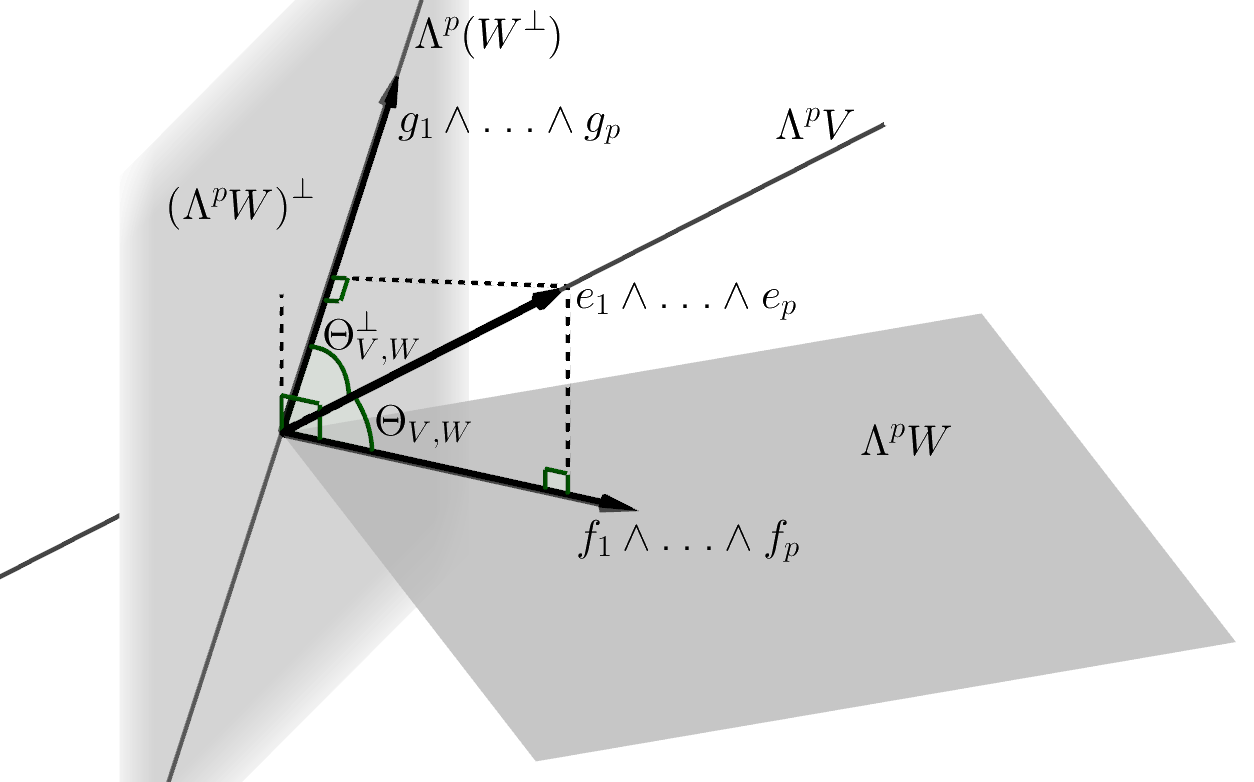}
	\caption{$\Theta_{V,W}$ and $\Theta_{V,W}^\perp$ as angles in $\Lambda^p X$ ($p=\dim V$).}
	\label{fig:complementary-total-angle}
\end{figure}

\Cref{fig:complementary-total-angle}, representing $\Theta_{V,W}$ and $\Theta_{V,W}^\perp$ in $\Lambda^p X$ ($p=\dim V$), may help understand their relation.
For simplicity, it shows $\Lambda^p(W^\perp)$ as a line spanned by $g_1\wedge\ldots\wedge g_p$, where the $g$'s form a principal basis of $W^\perp$, but we can have $\dim \Lambda^p(W^\perp)>1$ or  $\Lambda^p(W^\perp)=\{0\}$.

Results might look more natural in terms of $\Phi_{V,W} = \frac{\pi}{2}-\Theta_{V,W}^\perp$. For example, $\Phi_{V,W}=0 \Leftrightarrow V\cap W\neq\{0\}$, $\Phi_{V,W}=\frac{\pi}{2} \Leftrightarrow V\perp W$, and $\sin\Phi_{V,W} = \prod_{i=1}^m \sin\theta_i$.
But the geometric interpretation of  $\Phi_{V,W}$ in $\Lambda^p X$ would not be as natural as that of $\Theta_{V,W}^\perp$ in the figure.

With some information on the minimal and directed maximal angles we can further restrict the admissible values of $\Theta_{V,W}$ and $\Theta_{V,W}^\perp$.

\begin{definition}
Given nonzero subspaces $V,W\subset X$, the \emph{angular range} of $V$ w.r.t. $W$ is the length $\Delta\theta_{V,W}= \theta^{\max}_{V,W} - \theta^{\min}_{V,W} \in [0,\frac{\pi}{2}]$ of the interval $\left[\theta^{\min}_{V,W},\,\theta^{\max}_{V,W}\right]$ of possible angles between a nonzero $v\in V$ and $W$. 
\SELF{$\Delta\theta_{V,W}\neq\Delta\theta_{W,V}$.}
\end{definition}

\begin{proposition}\label{pr:cos+cos}
	Let $V,W\subset X$ be nonzero subspaces, and $p=\dim V$. 
	\begin{enumerate}[i)]
		\item If $p=1$ then \label{it:i}
		$\begin{cases}
			\Theta_{V,W} + \Theta_{V,W}^\perp = \frac{\pi}{2}, \\
			\cos\Theta_{V,W} + \cos\Theta_{V,W}^\perp \geq 1.
		\end{cases}$
		\item If $p=2$ then \label{it:ii}
		$\begin{cases}
			\Theta_{V,W} + \Theta_{V,W}^\perp \geq \frac{\pi}{2} + \Delta\theta_{V,W}, \\ 
			\cos\Theta_{V,W} + \cos\Theta_{V,W}^\perp = \cos \Delta\theta_{V,W}. 
		\end{cases}$
		\item If $p>2$ then \label{it:iii}
		$\begin{cases}
			\Theta_{V,W} + \Theta_{V,W}^\perp \geq \frac{\pi}{2} + \Delta\theta_{V,W}, \\
			\cos\Theta_{V,W} + \cos\Theta_{V,W}^\perp \leq \cos \Delta\theta_{V,W}.
		\end{cases}$
	\end{enumerate} 
	Equalities in {(\ref{it:i})} or {(\ref{it:ii})} happen if, and only if, $\Theta_{V,W}=\frac{\pi}{2}$ or $\Theta^\perp_{V,W}=\frac{\pi}{2}$.
	In {(\ref{it:iii})}, if and only if one of the following conditions is satisfied:
	\begin{enumerate}[A)]
		\item $\dim V\cap W^\perp\geq p-1$, in which case $\Theta_{V,W}=\frac{\pi}{2}$. \label{it:A}
		\item $\dim V\cap W\geq p-1$, in which case $\Theta_{V,W}^\perp=\frac{\pi}{2}$. \label{it:B}
		\item $\Delta\theta_{V,W}=\frac{\pi}{2}$, which happens if, and only if, $\Theta_{V,W}=\Theta_{V,W}^\perp=\frac{\pi}{2}$. \label{it:C}
	\end{enumerate}
\end{proposition}
\begin{proof}
Let $\theta_1\leq\ldots\leq\theta_m$ be all principal angles. Then $\theta^{\min}_{V,W} = \theta_1$.

If $p=1$ then $\cos\Theta_{V,W} + \cos\Theta_{V,W}^\perp = \cos\theta_1+\sin\theta_1\geq 1$, with equality if, and only if, $\theta_1= 0$ or $\frac{\pi}{2}$.

If $2\leq p=m\leq\dim W$ then $\theta^{\max}_{V,W}=\theta_m$ and $\cos\Theta_{V,W} + \cos\Theta_{V,W}^\perp =\prod_{i=1}^m \cos\theta_i + \prod_{i=1}^m \sin\theta_i \leq \cos\theta_1\cos\theta_m+\sin\theta_1\sin\theta_m = \cos(\theta_m-\theta_1)$.
Equality happens if, and only if, $m=2$, or \emph{(\ref{it:A})} $\theta_2=\ldots=\theta_m=\frac{\pi}{2}$, or \emph{(\ref{it:B})} $\theta_1=\ldots=\theta_{m-1}=0$, or \emph{(\ref{it:C})} $\theta_1=0$ and $\theta_m=\frac{\pi}{2}$.

If $p>\dim W=m$ then $\Theta_{V,W}=\theta^{\max}_{V,W}=\frac{\pi}{2}$, and $\cos\Theta_{V,W} + \cos\Theta_{V,W}^\perp = \prod_{i=1}^m \sin\theta_i \leq \sin \theta_1 = \cos(\frac{\pi}{2}-\theta_1)$, with equality if, and only if, \emph{(\ref{it:A})} $m=1$ or $\theta_2=\ldots=\theta_m=\frac{\pi}{2}$, or \emph{(\ref{it:C})} $\theta_1=0$.
In this case, (\ref{it:B}) implies (\ref{it:C}).

As $\Theta_{V,W} + \Theta_{V,W}^\perp \in[\frac{\pi}{2},\pi]$,
\SELF{by \cref{pr:zeta}\emph{\ref{it:range Theta}}}
its inequality for $p\geq 2$ follows from $\sin(\Theta_{V,W} + \Theta_{V,W}^\perp) = \sin\Theta_{V,W} \cos\Theta_{V,W}^\perp + \cos\Theta_{V,W} \sin\Theta_{V,W}^\perp \leq \cos\Theta^\perp_{V,W} + \cos\Theta_{V,W} \leq \cos \Delta\theta_{V,W} = \sin(\frac{\pi}{2} + \Delta\theta_{V,W})$. 
\end{proof}

\begin{figure}
	\centering
	\begin{subfigure}[b]{0.32\textwidth}
		\includegraphics[width=\textwidth]{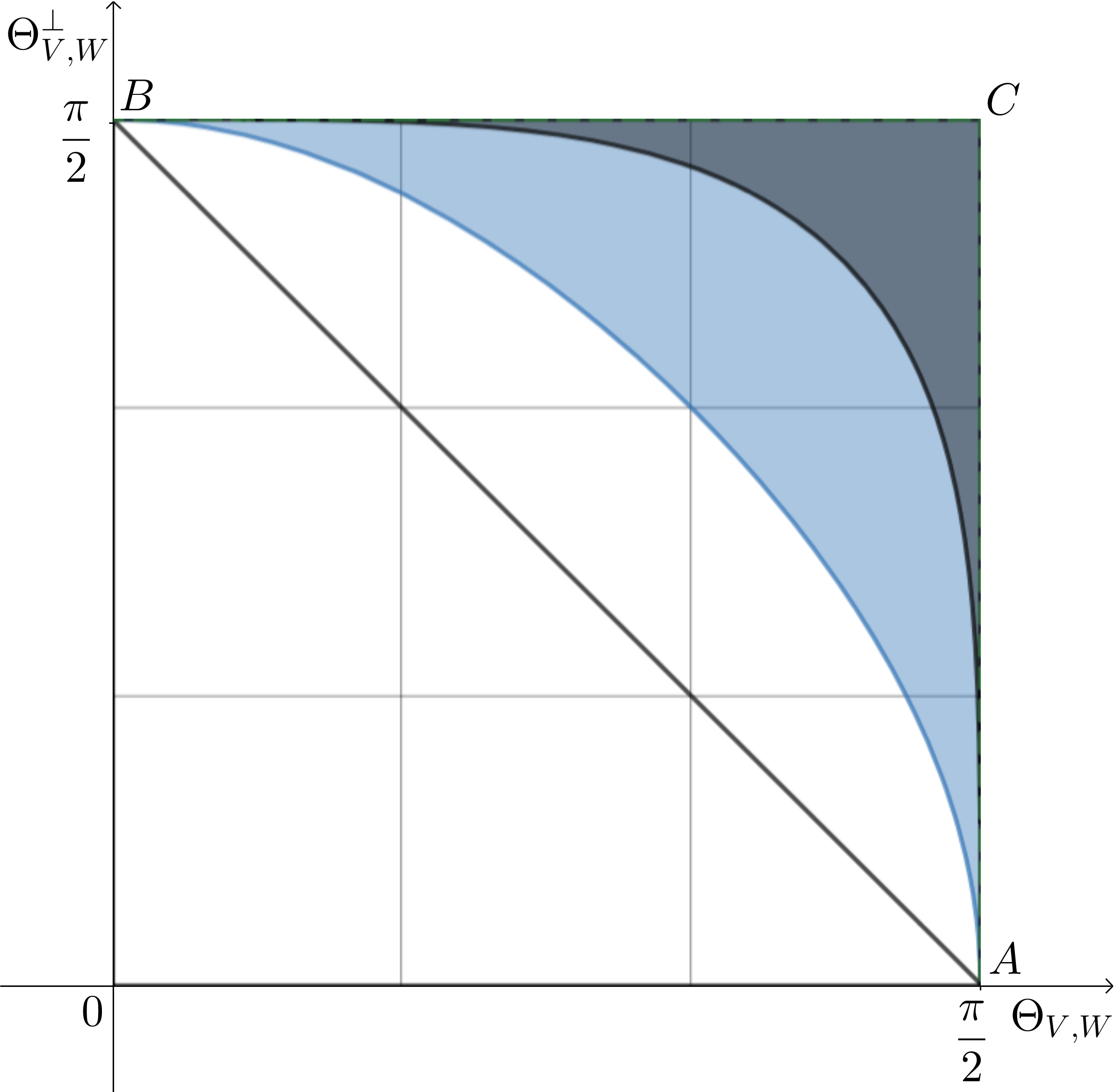}
		\caption{$\Delta\theta_{V,W}=0$}
		\label{fig:Theta Theta perp 1}
	\end{subfigure}
	\begin{subfigure}[b]{0.32\textwidth}
		\includegraphics[width=\textwidth]{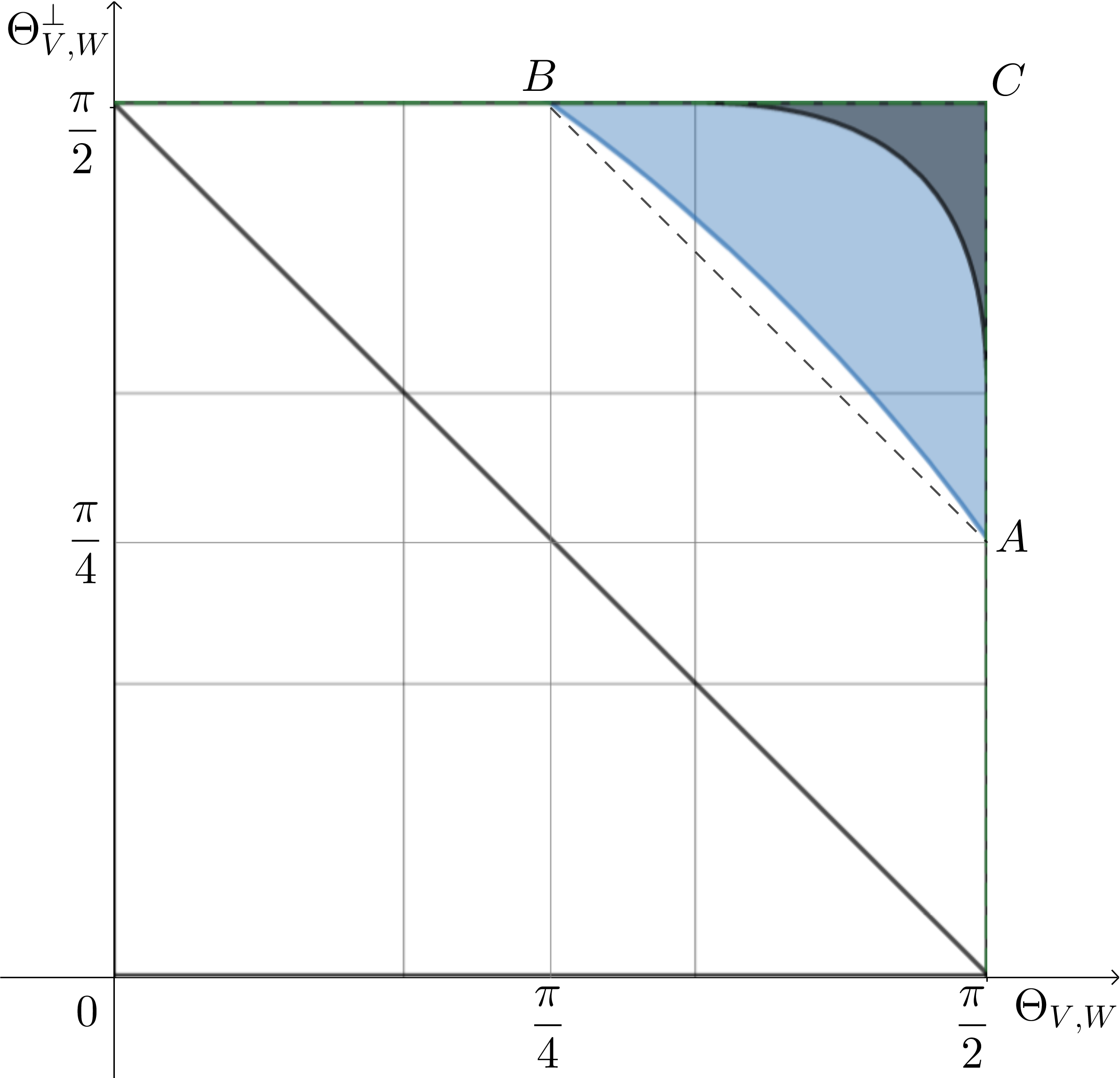}
		\caption{$\Delta\theta_{V,W}=\frac{\pi}{4}$}
		\label{fig:Theta Theta perp 2}
	\end{subfigure}
	\begin{subfigure}[b]{0.32\textwidth}
		\includegraphics[width=\textwidth]{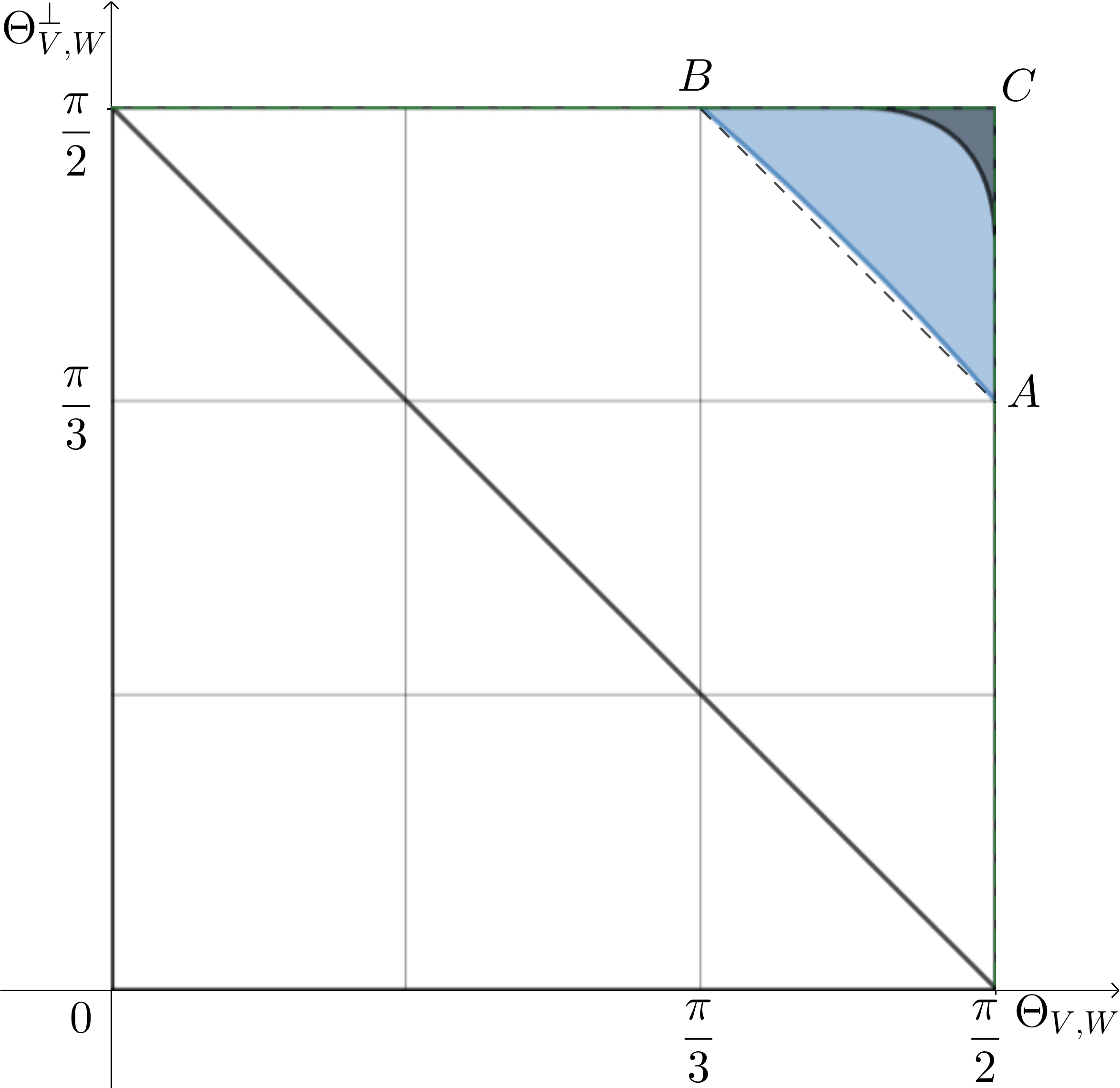}
		\caption{$\Delta\theta_{V,W}=\frac{\pi}{3}$}
		\label{fig:Theta Theta perp 3}
	\end{subfigure}
	\caption{\small Restrictions on  $\Theta_{V,W}$ and $\Theta^\perp_{V,W}$, depending on	$\Delta\theta_{V,W}$. 
	Shaded regions represent $\cos \Theta_{V,W} +\cos \Theta^\perp_{V,W}\leq\cos\Delta\theta_{V,W}$, with $(\cos \Theta_{V,W})^\frac{1}{2} +(\cos \Theta^\perp_{V,W})^\frac{1}{2}\leq\cos\Delta\theta_{V,W}$ in the darker one.
	The dashed line is $\Theta_{V,W} + \Theta_{V,W}^\perp = \frac{\pi}{2} + \Delta\theta_{V,W}$.
	Points $A$, $B$ and $C$ correspond to cases \emph{(\ref{it:A})}, \emph{(\ref{it:B})} and \emph{(\ref{it:C})} of \cref{pr:cos+cos}.
	}
	\label{fig:Theta Theta perp}
\end{figure}

\Cref{fig:Theta Theta perp} illustrates these results.
If $\dim V=1$, $(\Theta_{V,W},\Theta_{V,W}^\perp)$ is on the solid diagonal line, and if $\dim V= 2$ it is on the boundary curve of the shaded region (for now, ignore the distinction between its light and dark parts).
If $\dim V> 2$ it is in the interior of the region (which shrinks to $C$ as $\Delta\theta_{V,W}\rightarrow\frac{\pi}{2}$) or on the segments $AC$ or $BC$.
If $\dim V>\dim W$ it is on  $AC$.
It is at $C$ if, and only if, $\Delta\theta_{V,W}=\frac{\pi}{2}$.
\SELF{$\Theta_{V,W}=\frac\pi2$ with $\Theta_{V,W}^\perp=\frac\pi2$ implies C}

There are certainly other restrictions. For example, if $\Delta\theta_{V,W}=0$ and $p=\dim V$, the point will be on the curve $(\cos \Theta_{V,W})^\frac{2}{p} +(\cos \Theta^\perp_{V,W})^\frac{2}{p} = 1$.
Also, the admissible region retracts to a small neighborhood of $AC$ and $BC$ when $V$ and $W$ have high dimensions: as there are lots of principal angles, $\Theta_{V,W}\cong\frac{\pi}{2}$ unless $\theta_i\cong 0$ for almost all $i$, and $\Theta^\perp_{V,W}\cong\frac{\pi}{2}$ unless $\theta_i\cong \frac{\pi}{2}$ for almost all $i$.
A detailed analysis of how dimensions or the distribution of principal angles affect the relation between $\Theta_{V,W}$ and $\Theta_{V,W}^\perp$ would be interesting.

We note that, combining these results with propositions \ref{pr:Theta norm contraction} and \ref{pr:exterior product}, we can obtain, for subspaces $V,W\subset X$ represented by $\nu,\omega\in\Lambda X$, an upper bound for $\Delta\theta_{V,W}$ in terms of $\|\nu\lcontr\omega\| + \|\nu\wedge\omega\|$, or vice-versa.

%

\subsection{Simultaneously complexifiable subspaces}\label{sc:simult complex}

The exotic features of Grassmann angles can lead to unexpected results. As an example, we get an obstruction to having, in the real case, some complex structure turn two given subspaces into complex ones.
All spaces here are assumed to have even real dimension.


In the real case, an automorphism $J:X\rightarrow X$ such that $J^2=-I$ is a \emph{complex structure}.
It turns $X$ into a complex space, if we define $\im v=Jv$ for $v\in X$. 
A real subspace $V\subset X$ becomes complex if, and only if, it is invariant under $J$, i.e. $J(V)=V$. 
$J$ is \emph{compatible} (with the inner product) if $\inner{Ju,Jv} = \inner{u,v}$ for any $u,v\in X$, in which case $h(u,v) = \inner{u,v} + i\inner{Ju,v}$ is a Hermitian product.

For any $J$ most real subspaces do not become complex, but given a (even dimensional) real subspace there is a compatible $J$ that makes it complex.
This no longer holds if we set two subspaces to become complex.

\begin{definition}
Real subspaces $V,W\subset X$ are \emph{simultaneously complexifiable} if some compatible complex structure makes both complex.
\end{definition}

As the underlying real spaces of complex subspaces $\tilde{V}$ and $\tilde{W}$ have the same principal angles, but twice repeated, a necessary condition for $V$ and $W$ to be simultaneously complexifiable is that their principal angles be \emph{pairwise equal}, i.e. $\theta_{2i-1}=\theta_{2i}$ for $i=1,2,\ldots$. 
This is a strong requirement, which most pairs of subspaces do not satisfy.

By propositions \ref{pr:cos underlying real Theta} and \ref{pr:cos underlying real Theta perp}, if $V=(\tilde{V})_\R$ and $W=(\tilde{W})_\R$ then $\cos\Theta_{V,W} = \cos^2\Theta_{\tilde{V},\tilde{W}}$, $\cos\Theta^\perp_{V,W} = \cos^2\Theta^\perp_{\tilde{V},\tilde{W}}$ and $\Delta\theta_{V,W}=\Delta\theta_{\tilde{V},\tilde{W}}$.
\Cref{pr:cos+cos} gives the following obstruction.

\begin{proposition}
Nonzero real subspaces $V,W\subset X$ are not simultaneously complexifiable in the following cases:
\begin{enumerate}[i)]
\item $\dim V=4$ and $(\cos\Theta_{V,W})^\frac{1}{2} + (\cos\Theta_{V,W}^\perp)^\frac{1}{2} \neq \cos \Delta\theta_{V,W}$.
\item $\dim V> 4$ and $(\cos\Theta_{V,W})^\frac{1}{2} + (\cos\Theta_{V,W}^\perp)^\frac{1}{2}>\cos \Delta\theta_{V,W}$.
\end{enumerate}
\end{proposition}

The proof of \cref{pr:cos+cos} can be adapted to show that if $\dim V\geq 4$ and principal angles are pairwise equal then $(\cos\Theta_{V,W})^\frac{1}{2} + (\cos\Theta_{V,W}^\perp)^\frac{1}{2} \leq \cos \Delta\theta_{V,W}$, so this is in fact an obstruction on pairwise equality.

This result shows the difficulty of being simultaneously complexifiable. If $\dim V> 4$, it is necessary (but not sufficient) that $(\Theta_{V,W},\Theta_{V,W}^\perp)$ be in the darker region of \cref{fig:Theta Theta perp}, and if $\dim V=4$ it must be in its boundary curve. This region shrinks faster than the lighter one as $\Delta\theta_{V,W}$ increases, so the larger the angular range of one subspace with respect to the other, the harder it is that they can be simultaneously complexifiable.

\section{Formulas for Grassmann angles}\label{sc:Formulas}

We now obtain formulas for $\Theta_{V,W}$ and $\Theta_{V,W}^\perp$ in terms of arbitrary bases.
They also work for other volume projection angles, whose symmetry however imposes dimensional restrictions, like $\dim V\leq \dim W$ in \cref{pr:formula any base dimension} and $\dim V\leq \dim W^\perp$ in \cref{pr:formula complementary angle bases}.

\begin{theorem}\label{pr:formula any base dimension}
	Given bases $(v_1,\ldots,v_p)$ of $V$ and $(w_1,\ldots,w_q)$ of $W$, let $A=\big(\inner{w_i,w_j}\big), B=\big(\inner{w_i,v_j}\big)$, and $D=\big(\inner{v_i,v_j}\big)$. Then
	\begin{equation}\label{eq:formula Theta any base any dimension}
	\cos^2 \Theta_{V,W} = \frac{\det(\bar{B}^T \! A^{-1}B)}{\det D}.
	\end{equation}
	If $p=q$ this reduces to 
	\begin{equation}\label{eq:formula Theta equal dim}
		\cos^2 \Theta_{V,W} = \frac{|\det B\,|^2}{\det A \cdot\det D}.
	\end{equation}
\end{theorem}
\begin{proof}
	If $p>q$ then $\Theta_{V,W} = \frac \pi 2$, and the determinant of $\bar{B}^T \! A^{-1}B$ vanishes as it is a $p\times p$ matrix with rank at most $q$. 
	
	If $p\leq q$, Laplace expansion \cite{Muir2003}
	\SELF{p.80} 
	w.r.t. columns $q+1,\ldots,q+p$ of the $(q+p)\times(q+p)$ block matrix $M=\begin{psmallmatrix}
		A & B\  \\ 
		\bar{B}^T & 0_{p\times p}
	\end{psmallmatrix}$ gives
	\SELF{Como $J$ dá as últimas colunas, das $q+p$ linhas de $M$ só as $q$ primeiras não dão 0, por isso pode usar $\I_p^q$ ao invés de $\I_p^{q+p}$}
	\[ \det M = \sum_{\ii\in\I_p^q} (-1)^{\|\ii\|+pq+\frac{p(p+1)}{2}}\cdot \det B_\ii \cdot\det N_{\ii'}, \]
	where $B_\ii$ is the $p\times p$ submatrix of $M$ formed by lines of $B$ with indices in $\ii$, and $N_{\ii'}=\begin{psmallmatrix}
		A_{\ii'} \\[1pt] \bar{B}^T
	\end{psmallmatrix}$ is its $q\times q$ complementary submatrix, formed by lines of $A$ with indices not in $\ii$ and all of $\bar{B}^T$.
	
	For $\nu=v_1\wedge\ldots \wedge v_p$ and $\omega=w_1\wedge\ldots\wedge w_q$ we have, by \eqref{eq:contraction} and \eqref{eq:contraction coordinate decomposition},
	\begin{equation*}
		\|\nu\lcontr\omega\|^2 = \inner{\nu\wedge(\nu\lcontr\omega),\omega} = \sum_{\ii\in\I_p^q} \varepsilon_\ii \,\inner{\omega_\ii,\nu}\, \inner{\nu\wedge\omega_{\ii'},\omega}.
	\end{equation*}
	As $\det B_\ii=\inner{\omega_\ii,\nu}$ and $\det N_{\ii'} = \inner{\omega_{\ii'}\wedge\nu,\omega} = (-1)^{pq+p}\inner{\nu\wedge\omega_{\ii'},\omega}$,
	\SELF{$(-1)^{-p^2}=(-1)^p$ \\ $\varepsilon_\ii = (-1)^{\|\ii\|+\frac{p(p+1)}{2}}$}
	we get $\|\nu\lcontr\omega\|^2 = (-1)^p \det M$.
	\Cref{pr:Theta norm contraction} gives $\cos^2 \Theta_{V,W} = \frac{(-1)^p \det M}{\det A\det D}$, and the result follows from Schur's determinant identity \cite{Brualdi1983}. \SELF{$\det(-BA^{-1}\bar{B}^T) = (-1)^p \det(BA^{-1}\bar{B}^T)$}
\end{proof}

We note that \eqref{eq:formula Theta any base any dimension} is simpler than a formula given in \cite{Gunawan2005} for a similar angle (which corrects another from \cite{Risteski2001}).
And \eqref{eq:formula Theta equal dim} can be obtained directly from \cref{pr:Theta inner blades}. 

\begin{example}\label{ex:formula bases}
	In $\C^3$, let  $V=\Span(v_1,v_2)$ and $W=\Span(w_1,w_2)$ with $v_1=(1,-\xi,0), v_2=(0,\xi,-\xi^2)$, $w_1=(1,0,0)$ and $w_2=(0,\xi,0)$,  where $\xi=e^{\im\frac{2\pi}{3}}$. By \eqref{eq:formula Theta equal dim},	$\Theta_{V,W}=\arccos\frac{\sqrt{3}}{3}$. 
	Since $v=(\xi,\xi^2,-2)\in V$ and $w=(1,\xi,0)\in W$ are orthogonal to $V\cap W=\Span(v_1)$,  \cref{pr:properties Grassmann}\emph{\ref{it:orth complem inter}}
\SELF{also \ref{pr:properties Grassmann}\emph{\ref{it:Theta theta}} and \ref{pr:the angle lines}\emph{\ref{it:angle complex lines}}
} 
	gives $\Theta_{V,W}=\Theta_{\C v,\C w}=\gamma_{v,w}$, and \eqref{eq:angles vectors} confirms the result.
\end{example}

As the next example shows, \eqref{eq:formula Theta any base any dimension} only holds without any dimensional conditions thanks to the Grassmann angle asymmetry.

\begin{example}\label{ex:formula distinct dim}
	In $\R^4$, if $v=(1,0,1,0)$, $w_1=(0,1,1,0)$, $w_2=(1,2,2,-1)$, $V=\Span(v)$ and $W=\Span(w_1,w_2)$ then 
	$A=\begin{psmallmatrix}
		2 & 4 \\ 4 & 10
	\end{psmallmatrix}$,
	$B=\begin{psmallmatrix}
		1 \\ 3
	\end{psmallmatrix}$, 
	$D=(2)$, and \eqref{eq:formula Theta any base any dimension} gives $\Theta_{V,W}= 45^\circ$, as one can verify by projecting $v$ on $W$. 
	Switching the roles of $V$ and $W$, we now have $A=(2)$,
	$B=(1 \ 3)$, $D=\begin{psmallmatrix}
		2 & 4 \\ 4 & 10
	\end{psmallmatrix}$ and $\Theta_{W,V}= 90^\circ$, as expected since $\dim W>\dim V$.
\end{example}

\begin{example}
	In $\C^3$, let $v=(1,0,i)$, $w_1=(1,0,0)$, $w_2=(i,1,0)$, $V=\Span(v)$ and $W=\Span(w_1,w_2)$. 
	Then $A=\begin{psmallmatrix}
		1 & i \\ -i & 2
	\end{psmallmatrix}$,
	$B=\begin{psmallmatrix}
		1 \\ -i
	\end{psmallmatrix}$, 
	$D=(2)$, and we get  $\Theta_{V,W}= 45^\circ$.
	Identifying $\C^3$ with $\R^6$, we have $V_\R=\Span_\R(v,iv)$ and $W_\R=\Span_\R(w_1,iw_1,w_2,iw_2)$, with
	\begin{align*}
		v&=(1,0,0,0,0,1), & w_1&=(1,0,0,0,0,0), & w_2&=(0,1,1,0,0,0), \\
		iv&=(0,1,0,0,-1,0), & iw_1&=(0,1,0,0,0,0), & iw_2&=(-1,0,0,1,0,0).
	\end{align*}
	Now $A=\begin{psmallmatrix}
		1 & 0 & 0 & -1 \\
		0 & 1 & 1 & 0 \\
		0 & 1 & 2 & 0 \\
		-1 & 0 & 0 & 2
	\end{psmallmatrix}$,
	$B=\begin{psmallmatrix}
		1 & 0 \\
		0 & 1 \\
		0 & 1 \\
		-1 & 0
	\end{psmallmatrix}$,
	$D=\begin{psmallmatrix}
		2 & 0 \\
		0 & 2
	\end{psmallmatrix}$
	and $\Theta_{V_\R,W_\R}= 60^\circ$, in agreement with \cref{pr:cos underlying real Theta}.
\end{example}

\begin{theorem}\label{pr:formula complementary angle bases}
	Given bases $(v_1,\ldots,v_p)$ of $V$ and $(w_1,\ldots,w_q)$ of $W$, let $A=\big(\inner{w_i,w_j}\big), B=\big(\inner{w_i,v_j}\big)$, and $D=\big(\inner{v_i,v_j}\big)$. Then
	\begin{equation}\label{eq:complementary bases}
		\cos^2 \Theta^\perp_{V,W} = \frac{\det(A-BD^{-1}\bar{B}^T )}{\det A}.
	\end{equation}
\end{theorem}
\begin{proof}
	Follows by applying \cref{pr:exterior product} to $\omega\wedge\nu$, where $\nu=v_1\wedge\ldots\wedge v_p$ and $\omega=w_1\wedge\ldots\wedge w_q$,
	and using Schur's identity with 
	$M=\begin{psmallmatrix}
		A & B \\ 
		\bar{B}^T & D
	\end{psmallmatrix}$.
\end{proof}


\begin{corollary}\label{pr:complementary orthonormal bases}
	If $P$ is a matrix representing $\Proj^V_W$ in orthonormal bases of $V$ and $W$ then $\cos^2\Theta_{V,W}^\perp=\det(\mathds{1}_{q\times q}-P\bar{P}^T )$.
\end{corollary}

With \cref{pr:Proj diagonal}, this gives another easy proof for \cref{pr:complementary product sines}.

\begin{example}
	In \cref{ex:formula bases}, using \eqref{eq:complementary bases} with bases $(v_1,v_2)$ and $(w_1,w_2)$, or \cref{pr:complementary orthonormal bases} with orthonormal bases $(\frac{v_1}{\sqrt{2}},\frac{v}{\sqrt{6}})$ and $(\frac{v_1}{\sqrt{2}},\frac{w}{\sqrt{2}})$, we get $\Theta^\perp_{V,W}=90^\circ$, as expected since $V\cap W\neq\{0\}$.
	\SELF{\ref{pr:complementary simple cases}\emph{\ref{it:Theta perp pi 2}}}
\end{example}

\begin{example}
	In \cref{ex:formula distinct dim}, \eqref{eq:complementary bases} gives $ 45^\circ$ for both $\Theta^\perp_{V,W}$ and $\Theta^\perp_{W,V}$, in agreement with \cref{pr:symmetry}\emph{\ref{it:symmetry complementary}}.
	Principal angles confirm the results: $V$ and $W^\perp$ have only $45^\circ$, while $W$ and $V^\perp$ have $0^\circ$ and $45^\circ$.
\end{example}

\section{Generalized Pythagorean identities}\label{sc:Pythagorean trigonometric}

The Pythagorean trigonometric identity $\cos^2\theta+\sin^2\theta=1$ can be written as 
$\cos^2\theta_x+\cos^2\theta_y=1$, with $\theta_x$ and $\theta_y$ being angles a line in $\R^2$ makes with the axes. 
It has known generalizations for higher dimensional real spaces, which we extend to complex ones.
This is important because, with \cref{pr:Grassmann angle volume projection}, they lead to Pythagorean theorems for volumes \cite{Mandolesi_Pythagorean}, which in the complex case are simpler and have important implications for quantum theory \cite{Mandolesi_Born}. We also get a geometric interpretation for a property of the Clifford product \cite{Hestenes1984clifford,Hitzer2010a}.

The first identity relates the Grassmann angles of a (real or complex) line with all subspaces of an orthogonal partition  of $X$.

\begin{theorem}\label{pr:pythagorean line}
	Given an orthogonal partition $X=W_1\oplus\cdots\oplus W_k$ and a line $L\subset X$, we have $	\sum_{i=1}^k \cos^2\Theta_{L,W_i} = 1$.
\end{theorem}
\begin{proof}
	Given a nonzero $v\in L$, as $\|v\|^2 = \sum_{i} \left\|\Proj_{W_i} v \right\|^2$ the result follows from \cref{pr:properties Grassmann}\emph{\ref{it:Pv}}.
\end{proof}

\begin{figure}
	\centering
	\subfloat[$\cos^2\theta_x+\cos^2\theta_y+\cos^2\theta_z=1$]{
		\resizebox*{0.38\textwidth}{!}{\includegraphics{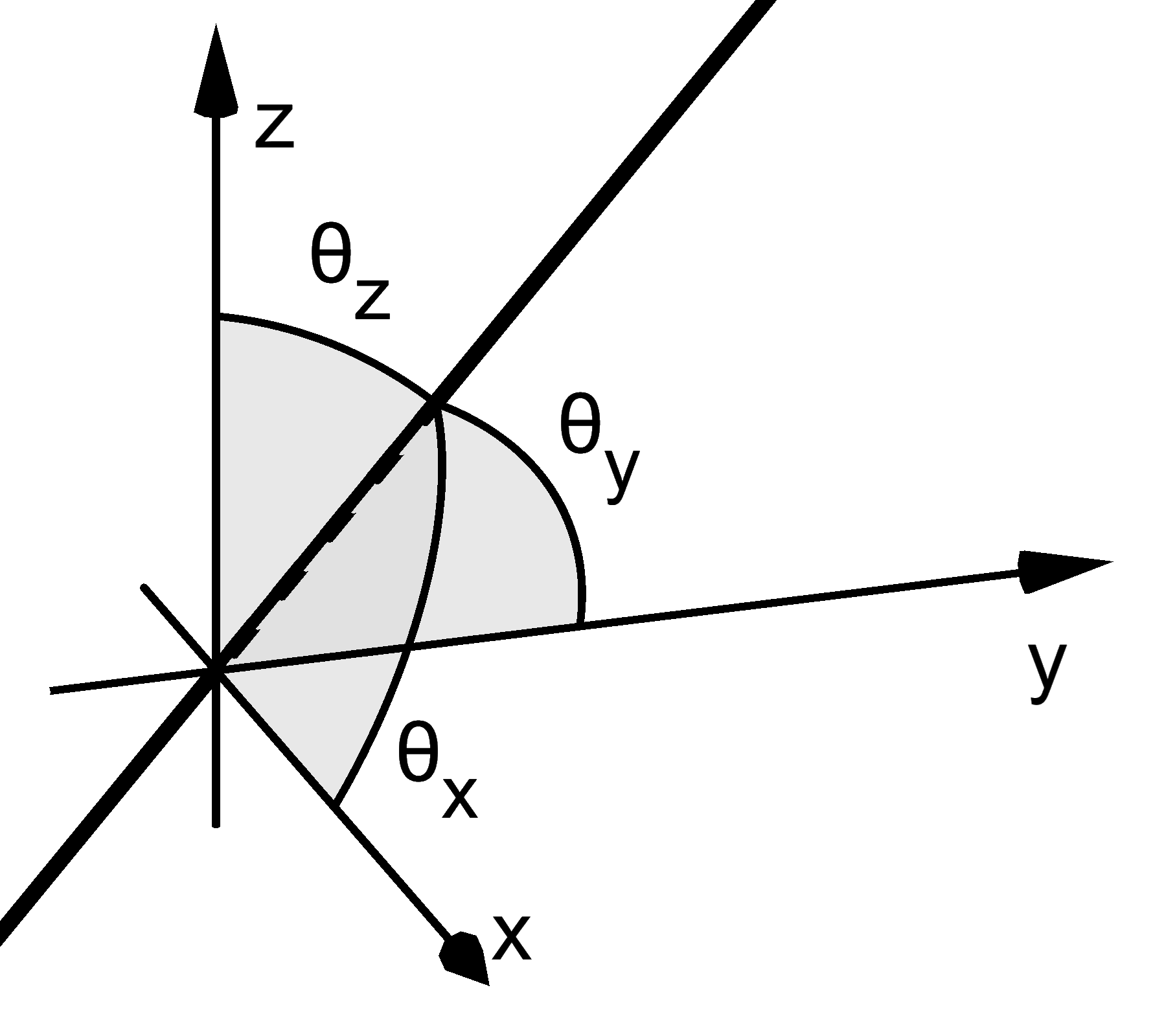}}\label{fig:angulos-eixos-edit}}
	\hfill
	\subfloat[$\cos^2\theta_{xy}+\cos^2\theta_{xz}+\cos^2\theta_{yz}=1$]{
		\resizebox*{0.41\textwidth}{!}{\includegraphics{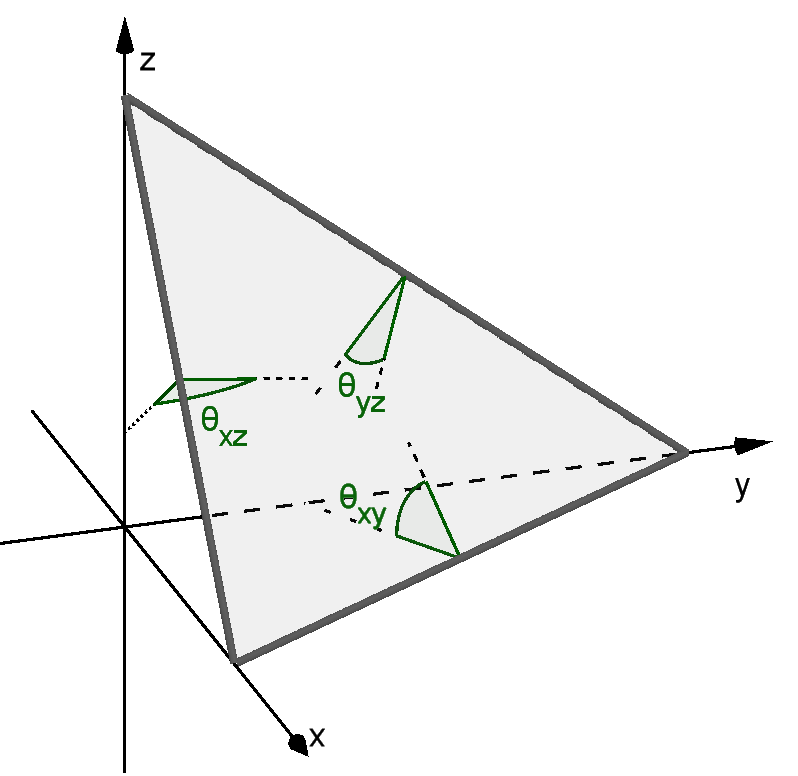}}\label{fig:angulos-faces-edit}}
	\caption{Pythagorean identities for subspaces of equal dimensions.}
	\label{fig:equal dimensions}
\end{figure}

The following example is a known identity for direction cosines and, like others we give, is only meant to illustrate the theorem in $\R^3$. 

\begin{example}\label{ex:direction cosines}
	If $\theta_x, \theta_y$ and $\theta_z$ are the angles between a line in $\R^3$ and the axes (\cref{fig:angulos-eixos-edit}), then $\cos^2\theta_x+\cos^2\theta_y+\cos^2\theta_z=1$. 
\end{example}

The relevance of our result lies mainly in the complex case, where \cref{pr:projection factor Theta} turns it into $\sum_{i=1}^k \pi_{L,W_i} = 1$.
The meaning of this is that, as seen in \cref{ex:sum projections}, the sum of the (non-squared) volumes of the projections equals the original volume.

\begin{example}
	If $X$ is the complex Hilbert space of a quantum system \cite{CohenTannoudji2019}, $L=\C\psi$ for a quantum state vector $\psi$, and the $W_i$'s are eigenspaces of a quantum observable,
	the probability of getting result $i$ when measuring $\psi$ is  $p_i=\|\Proj_{W_i}\psi\|^2/\|\psi\|^2 =\cos^2 \Theta_{L,W_i}$,
	\SELF{\ref{pr:properties Grassmann}\emph{\ref{it:Pv}}}
	so \cref{pr:pythagorean line} simply means the total probability is 1.
	But the fact that
	projection areas add up to the original area leads in \cite{Mandolesi_Born} to a new interpretation for quantum probabilities, a proof of the Born rule, and an explanation for why the quantum Hilbert space must be complex.
\end{example}

The next identities relate Grassmann angles of a (real or complex) subspace with coordinate subspaces of an orthogonal basis of $X$.

\begin{theorem}\label{pr:Grassmann coordinate any dim}
	Let $V\subset X$ be a $p$-subspace, $0\leq q\leq n=\dim X$, and the $W_\ii$'s be  coordinate $q$-subspaces of an orthogonal basis of $X$.
	\begin{enumerate}[i)]
		\item If $p\leq q$ then $\sum_{\ii\in\I_q^n} \cos^2\Theta_{V,W_\ii} =\binom{n-p}{n-q}$. \label{it:V Wi}
		\item If $p>q$ then $\sum_{\ii\in\I_q^n} \cos^2\Theta_{W_\ii,V} = \binom{p}{q}$. \label{it:Wi V}
	\end{enumerate}
\end{theorem}
\begin{proof}
	Assume $p,q\neq 0$ and that the basis $(w_1,\ldots,w_n)$ is orthonormal.
	Its coordinate blades \eqref{eq:coordinate blades} form, for any $0\leq r\leq n$, orthonormal bases $\{\omega_\ii\}_{\ii\in\I_r^n}$ of $\Lambda^r X$ and $\{\omega_{\ii'}\}_{\ii\in\I_r^n}$ of $\Lambda^{n-r} X$. 
	
	\emph{(\ref{it:V Wi})} For an unit blade $\nu\in\Lambda^p V$ and  $\ii=(i_1,\ldots,i_q)\in\I_q^n$, \eqref{eq:contraction coordinate decomposition} gives
		$ \nu \lcontr  \omega_\ii = \sum_{\jj\in\I_p^q} \varepsilon_\jj \,\inner{\nu,(\omega_\ii)_\jj}\, (\omega_\ii)_{\jj'}$,
		where $(\omega_\ii)_\jj = w_{ i_{j_1}}\wedge\ldots\wedge w_{ i_{j_p}}$ for $\jj=(j_1,\ldots,j_p)$, and likewise for $(\omega_\ii)_{\jj'}$.
		As the $(\omega_\ii)_{\jj'}$'s are orthonormal, \cref{pr:Theta norm contraction} gives
		$ \cos^2\Theta_{V,W_\ii} = \|\nu \lcontr  \omega_\ii\|^2 = \sum_{\jj\in\I_p^q} \left|\inner{\nu,(\omega_\ii)_\jj}\right|^2$.
		So
		\begin{align*}
			\sum_{\ii\in\I_q^n} \cos^2\Theta_{V,W_\ii} 
			&= \sum_{\ii\in\I_q^n} \sum_{\jj\in\I_p^q} \left|\inner{\nu,(\omega_\ii)_\jj}\right|^2 \\
			&= \frac{\binom{n}{q} \binom{q}{p}}{\binom{n}{p}} \sum_{\kk\in\I_p^n} \left|\inner{\nu,\omega_\kk}\right|^2 = \binom{n-p}{n-q} \|\nu\|^2,
		\end{align*}
		where the binomial coefficients account for the number of times each $\omega_\kk$ appears as a $(\omega_\ii)_\jj$ in the double summation.
		
		\emph{(\ref{it:Wi V})} For $\ii\in\I_q^n$, \cref{pr:symmetry}\emph{\ref{it:Theta perp perp}} gives $\Theta_{W_\ii,V} = \Theta_{V^\perp,{W_\ii}^\perp}$. 
		As ${W_\ii}^\perp = W_{\ii'}$ for $\ii'\in\I_{n-q}^n$, and $\dim V^\perp = n-p < n-q = \dim {W_\ii}^\perp$, the result follows from the previous case. \qedhere
\end{proof}

For $p=q$, $\sum_{\ii\in\I_p^n} \cos^2\Theta_{V,W_\ii} = 1$ follows directly from \cref{pr:Theta inner blades}.
In the complex case, for $V_\R$ we have $\sum_{\ii\in\I_{2p}^{2n}} \cos^2\Theta_{V_\R,W_\ii} = 1$, where the $W_\ii$'s are now all $\binom{2n}{2p}$ coordinate $2p$-subspaces of an orthogonal basis of $X_\R$. 
\Cref{pr:cos underlying real Theta}\emph{\ref{it:Theta VR bigger}} compensates the larger number of coordinate subspaces.\SELF{$\cos\Theta_{V_\R,W_\R} \leq \cos\Theta_{V,W}$}


The next example is a dual of \cref{ex:direction cosines} via \cref{pr:symmetry}\emph{\ref{it:Theta perp perp}}, and the following two are duals of each other.

\begin{example}
If $\theta_{xy}, \theta_{xz}$ and $\theta_{yz}$ are the angles a plane in $\R^3$ makes with the coordinate planes (\cref{fig:angulos-faces-edit}) then
\SELF{Another way to express this result is that the sum of the squared cosines of all angles 
	between the faces of a trirectangular tetrahedron equals 1. By \cref{pr:Grassmann coordinate any dim}, this generalizes for simplices of any dimension {Cho1992}.
}
$\cos^2\theta_{xy}+\cos^2\theta_{xz}+\cos^2\theta_{yz}=1$.
\end{example}


\begin{figure}
	\centering
	\subfloat[$\cos^2\theta_{xy}+\cos^2\theta_{xz}+\cos^2\theta_{yz}=2$]{
		\resizebox*{0.395\textwidth}{!}{\includegraphics{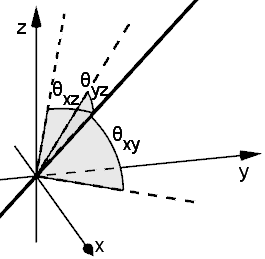}}\label{fig:angulos-linha-planos}}
	\hfill
	\subfloat[$\cos^2\theta_{x}+\cos^2\theta_{y}+\cos^2\theta_{z}=2$]{
		\resizebox*{0.42\textwidth}{!}{\includegraphics{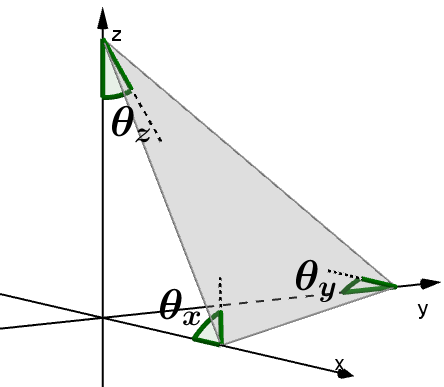}}\label{fig:angulos plano eixos}}
	\caption{Pythagorean identities for subspaces of different dimensions.}
	\label{fig:diferent dimensions}
\end{figure}

\begin{example}
	If $\theta_{xy}, \theta_{xz}$ and $\theta_{yz}$ are the angles a line in $\R^3$ makes with the coordinate planes (\cref{fig:angulos-linha-planos}) then   $\cos^2\theta_{xy}+\cos^2\theta_{xz}+\cos^2\theta_{yz}=2$.
\end{example}

\begin{example}
	If $\theta_{x}, \theta_{y}$ and $\theta_{z}$ are the angles between the axes and a plane in $\R^3$ (\cref{fig:angulos plano eixos}) then $\cos^2\theta_{x}+\cos^2\theta_{y}+\cos^2\theta_{z}=2$.
\end{example}

\begin{example}
	In \cref{ex:formula bases}, let $w_3=(0,0,\xi^2)$ and $W_{ij}=\Span(w_i,w_j)$. 
	The unitary transformation given by $T=\left(\begin{smallmatrix}
		0 & 0 & \xi \\ 
		\xi & 0 & 0 \\ 
		0 & \xi & 0
	\end{smallmatrix}\right)$ maps $W_{12}\mapsto W_{23}$, $W_{23}\mapsto W_{13}$, and preserves $V$, so \cref{pr:properties Grassmann}\emph{\ref{it:transformation}} gives $\Theta_{V,W_{12}} = \Theta_{V,W_{23}} = \Theta_{V,W_{13}}$.
	\SELF{$T:w_1\mapsto w_2\mapsto w_3\mapsto w_1$, $Tv_1=v_2$,\\ $Tv_2= -(v_1+v_2)$}
	As $W_{12}$, $W_{13}$ and $W_{23}$ are the coordinate $2$-subspaces of the orthogonal basis $(w_1,w_2,w_3)$  of $\C^3$, \cref{pr:Grassmann coordinate any dim} gives $\cos\Theta_{V,W_{ij}}=\frac{\sqrt{3}}{3}$, in agreement with that example.
\end{example}
  
In \cite{Mandolesi_Products} we show that the Clifford product of blades $\ganu,\gaomega\in\Lambda X$ (of same grade $p$, for simplicity), representing $V,W\subset X$, satisfies $\|\nu\omega\|^2 = \|\nu\|^2\|\omega\|^2 \sum_{\ii} \cos^2\Theta_{V,Y_\ii}$, where the $Y_\ii$'s are all coordinate $p$-subspaces of a certain orthogonal basis of $Y=V+W$.
\Cref{pr:Grassmann coordinate any dim} gives a geometric interpretation for the property $\|\nu\omega\|= \|\nu\|\|\omega\|$: it only holds because the product involves projections on all $Y_\ii$'s.
Other products in Grassmann or Clifford algebras are submultiplicative for blades because, as propositions \ref{pr:Theta inner blades}, \ref{pr:Theta norm contraction} and \ref{pr:exterior product} show, they involve projections on single subspaces, which contract volumes and reduce norms.

\section{Other identities}\label{sc:other identities}

The next identity relates the angle between oriented subspaces of same dimension to angles they make with coordinate subspaces of $X$. 

\begin{theorem}\label{pr:identity oriented}
	Given oriented $p$-subspaces $V,W\subset X$,
	\begin{equation*}
		\cos\mathbf{\Theta}_{V,W} = \sum_{\ii\in\I_p^n} \cos\mathbf{\Theta}_{V,X_\ii} \cos\mathbf{\Theta}_{X_\ii,W},
	\end{equation*}
	where the $X_\ii$'s are the coordinate $p$-subspaces of an orthogonal basis of $X$, with orientations given by the corresponding coordinate $p$-blades.
\end{theorem}
\begin{proof}
	Follows by decomposing unit blades $\nu\in\Lambda^p V$ and $\omega\in\Lambda^p W$ (with the orientations of $V$ and $W$) in the orthonormal basis of $\Lambda^p X$ formed with the normalized coordinate $p$-blades, and applying \eqref{eq:oriented Theta inner blades}.
\end{proof}

Note the order of the subspaces, which is relevant in the complex case.
This identity improves upon the following inequality, given in \cite{Miao1996} n terms of products of cosines of principal angles.

\begin{corollary}\label{pr:cor Miao}
	$\cos\Theta_{V,W} \leq \sum_\ii \cos\Theta_{V,X_\ii} \cos\Theta_{W,X_\ii}$.
\end{corollary}

\begin{example}
	$\cos\theta = \cos\alpha_x\cos\beta_x + \cos\alpha_y\cos\beta_y + \cos\alpha_z\cos\beta_z$ for the angle $\theta\in[0,\pi]$ between 2 oriented lines in $\R^3$ forming angles $\alpha_x,\alpha_y,\alpha_z$ and $\beta_x,\beta_y,\beta_z$ (all in $[0,\pi]$) with the positive axes.
\end{example}

\begin{example}
In \cref{ex:oriented Theta}, if $W$ is an oriented 2-subspace having $\Theta_{W,X_{13}} = \Theta_{W,X_{23}} = \frac{\pi}{3}$, $\phi_{W,X_{13}}=\frac{5\pi}{6}$ and $\phi_{W,X_{23}}=0$
\SELF{Ex.: $\omega = (e^{\im\frac{5\pi}{6}} e_{1} - \frac{\sqrt{2}}{2} e_{3}) \wedge (\frac{\sqrt{2}}{2} e_{2} + \frac{1}{2} e_{3})$}
then $\cos\mathbf{\Theta}_{V,W} = 0 - \frac{\sqrt{2}}{2}\cdot\frac{1}{2}e^{-i\frac{5\pi}{6}} + i \frac{\sqrt{2}}{2}\cdot\frac{1}{2} = \frac{\sqrt{6}}{4} e^{i\frac{\pi}{3}}$,
\SELF{\ref{pr:identity oriented}}
so that $\Theta_{V,W} = \arccos(\frac{\sqrt{6}}{4}) \cong 52.2^\circ$ (volumes orthogonally projected between $V$ and $W$ shrink by $\frac{3}{8}$) and $\phi_{V,W} = \frac{\pi}{3}$  (orientation of $V$ rotated by $\frac{\pi}{3}$ aligns with that of $W$).
\end{example}

The following identity gives the angle between $U\subset V$ and $W$ in terms of angles with principal coordinate subspaces $V_\ii$ of $V$.

\begin{theorem}\label{pr:identity Theta principal subspaces}
	Given nonzero
	\SELF{Só porque não defini principal basis com $\{0\}$} 
	subspaces $V,W\subset X$ and $U\subset V$, let $r=\dim U$, $p=\dim V$, and the $V_\ii$'s be the coordinate $r$-subspaces of a principal basis of $V$ w.r.t. $W$. Then
	\[ \cos^2\Theta_{U,W} = \sum_{\ii\in\I^p_r} \cos^2\Theta_{U,V_\ii} \cdot \cos^2\Theta_{V_\ii,W}. \]
\end{theorem}
\begin{proof}
	Let $P=P_W$.
	The coordinate $r$-blades $\nu_\ii\in\Lambda^r V_\ii$ form an orthonormal basis of $\Lambda^r V$, so $P\mu = \sum_\ii \inner{\nu_\ii,\mu} P\nu_\ii$ for an unit $\mu\in\Lambda^r U$.
	By \cref{pr:principal blades proj orth}  the $P\nu_\ii$'s are mutually orthogonal, so $\|P\mu\|^2 = \sum_\ii |\inner{\mu,\nu_\ii}|^2 \|P\nu_\ii\|^2$. 
	The result follows from propositions \ref{pr:Theta Pnu} and \ref{pr:Theta inner blades}. 
\end{proof}

As $\sum_{\ii} \cos^2\Theta_{U,V_\ii} =1$, by \cref{pr:Grassmann coordinate any dim},   $\cos^2\Theta_{U,W}$ is a weighted average of the $\cos^2\Theta_{V_\ii,W}$'s, with weights given by the $\cos^2\Theta_{U,V_\ii}$'s.

\begin{figure}
	\centering
	\includegraphics[width=0.7\linewidth]{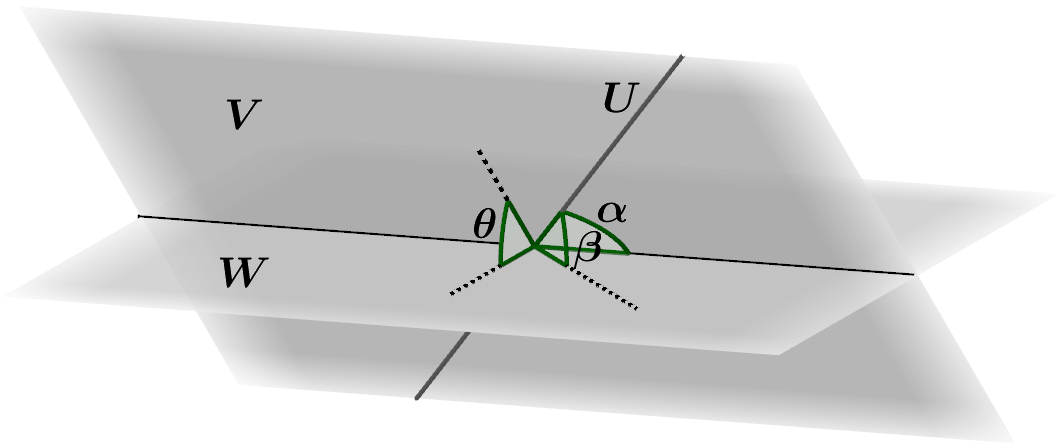}
	\caption{$\cos^2\beta = \cos^2\alpha +\sin^2\alpha\cos^2\theta$}
	\label{fig:angles principal lines}
\end{figure}

\begin{example}
	Given planes $V,W\subset\R^3$ and a line $U\subset V$, let $\alpha=\Theta_{U,V\cap W}$, $\beta=\Theta_{U,W}$ and $\theta=\Theta_{V,W}$ (\cref{fig:angles principal lines}). 
	As $V\cap W$ and $(V\cap W)^\perp\cap V$ are principal lines of $V$ w.r.t. $W$,  $\cos^2\beta = \cos^2\alpha\cdot 1 +\sin^2\alpha\cdot\cos^2\theta$. 
\end{example}



The final identities relate the angle between a direct sum and $W$ to the angles the subspaces  make with $W$ and the complementary Grassmann angles between their projections.

\begin{theorem}\label{pr:Theta direct sum}
	Let $V_1,V_2,W\subset X$ be subspaces, with $V_1\cap V_2=\{0\}$, and $P=\Proj_W$. Then
	\begin{equation*}
	\cos \Theta_{V_1\oplus V_2,W} = \cos \Theta_{V_1,W}\cdot \cos \Theta_{V_2,W}\cdot \frac{\cos \Theta^\perp_{P(V_1),P(V_2)}}{\cos \Theta^\perp_{V_1,V_2}}.
	\end{equation*}
\end{theorem}
\begin{proof}
	Assume $V_1\not\pperp W$ and $V_2\not\pperp W$, otherwise $V_1\oplus V_2\pperp W$ and both sides vanish.
	\SELF{\cref{pr:properties Grassmann}\emph{\ref{it:Theta pi2}}}
	Let $\nu_1$ and $\nu_2$ be unit blades representing $V_1$ and $V_2$. As $\nu_1\wedge\nu_2$ represents $V_1\oplus V_2$, 
	\SELF{as $V_1\perp V_2$}
	and by \cref{pr:P(V)} $P\nu_1$ and $P\nu_2$ represent $P(V_1)$ and $P(V_2)$, the result follows from propositions \ref{pr:Theta Pnu} and \ref{pr:exterior product}.
	\SELF{$\cos \Theta_{V_1\oplus V_2,W} = \|P(\nu_1\wedge\nu_2)\| = \|(P\nu_1)\wedge(P\nu_2)\|$}
\end{proof}


\begin{corollary}\label{pr:Theta orthog partition}
	Let $V,W\subset X$ be subspaces, $V=\bigoplus_{i=1}^k V_i$ be a partition, and $P=\Proj_W$. Then
	\begin{equation}\label{eq:Theta orthog partition}
		\cos \Theta_{V,W} = \prod_{i=1}^k \cos \Theta_{V_i,W} \cdot \prod_{i=1}^{k-1} \frac{\cos \Theta^\perp_{P(V_i),P(V_{i+1}\oplus\ldots\oplus V_k)}}{\cos \Theta^\perp_{V_i,V_{i+1}\oplus\ldots\oplus V_k}}.
	\end{equation}
\end{corollary}




This allows us to use angles to characterize principal partitions.

\begin{proposition}\label{pr:converse Theta partition}
	Let $V,W\subset X$ be nonzero subspaces, with $V\not\pperp W$. A partition $V=\bigoplus_i V_i$ is principal w.r.t. $\!W$ if, and only if, it is orthogonal and $\cos \Theta_{V,W} \!= \prod_i\! \cos\Theta_{V_i,W}$.
\end{proposition}
\begin{proof}
	If the partition is principal, the result is immediate from the definition. For the converse, \cref{pr:complementary simple cases}\emph{\ref{it:Theta perp 0}} and \eqref{eq:Theta orthog partition} imply the $P(V_i)$'s are orthogonal, so it follows from \cref{pr:principal partition}.
\end{proof}


\section{Final remarks}\label{sc:conclusion}

Grassmann angles have other interesting properties and applications which we develop in \cite{Mandolesi_Products} and other articles.
Results found in the literature for other volume projection angles should also apply to them.
For example, a theorem of \cite{Gluck1967} implies the generalized Frenet curvature $k_p$ of a curve measures the rate of change of the Grassmann angle for the osculating subspaces spanned by the first $p$ Frenet vectors.

As indicated in \cref{sc:Related angles}, some of our results correspond to known facts from other formalisms (e.g. matrices and determinants in \cite{Miao1992,Miao1996}). The rest can certainly be translated into properties of matrices or other structures, and our formalism can also allow other results about determinants, for example, to be more easily expressed and obtained. 

For example, Miao and Ben-Israel \cite{Miao1992} show that, if $B_1$ and $B_2$ are $2r\times r$ matrices of rank $r$ then, in their notation,
\begin{multline*}
\det\begin{pmatrix}
	B_1^T B_1 & B_1^T B_2 \\
	B_2^T B_1 & B_2^T B_2
\end{pmatrix} =
\det\begin{pmatrix}
	\det(B_1^T B_1) & \det(B_1^T B_2) \\
	\det(B_2^T B_1) & \det(B_2^T B_2)
\end{pmatrix} \\
 \times\frac{\sin^2\{R(B_1),R(B_2)\}}{1-\cos^2\angle\{C_r(B_1),C_r(B_2)\}}.
\end{multline*}
Translating this into our formalism, it means that, if $V_i$ is spanned by  the columns of $B_i$, and $\nu_i$ is their exterior product, then
\begin{equation*}
\|\nu_1\wedge\nu_2\|^2 = 
\det\begin{pmatrix}
	\inner{\nu_1,\nu_1} & \inner{\nu_1,\nu_2} \\
	\inner{\nu_2,\nu_1} & \inner{\nu_2,\nu_2}
\end{pmatrix}
\cdot \frac{\cos^2\Theta_{V_1,V_2}^\perp}{\sin^2\Theta_{V_1,V_2}},
\end{equation*}
what follows immediately from \cref{pr:Theta inner blades,pr:exterior product}.
They also show the first determinant is less or equal than the other, by proving that $\frac{\sin^2\{R(B_1),R(B_2)\}}{1-\cos^2\angle\{C_r(B_1),C_r(B_2)\}}\leq 1$.
In our case, this means $\frac{\cos^2\Theta_{V_1,V_2}^\perp}{\sin^2\Theta_{V_1,V_2}}\leq 1$, what is immediate from \cref{pr:zeta}\emph{\ref{it:range zeta}}.

\appendix

\section{Related angles}\label{sc:Related angles}

We review related angles found in the literature, including some that do not correspond so directly to ours, and works which do not really define an angle, but whose results are closely related to ours, even if expressed in different formalisms.

Most works deal only with real spaces. Those that consider subspaces of different dimensions usually take the angle between the smaller one and its projection, or some equivalent construction, corresponding to our min-symmetrized angle $\check{\Theta}_{V,W}$, which, as seen, has worse properties.
Results  are, mostly, particular cases of ours, with extra conditions.
By \emph{product cosine formula} we mean the statement that the cosine of the angle is the product of cosines of principal angles.

Venticos \cite{Venticos1956} defines the angle between complex subspaces of same dimension as the Hermitian angle between their blades, and obtains the product cosine formula. If $\dim V<\dim W$ the angle is defined as that between $V$ and $\Proj_W V$, if these have equal dimensions, otherwise as $\frac \pi 2$.

Afriat \cite{Afriat1957} uses, for real subspaces, symbols like $\cos(V,W)$ and $\sin(V,W)$ for products of cosines or sines of principal angles. But $(V,W)$ is not meant as an angle (it is not even defined by itself), and the symbols do not satisfy the usual trigonometric relations. He gives a result analogous to \cref{pr:exterior product}, but assuming $V\cap W=\{0\}$ and expressed in terms of volumes of parallelotopes and $\sin(V,W)$. His \emph{coefficients of inclusion and separation} correspond, respectivelly, to $\cos^2\Theta_{V,W}$ and $\cos^2\Theta_{V,W}^\perp$. 

Hawidi \cite{Hawidi1966,Hawidi1980} defines an asymmetric angle operator by $\sphericalangle(V,W) = \Proj^W_V\Proj^V_W$. It carries all information about principal angles (its eigenvalues are their squared cosines), but it is not as easy to use as a scalar angle.
It relates to our angle by $\cos^2\Theta_{V,W}=\det \sphericalangle(V,W)$.

Gluck \cite{Gluck1967} defines the angle for real subspaces of same dimension in terms of volume contraction, as in \cref{pr:Grassmann angle volume projection}\emph{\ref{it:real volume projection}}, and obtains the product cosine formula, \eqref{eq:formula Theta equal dim} and particular cases of \cref{pr:angle external powers,pr:spherical triangle inequality}. 

Górski \cite{Gorski1968} defines the angle for real subspaces $V$ and $W$ of same dimension as in \cref{pr:Theta inner blades}, and if $\dim V<\dim W$ he uses a construction with blades which, ultimately, corresponds to \cref{fig:angle VWperp}. He obtains the product cosine formula and a particular case of \cref{pr:Theta norm contraction}. 

Degen \cite{Degen1976} defines, for real subspaces of dimension $p$, a \emph{projection angle} $\psi$, using volume projection,
and an \emph{aperture angle} $\varphi$, by comparing volumes of parallelepipeds, 
in such a way that $\cos\psi$ (resp. $\sin\varphi$) is a geometric mean of  cosines (resp. sines) of principal angles.
\SELF{$\psi$ (resp. $\varphi$) is between the aritmetic mean of the $\theta_i$'s and the largest (resp. smallest) one}
But $p$-powers in the definitions make these angles harder to use, and they do not even coincide with the dihedral angle for planes in $\R^3$.

Wedin \cite{Wedin1983} defines an angle for subspaces of same dimension as in \cref{pr:Theta inner blades}, saying it was introduced by Lu \cite{Lu1963},
\SELF{\cite[p.277]{Wedin1983}} 
and states the product cosine formula. 

Miao and Ben-Israel \cite{Miao1992,Miao1996} use, for real subspaces, $\cos$- and $\sin$-symbols like Afriat's, not interpreting them in terms of angles either. They get results analogous to \cref{pr:Theta inner blades}, but in terms of determinants, matrix volumes (products of singular values) and the $\cos$-symbol, to a particular case of \cref{pr:exterior product}, in terms of matrix volumes and the $\sin$-symbol, to \cref{pr:angle external powers}, in terms of compound matrices and the $\cos$-symbol, and to propositions \ref{pr:Grassmann coordinate any dim} and \ref{pr:cor Miao} for the $\cos$-symbol.

Jiang \cite{Jiang1996} defines  a \emph{$p$-dimensional angle} for real subspaces $V$ and $W$, with $\dim V\leq\dim W$, in a way similar to \cref{pr:Theta Pnu}, and gets the product cosine formula. 
He gives, for the case of equal dimensions, propositions \ref{pr:angle external powers}, \ref{pr:Theta inner blades}, \ref{pr:spherical triangle inequality}  and \ref{pr:triangle equality same dim}, and with other dimensional conditions obtains proposition \ref{pr:spherical Pythagorean theorem} and parts of \ref{pr:zeta}. 

Risteski and Tren\v{c}evski \cite{Risteski2001} gave a formula for the angle between real subspaces of distinct dimensions in terms of arbitrary bases, but it was later \cite{Gunawan2005} shown to be valid only under certain conditions. 

Gunawan et al. \cite{Gunawan2005} use volume contraction of parallelepipeds to define the angle for real subspaces $V$ and $W$, with $\dim V\leq \dim W$.
They get a particular case of \eqref{eq:formula Theta any base any dimension}, assuming the basis of $W$ is orthonormal, and a more complicated formula without this condition.

Galántai and Heged\H{u}s \cite{Galantai2006} define a \emph{product angle} between complex subspaces in terms of the product cosine formula, but do not explore it. In \cite{Galantai2008}, Afriat's coefficients of inclusion and separation are related to Davis' separation and closeness operators \cite{Davis1958}.

Dorst et al. \cite{Dorst2007}
\SELF{p. 69--71,75} define the angle for real subspaces $V$ and $W$ of same dimension using a Clifford algebra version of \cref{pr:Theta inner blades}. They erroneously interpret it as a sort of dihedral angle and say its cosine is $0$ if $V\cap W$ has codimension 2 or greater \cite[p.\,70]{Dorst2007}. If $\dim V<\dim W$ they take the angle between $V$ and $\Proj_W V$, and obtain for this case a Clifford algebra version of \cref{pr:Theta norm contraction}.

Hitzer \cite{Hitzer2010a} uses the product cosine formula to define the angle for real subspaces of same dimension. He gives a Clifford algebra version of \cref{pr:Theta inner blades}, and a formula similar to one we give in \cite{Mandolesi_Products} for the Clifford product, but in terms of products mixing cosines and sines of principal angles.
Later he changes the angle definition to exclude any $\theta_i=\frac \pi 2$, and also applies it to subspaces of different dimensions.

In \cite{Mandolesi_Products} we use Clifford algebra to define, for real subspaces of dimension $p$, an \emph{angle bivector} $\bm{\theta}$ which codifies all data about their relative position. In particular, the norms of the components of grades $0$ and $2p$ in $\exp(\bm{\theta})$ give, respectively, $\cos\Theta_{V,W}$ and $\cos\Theta_{V,W}^\perp$.

\section*{Acknowledgment}

I am grateful to Dr. Klaus Scharnhorst for his helpful comments and for bringing to my attention many interesting references on the subject.


\providecommand{\bysame}{\leavevmode\hbox to3em{\hrulefill}\thinspace}
\providecommand{\MR}{\relax\ifhmode\unskip\space\fi MR }
\providecommand{\MRhref}[2]{%
	\href{http://www.ams.org/mathscinet-getitem?mr=#1}{#2}
}
\providecommand{\href}[2]{#2}

\end{document}